\documentclass[11pt,pdflatex]{amsart}
\usepackage[usenames]{color}
\usepackage{color}
\usepackage{amssymb}
\usepackage{graphicx, epsfig}
\usepackage{latexsym, amsfonts, amscd, amsmath}
\usepackage{mathrsfs}
\makeindex 
\input xy
\xyoption{all}

\definecolor{orange}{rgb}{1,0.5,0}
\definecolor{Indigo}{rgb}{0.2,0.1,0.7}
\definecolor{Violet}{rgb}{0.5,0.1,0.7}

\newtheorem{thm}{Theorem}[subsection]
\newtheorem{prop}[thm]{Proposition}
\newtheorem{lem}[thm]{Lemma}
\newtheorem{cor}[thm]{Corollary}

\theoremstyle{definition}
\newtheorem{dfn}[thm]{Definition}

\theoremstyle{remark}
\newtheorem{rmk}[thm]{Remark}

\newenvironment{mat}{\left(\begin{smallmatrix}
\end{smallmatrix}\right)}{enddef}

\numberwithin{equation}{subsection}
\numberwithin{figure}{subsection} \numberwithin{table}{subsection}

\newcommand{\Emb}{{\operatorname{Emb}}}
\newcommand{\End}{{\operatorname{End}}}
\newcommand{\Fr}{{\operatorname{Fr }}}
\newcommand{\Hom}{{\operatorname{Hom}}}

\newcommand{\Ima}{{\operatorname{Im}}}

\newcommand{\Ker}{{\operatorname{Ker}}}

\newcommand{\Spec}{{\operatorname{Spec }}}

\newcommand{\val}{{\operatorname{val}}}
\newcommand{\Ver}{{\operatorname{Ver }}}

\newcommand{\SL}{{\operatorname{SL }}}

\newcommand{\Gal}{{\operatorname{Gal}}}

\newcommand{\Lie}{{\operatorname{Lie}}}

\newcommand{\gera}{{\frak{a}}}
\newcommand{\gerb}{{\frak{b}}}

\newcommand{\gerd}{{\frak{d}}}

\newcommand{\germ}{{\frak{m}}}

\newcommand{\gerp}{{\frak{p}}}

\newcommand{\gert}{{\frak{t}}}

\newcommand{\gerz}{{\frak{z}}}

\newcommand{\gerX}{{\frak{X}}}
\newcommand{\gerY}{{\frak{Y}}}
\newcommand{\gerZ}{{\frak{Z}}}

\newcommand{\uA}{{\underline{A}}}
\newcommand{\uB}{{\underline{B}}}

\newcommand{\calA}{{\mathcal{A}}}
\newcommand{\calB}{{\mathcal{B}}}

\newcommand{\calF}{{\mathcal{F}}}

\newcommand{\calI}{{\mathcal{I}}}

\newcommand{\calL}{{\mathcal{L}}}

\newcommand{\calO}{{\mathcal{O}}}
\newcommand{\calP}{{\mathcal{P}}}

\newcommand{\calR}{{\mathcal{R}}}

\newcommand{\calU}{{\mathcal{U}}}
\newcommand{\calV}{{\mathcal{V}}}
\newcommand{\calW}{{\mathcal{W}}}

\def\AA{\mathbb{A}}
\def\BB{\mathbb{B}}
\def\CC{\mathbb{C}}
\def\DD{\mathbb{D}}

\def\FF{\mathbb{F}}

\def\HH{\mathbb{H}}
\def\II{\mathbb{I}}

\def\PP{\mathbb{P}}
\def\QQ{\mathbb{Q}}
\def\RR{\mathbb{R}}

\def\UU{\mathbb{U}}
\def\VV{\mathbb{V}}
\def\WW{\mathbb{W}}
\def\XX{\mathbb{X}}
\def\YY{\mathbb{Y}}
\def\ZZ{\mathbb{Z}}

\newcommand{\scrG}{{\mathscr{G}}}

\newcommand{\scrP}{{\mathscr{P}}}

\newcommand{\id}{{\noindent}}

\newcommand{\arr}{{\; \rightarrow \;}}
\newcommand{\Arr}{{\; \longrightarrow \;}}

\newcommand{\injects}{{\; \hookrightarrow \;}}

\newcommand{\ol}{{\mathcal{O}_L}}
\newcommand{\ok}{{\mathcal{O}_K}}

\newcommand{\Cl}{{\text{\rm Cl}}}

\newcommand{\Spf}{\operatorname{Spf }}

\newcommand{\Xrig}{{\gerX_\text{\rm rig}}}
\newcommand{\Xord}{{\gerX^\circ_\text{\rm rig}}}
\newcommand{\Yrig}{{\gerY_\text{\rm rig}}}
\newcommand{\Yord}{{\gerY^\circ_\text{\rm rig}}}

\newcommand{\rig}{{\operatorname{rig}}}
\newcommand{\spe}{{\operatorname{sp }}}

\newcommand{\Xbar}{\overline{X}}
\newcommand{\Ybar}{\overline{Y}}

\newcommand{\pe}{(\varphi,\eta)}
\newcommand{\peprime}{(\varphi',\eta')}
\newcommand{\Wpe}{W_{\varphi,\eta}}
\newcommand{\Zpe}{Z_{\varphi,\eta}}
\newcommand{\Up}{U_\phi}
\newcommand{\Upplus}{U_\phi^+}
\newcommand{\Ve}{V_\eta}
\newcommand{\Veplus}{V_\eta^+}

\newcommand{\Qbar}{\overline{Q}}
\newcommand{\Pbar}{\overline{P}}
\newcommand{\Rbar}{\overline{R}}
\newcommand{\Qpbar}{\overline{\QQ}_p}
\newcommand{\lcm}{\text{\rm lcm}}
\newcommand{\eps}{\epsilon}

\newcommand{\OhatY}{\widehat{\calO}_{\Ybar, \Qbar}}

\newcommand{\OhatP}{\widehat{\calO}_{X, \Pbar}}
\newcommand{\OhatQ}{\widehat{\calO}_{Y, \Qbar}}
\newcommand{\OhatPbar}{\widehat{\calO}_{\Xbar, \Pbar}}
\newcommand{\OhatQbar}{\widehat{\calO}_{\Ybar, \Qbar}}
\newcommand{\Yoord}{{\gerY^{\circ\circ}_\text{\rm rig}}}

\newcommand{\om}{{\mathcal{O}_M}}
\newcommand{\Spm}{{\rm Spm}}

\newcommand{\cX}{\XX}
\newcommand{\cXbar}{\overline{\XX}}
\newcommand{\cY}{\YY}
\newcommand{\cYbar}{\overline{\YY}}
\newcommand{\cWpe}{\WW_{\varphi,\eta}}
\newcommand{\cZpe}{\ZZ_{\varphi,\eta}}

\begin{document}
\marginparwidth 50pt

\openup 1pt

\title[Canonical Subgroup]{Canonical Subgroups over Hilbert Modular Varieties}
\author{Eyal Z. Goren \& Payman L Kassaei}
\address{Department of Mathematics and Statistics,
McGill University, 805 Sherbrooke St. W., Montreal H3A 2K6, QC,
Canada.}\address{Department of Mathematics, King's College London,
Strand, London WC2R 2LS, United Kingdom.}
\email{eyal.goren@mcgill.ca; kassaei@alum.mit.edu}

\begin{abstract}
We obtain new results  on the geometry of Hilbert modular varieties
in positive characteristic and morphisms between them. Using these
results and methods of rigid geometry, we develop a theory of
canonical subgroups for abelian varieties with real multiplication.
\end{abstract}

\maketitle

\section{Introduction}

The theory of canonical subgroups was developed by Katz
\cite{Katz}, building on work of Lubin  on canonical subgroups of
formal groups of dimension one.  Katz's motivation was to
show that the $U_p$ operator on the space of $p$-adic elliptic modular forms preserves the subspace of overconvergent
modular forms.

The kernel of multiplication by $p$ in the formal group of an
elliptic curve of ordinary reduction over a $p$-adic base has a
distinguished subgroup of order $p$, which reduces to the Kernel of
Frobenius modulo $p$ and is called the canonical subgroup. The $U_p$ operator can be defined by a moduli-theoretic formula
involving the canonical subgroup. Extending this Hecke operator to overconvergent modular forms directly involves extending the notion of
canonical subgroups from elliptic curves of ordinary reduction to those of  a ``slight" supersingular reduction (quantified by an appropriate measure of supersingularity involving the Hasse invariant).  However, Katz-Lubin proved much more: they provided
optimal bounds on the measure of supersingularity for the existence
of canonical subgroups, and proved that the canonical subgroup again reduces to the kernel
of Frobenius, albeit only modulo a divisor of $p$ determined by a measure of supersingularity of $E$.

The power of canonical subgroups and their properties became apparent, for example, in the
work of Buzzard-Taylor on the Artin conjecture \cite{BuzzardTaylor,Buzzard}, where modularity of certain Galois represenations was proved by analytic continuation
of overconvergent modular forms,
and in the work of Kassaei on classicality of overconvergent modular
forms in \cite{K1, K2}, where analytic continuation of overconvergent modular forms was used to provide a method for proving Coleman's classicality theorem \cite{Coleman}
for more general Shimura varieties; in both examples, canonical subgroups and their
properties were used in the process of analytic continuation.

The story of $p$-adic modular forms began when Serre introduced them
in \cite{Serre} as $p$-adic limits of $q$-expansions of classical
modular forms. The theory was motivated by two issues. The $p$-adic
interpolation of constant terms of certain Eisenstein series, in
particular special values of abelian $L$-functions, and the
construction of $p$-adic analytic families of modular forms with
connections to Galois representations and Iwasawa theory in mind.
Serre's point of view can be generalized to the Hilbert modular
case. This has been done by Andreatta-Goren
\cite{AndreattaGorenHMF}, but already previously some aspects of the
theory were generalized by Deligne-Ribet and the applications to
constant terms of Eisenstein series were harvested
\cite{DeligneRibet}. Dwork studied $p$-adic modular functions  with
``growth condition" and showed that the $U_p$ operator is completely
continuous on the space of these functions. Almost at the same time
of Serre's work, Katz \cite{Katz} interpreted Serre's $p$-adic
modular forms as sections of suitable line bundles over the ordinary
locus of the corresponding modular curve. Katz incorporated Dwork's
notion of growth condition into this geometric construction by
considering sections of the same line bundles over larger regions of
the $p$-adic analytic space of the modular curve, thus giving birth
to the notion of overconvergent modular forms.

As we mentioned above, the $p$-torsion in the formal group of an
elliptic curve of ordinary reduction over a $p$-adic base provides a
lifting of the Kernel of Frobenius modulo $p$. This is not hard to
see using some fundamental facts about \'etale formal groups over a
$p$-adic base.  When the elliptic curve has supersingular reduction,
the $p$-torsion of the formal group grows to a subgroup of rank
$p^2$, no longer providing a canonical subgroup. Using Newton
polygons, Katz-Lubin showed that for an elliptic curve of 
``not-too-supersingular" reduction there are $p-1$ non-zero elements
in the $p$-torsion of the formal group which are of (equal and)
closer distance to the origin than the remaining non-zero elements; along with zero, they from a distinguished subgroup of rank $p$ called the
canonical subgroup of the elliptic curve.  Katz showed that this
construction works in families and consequently proved that the
Hecke operator $U_p$ preserves the space of overconvergent modular
forms. This completely continuous operator has been essential to the
development of the theory of overconvergent modular forms
(especially through the pioneering works of Hida and Coleman).

After Katz and Lubin the first major advancement in the study of the
canonical subgroup was made by Abbes and Mokrane in
\cite{AbbesMokrane}. Since then, many authors have studied the
canonical subgroup in various settings and with different
approaches. We mention the works \cite{Conrad}, \cite{Fargues}, \cite{KisinLai}, \cite{Tian},  as well as yet unpublished
results by K. Buzzard, E. Nevens and J. Rabinoff. Broadly speaking,
the traditional approach to the canonical subgroup problem proceeds
through a careful examination of subgroup schemes of either abelian
varieties, or $p$-divisible groups, and, again broadly speaking,
much of the complications arise from the fact that formal groups in
several variables are hard to describe and one lacks  a theory of
Newton polygons for power-series in several variables. These
deficiencies, as well as Raynaud's interpretation of rigid spaces in
the language of formal schemes, prompted us to investigate another
approach to canonical subgroups, where geometry plays a more
prominent role. In \cite{GK} we tested our ideas in the case of
curves and showed that one can develop the theory of canonical
subgroup for ``all" Shimura curves at once, getting results that
are, in a sense, more precise than previously known. As special cases, we recovered results of Katz-Lubin (modular curves) and \cite{K3} (unitary Shimura curves). In our setting,
one considers a morphism $\pi\colon Y \arr X$, where $X, Y$, are any
``nice" curves over a dvr $R$, finite over $\ZZ_p$, such that the
reductions of $\pi, X$, and $Y$, modulo the maximal ideal satisfy
certain simple geometric properties known to be present in the case
of Shimura curves. Canonical subgroups and their finer properties are then studied by constructing a section to the rigid analytic fibre of $\pi$ and studying its properties.  This theory is used in \cite{K2} to prove
classicality results for overconvergent modular forms simultaneously
for ``all" Shimura curves.

\medskip

\id The general canonical subgroup problem can be formulated as
follows. We restrict to Shimura varieties of PEL type, although one
should be able to extend it to Shimura varieties of Hodge type, for
example. The typical context is that one is given a moduli problem
of PEL type and a fine moduli space $X$ representing it over
$\Spec(\calO_K[S^{-1}])$ -- a localization of the ring of integers
of some number field $K$. If $p$ is a rational prime, such that the
reductive group corresponding to the moduli problem is unramified at
$p$ and the level structure is prime to $p$, $X$ usually has a
smooth model over $R$, the completion of $\calO_K[S^{-1}]$ at a
prime $\gerp$ above $p$. One then adds a suitable level structure at
$p$ to the moduli problem represented by $X$ to obtain a fine moduli
scheme $Y/\Spec(R)$ and a natural forgetful morphism $\pi:Y \arr X$.
Under minimal conditions, the reduction $\overline{X}$ of $X$ modulo
$\germ_R$ has a generalized ordinary locus $\overline{X}^{\rm ord}$,
which is open and dense in $\overline{X}$, and a section $s:
\overline{X}^{\rm ord} \arr \overline{Y}^{\rm ord}$ to $\pi$. One
certainly expects to be able to lift $s: \overline{X}^{\rm ord} \arr
\overline{Y}^{\rm ord}$ to a section $s^\circ: \gerX^\circ \arr
\gerY^\circ$, where $\gerX^\circ$ is the (admissible open) set in
the rigid analytic space associated to $X$ consisting of the points
that specialize to $\overline{X}$; similarly for $\gerY^\circ$. The
canonical subgroup problem is to find an explicitly described
admissible open set $\calU\supseteq \gerX^\circ$ to which one can
extend $s^\circ$ (if one chooses $\calU$ appropriately the extension
is unique) and characterize its image.

The completed local rings $\widehat{\calO}_{X, \Pbar},
\widehat{\calO}_{Y, \Qbar}$ at closed points $\Pbar, \Qbar$ of the
special fibers $\Xbar, \Ybar$ afford an interpretation in terms of
pro-representing a moduli problem on a category of certain local
artinian rings.  Using the theory of local models, one expects to be
able to to write down models for these rings that are  the completed
local rings of suitable points on certain generalized Grassmann
varieties.

 Assume
that $\pi(\Qbar) = \Pbar$. According to our approach, it is the map
$\pi^\ast\colon \widehat{\calO}_{X, \Pbar}\arr \widehat{\calO}_{Y,
\Qbar}$ that ``holds all the secrets" concerning the canonical
subgroup. In this paper we show that this is the case for Hilbert
modular varieties. The information we can find on $\pi^\ast$ uses
heavily the moduli description, but once obtained, the specific
nature of $X, Y,$ as moduli schemes plays no role anymore.

We find it remarkable that not only does this suffice for the
construction of the canonical subgroup, in fact the results we
obtain improve significantly on what is available in the literature
as a consequence of work by others. For example, we are able to
prove the existence of canonical subgroups on domains described by
valuations of as many parameters as the dimension of $X$, whereas in
the literature these constructions are almost always carried out on
coarser domains defined by the valuation of one variable (the Hodge
degree). Also our results improve significantly on the bounds for
these variables; in fact these bounds can be shown to be optimal in
a sense explained in Corollary \ref{Cor: optimality}.

We next describe our results in more detail.

\

\id Let $L$ be a totally real field of degree $g$ over $\QQ$ and $p$
a rational prime, unramified in $L$. We consider the moduli space
$X$ parameterizing polarized abelian varieties $\uA$ with RM by $L$
and a rigid level structure prime to $p$, and the moduli space $Y$
parameterizing the same data and, in addition, a maximal ``cyclic"
$\ol$-subgroup of $\uA[p]$. There is a forgetful morphism $\pi\colon
Y \arr X$. Both $X$ and $Y$ are considered over the Witt vector
$W(\kappa)$ of a suitable finite field $\kappa$. We let $\XX$ and
$\YY$ denote the minimal compactifications. See below and section \S
2.1.

The special fibre $\Xbar$ of $X$ was studied in \cite{GorenOort},
where a stratification $\{Z_\tau\}$ indexed by subsets $\tau$ of
$\BB = \Hom(L, \overline{\QQ}_p)$ was constructed. This
stratification is intimately connected to theory of Hilbert modular
forms. In particular, for every $\beta \in \BB$ there is a Hilbert
modular form $h_\beta$ -- a partial Hasse invariant -- whose divisor
is $Z_\beta$; see \S \ref{subsection: stratification of Y bar}. The
partial Hasse invariants are a purely characteristic $p$ phenomenon,
but they can be lifted locally in the Zariski topology to $X$.

Let $\Xrig, \Yrig$ denote the rigid analytic spaces associated \`a
la Raynaud to $\XX, \YY$, and $\gerX_{\rm rig}^\circ, \gerY_{\rm
rig}^\circ$ their ordinary locus. There is a kernel-of-Frobenius
section $\Xbar \arr \Ybar$ given on points by $\uA \mapsto
\left(\uA, \Ker(\Fr\colon \uA \arr \uA^{(p)})\right),$ which extends
to compactifications. We show, using a Hensel's lemma type of
argument, that this section lifts to a canonical morphism
$s^\circ\colon \gerX_{\rm rig}^\circ\arr\gerY_{\rm rig}^\circ$,
which is a section to $\pi$.

For a point $P$ in $\Xrig$, and $\beta\in\BB$, let
$\tilde{h}_\beta(P)$ denote the evaluation of a Zariski local lift
of the partial Hssse invariant $h_\beta$ at $P$. Let $\sigma$ be the
Frobenius automorphism on $\QQ_p^{\rm ur}$. One of our main theorems
(Theorem \ref{theorem: canonical subgroup}) states the following.

\medskip

\id {\bf Theorem {\bf A}.} \emph{  Let $\calU \subseteq  \Xrig$ be
defined as
\[
\calU = \{P : \nu(\tilde{h}_\beta(P)) + p\nu(\tilde{h}_{\sigma\circ
\beta}(P)) < p, \;\forall \beta \in \BB\}.
\]
There exists a section, \[s^\dagger\colon \calU \arr \Yrig,\]
extending the section $s^\circ$ on the ordinary locus.}

\medskip
\id In the theorem, $\nu$ is the $p$-adic valuation, normalized by
$\nu(p) = 1$ and truncated at $1$. In fact, comparing to Theorem
\ref{theorem: canonical subgroup}, the reader will notice that our
formulation is not the same. In \S \ref{subsection: valuations on X
and Y} we define vector valued valuations on $\Xrig, \Yrig$. The
valuation vector $\nu(P)$ of a closed point $P$ on $\Xrig$ takes
into account the strata $Z_\tau$ on which $\Pbar$, the reduction of
$P$, lies. Also the valuation vector $\nu(Q)$ of a closed point $Q$
on $\Xrig$ takes into account on which strata $\Qbar$ lies. A
substantial part of the paper (\S\S2-3)  is in fact devoted to
defining this stratification of $\Ybar$ and studying it properties.
Let $\pe$ be a pair of subsets of $\BB$ which is admissible (\S
\ref{subsec: discrete invariants}; there are $3^g$ such strata). We
define a stratification $\{Z_{\varphi, \eta}\}$ indexed by such
pairs $\pe$;  the fundamental results concerning this stratification
appear in Theorem~\ref{theorem: fund'l facts about the
stratificaiton of Ybar}. Some key facts are: \begin{enumerate} \item
$\pi(Z_{\varphi, \eta}) = Z_{\varphi\cap \eta}$
(Theorem~\ref{theorem: pi of Zpe});\item every irreducible component
of $Z_{\varphi, \eta}$ intersects non-trivially the finite set of
points corresponding to data $\left(\uA, \Ker(\Fr\colon \uA \arr
\uA^{(p)})\right)$ consisting of superspecial abelian varieties with
the kernel of Frobenius group scheme (Theorem \ref{theorem: C cap YF
cap YV});
\item Every irreducible component of $\Spf(\widehat{\calO}_{\Ybar,
\Qbar})$ is accounted for by a unique maximal stratum $Z_{\varphi,
\eta}$ (Theorem~\ref{theorem: fund'l facts about the stratificaiton
of Ybar}).
\end{enumerate}

Let $\Qbar \in \Ybar (k)$ and $\Pbar = \pi(\Qbar)$ its image in
$\Xbar(k)$. One has natural parameters such that
$\widehat{\calO}_{X, \Pbar} \cong W(k)[\![t_\beta: \beta \in \BB
]\!]$ (Equation (\ref{equation: local def ring of X at P bar})), and
parameters such that $
 \widehat{\calO}_{Y, \Qbar} \cong W(k) [\![ \{x_\beta,
 y_\beta: \beta \in I\}, \{z_\beta: \beta \in I^c\}]\!]/(\{x_\beta y_\beta - p: \beta\in
 I\})$ (Equation (\ref{equation: local deformation ring at Q bar --
arithmetic version})), where if $\Qbar$ has invariants $\pe$, then
$I = \ell(\varphi) \cap \eta$ and $\ell(\varphi) =
\{\sigma^{-1}\circ \beta : \beta \in \varphi\}$ . One of our main
results, referred to as ``Key Lemma" (Lemma \ref{lemma: key lemma })
describes $\pi^\ast (t_\beta)$, under the induced ring homomorphism
$\pi^\ast\colon \widehat{\calO}_{\Xbar, \Pbar} \arr
\widehat{\calO}_{\Ybar, \Qbar}$, for $\beta\in \varphi\cap\eta$. (It
remains an interesting problem to extend this lemma to $\beta
\not\in \varphi\cap\eta$.) The Key lemma is crucially used in the
proof of Theorem {\bf A}. It allows us to compare the valuation
vectors of a point in $\Yrig$ and its image under $\pi$ in $\Xrig$
as in \S \ref{subsec: main theorem}. Using such valuation vector
calculations,  we are able to  isolate a union of connected
components of $\pi^{-1}(\calU)$ (which we call $\calV$), and prove
that $\pi$ is an isomorphism when restricted to $\calV$. The section
$s^\dagger$ is defined as the inverse of $\pi\vert_{\calV}$.

In \S \ref{subsection: Properties of the canonical subgroup} we
describe the properties of the canonical subgroup. In particular, we
study the property of the reduction to the kernel of Frobenius
(Theorem~\ref{theorem: reduction of canonical subgroup}) and what
happens under the iteration of our construction
(Theorem~\ref{theorem: hecke}). The following theorem summarizes
some aspects of these results.

\medskip

\id {\bf Theorem  B.} \emph{Let $\uA$ be an abelian variety with RM
corresponding to $P \in \calU$. Let $r$ be an element whose
valuation is $\max\{\nu(\tilde{h}_\beta(P))\}_{\beta\in\BB}$.
\begin{enumerate}
\item The canonical subgroup $H$ of $\uA$ reduces to $\Ker(\Fr)$ modulo $p/r$.
\item Let $C$ be a subgroup of $\uA$ such that $(\uA,C)\in\Yrig$. Let $P^\prime\in\Xrig$ correspond to $\uA/C$.
There is a recipe for calculating
$\{\nu(\tilde{h}_\beta(P^\prime))\}_\beta$ in terms of
$\{\nu(\tilde{h}_\beta(P))\}_\beta$, in particular, for determining
if it affords a canonical subgroup.
\end{enumerate}}

\medskip

\id Theorem {\bf B} determines the $p$-adic geometry of the Hecke
operator on the ``not-too-singular'' locus of  $\Yrig$.  Theorem
{\bf B} also applies directly to deriving a theorem about
higher-level canonical subgroups (Proposition~\ref{proposition:
iterations}).  The results are similar to the case of elliptic
curves, only that the situation is richer as the position of a point
on $\Yrig$ is described by $g$ parameters (the components of its
valuation vector) in contrast to the case of elliptic curves where
there is only one parameter.

Finally, the Appendix describes a certain generalization of the
moduli scheme $Y$, obtained by considering for an $\ol$-ideal
$\gert\vert p$, a cyclic $\ol/\gert$ subgroup instead of a cyclic
$\ol/(p)$ subgroup, and briefly describes the extension of our
results to this situation. The results are relevant to the
construction of partial $U$ operators, indexed again by ideals
$\gert\vert p$.

\medskip

\id{\bf Acknowledgements.} We would like to thank F. Andreatta for
discussions concerning the contents of this paper.

{\small\tableofcontents}

 \noindent {\bf\textsc{Notation}}. Let $p$ be a prime number,
$L/\mathbb{Q}$ a totally real field of degree $g$ in which $p$ is
unramified, $\mathcal{O}_L$ its ring of integers, $\gerd_L$ the different ideal, and $N$ an integer prime to $p$. Let $L^+$ denote the elements of $L$ that are positive under every embedding $L\injects \RR$. For a  prime ideal $\gerp$  of
 $\ol$ dividing $p$, let $\kappa_\gerp = \ol/\gerp$, $f_\gerp = \deg(\kappa_\gerp/\FF_p)$, $f =
 \lcm\{f_\gerp: \gerp\vert p\}$, and $\kappa$ a finite field with $p^f$
 elements. We identify $\kappa_\gerp$ with a subfield of $\kappa$ once
 and for all. Let $\mathbb{Q}_\kappa$ be the fraction field of $W(\kappa)$.
 We fix embeddings $\mathbb{Q}_\kappa \subset \QQ_p^{\rm ur}
 \subset \Qpbar$.

 Let $[\Cl^+(L)]$ be a complete set of representatives
 for the strict (narrow) class group $\Cl^+(L)$ of~$L$, chosen so
 that its elements are ideals $\gera\vartriangleleft\ol$, equipped with their natural positive cone $\gera^+ = \gera \cap L^+$. Let
\[\mathbb{B}={\rm Emb}(L,\mathbb{Q}_\kappa)=\textstyle\coprod_{\mathfrak{p}}
\mathbb{B}_\mathfrak{p},\] where $\mathfrak{p}$ runs over prime
ideals of $\mathcal{O}_L$ dividing $p$, and
$\mathbb{B}_\mathfrak{p}= \{\beta\in\mathbb{B}\colon
\beta^{-1}(pW(\kappa)) = \mathfrak{p}\}$. Let $\sigma$ denote the
Frobenius automorphism of $\mathbb{Q}_\kappa$, lifting $x \mapsto
x^p$ modulo $p$. It acts on $\mathbb{B}$ via $\beta \mapsto \sigma
\circ \beta$, and transitively on each $\mathbb{B}_\mathfrak{p}$.
For $S \subseteq \mathbb{B}$ we let
\[\ell(S)=\{\sigma^{-1}\circ\beta\colon \beta \in S\},\qquad
r(S)=\{\sigma\circ\beta\colon \beta \in S\},\] and
\[S^c=\mathbb{B} - S.\] The decomposition \[\mathcal{O}_L
\otimes_\mathbb{Z} W(\kappa)=\bigoplus_{\beta \in \mathbb{B}}
W(\kappa)_\beta,\] where $W(\kappa)_\beta$ is $W(\kappa)$ with the
$\mathcal{O}_L$-action given by $\beta$, induces a decomposition,
\[M=\bigoplus_{\beta\in \mathbb{B}} M_\beta,\] on any $\mathcal{O}_L
\otimes_\mathbb{Z} W(\kappa)$-module $M$.

\

\id Let $A$ be an abelian scheme over a scheme $S$, equipped with
real multiplication $\iota\colon \ol \rightarrow \End_S(A)$. Then
the dual abelian scheme $A^\vee$ has a canonical real
multiplication, and we let $\calP_A = \Hom_\ol(A, A^\vee)^{\rm
sym}$. It is a projective $\ol$-module of rank 1 with
a notion of positivity; the positive elements correspond to
$\ol$-equivariant polarizations.

For a $W(\kappa)$-scheme $S$ we shall denote by $\underline{A}/S$,
or simply $\underline{A}$ if the context is clear, a quadruple:
\[\underline{A}/S=(A/S,\iota,\lambda,\alpha),\]
comprising the following data: $A$ is an abelian scheme of relative
dimension $g$ over a $W(\kappa)$-scheme $S$,
$\iota\colon\mathcal{O}_L \hookrightarrow {\rm End}_S(A)$ is a ring
homomorphism. The map $\lambda$ is a polarization as in \cite{DP},
namely, an isomorphism $\lambda\colon (\calP_A, \calP_A^+)
\rightarrow (\gera, \gera^+)$ for a representative $(\gera,
\gera^+)\in [\Cl^+(L)]$ such that $A \otimes_\ol \gera \cong
A^\vee$. The existence of $\lambda$ is equivalent, since $p$ is
unramified, to $\Lie(A)$ being a locally free $\ol\otimes
\calO_S$-module \cite{GorenBook}. Finally, $\alpha$ is a rigid
$\Gamma_{00}(N)$-level structure,  that is, $\alpha\colon \mu_N
\otimes_{\ZZ} \gerd_L^{-1} \arr A$ is an $\ol$-equivariant
closed immersion of group schemes.

 Let $X/W(\kappa)$ be the Hilbert modular scheme classifying
such data $\underline{A}/S=(A/S,\iota,\lambda,\alpha)$. Let
$Y/W(\kappa)$ be the Hilbert modular scheme classifying
$(\underline{A}/S, H)$, where $\underline{A}$ is as above and  $H$
is a finite flat isotropic $\mathcal{O}_L$-subgroup scheme of $A[p]$
of rank $p^g$, where isotropic means relative to the $\mu$-Weil
pairing for some $\mu\in \calP_A^+$ of degree prime to $p$. Let
\[\pi\colon Y \rightarrow X\] be the natural morphism, whose effect on points
is $(\underline{A},H) \mapsto \underline{A}$.

Let $\overline{X}, \mathfrak{X}, \mathfrak{X}_{\rm rig}$  be,
respectively,  the special fibre of $X$, the completion of $X$ along
$\overline{X}$, and the rigid analytic space associated to
$\mathfrak{X}$ in the sense of Raynaud. We use similar notation
$\overline{Y}, \mathfrak{Y}, \mathfrak{Y}_{\rm rig}$ for $Y$ and let
$\pi$ denote any of the induced morphisms. These spaces have models
over $\mathbb{Z}_p$ or $\mathbb{Q}_p$, denoted $X_{\mathbb{Z}_p},
\mathfrak{X}_{\rm rig, \mathbb{Q}_p}$, etc.  For a point $P \in
\mathfrak{X}_{\rm rig}$ we denote by $\overline{P}={\rm sp}(P)$ its
specialization in $\overline{X}$, and similarly for $Y$. Let
$w\colon Y \rightarrow Y$ be the automorphism $(\underline{A},H)
\mapsto (\underline{A}/H,A[p]/H)$. Let $s\colon\overline{X}
\rightarrow \overline{Y}$ be the kernel-of-Frobenius section to
$\pi$, $\underline{A} \mapsto (\underline{A},{\rm Ker}({\rm
Fr}_A\colon A \rightarrow A^{(p)}))$, which exists by
Lemma~\ref{lemma: cyclic is isotropic}. We denote $s(\overline{X})$
by $\overline{Y}_F$, and $w(\overline{Y}_F)$ by $\overline{Y}_V$.
These are components of $\overline{Y}$, and the geometric points of
$\overline{Y}_F$ (respectively, $\overline{Y}_V$) are the geometric
points $(\underline{A},H)$ where $H$ is ${\rm Ker}({\rm Fr}_A)$
(respectively, ${\rm Ker}({\rm Ver}_A)$). We denote the ordinary
locus in $\Xbar$ (respectively, $\Ybar$) by $\Xbar^{\rm ord}$
(respectively, $\Ybar^{\rm ord}$). We define $\Ybar^{\rm ord}_F$ to
be $s(\Xbar^{\rm ord})$, and $\Ybar^{\rm ord}_V$ to be $w(\Ybar^{\rm
ord}_F)$;   they are both a union of connected components of
$\Ybar^{\rm ord}$.

\

\

\section{Moduli spaces in positive characteristic}
\subsection{Two formulations of a moduli problem}
Recall that $X/W(\kappa)$ is the moduli space parameterizing data
$\uA/S = (A, \iota, \lambda, \alpha)$ where $S$ is a
$W(\kappa)$-scheme, $A$ an abelian scheme over $S$ of relative
dimension $g$, $\iota\colon \ol \arr \End_S(A)$ a ring homomorphism
such that $\Lie(A)$ is a locally free $\ol\otimes \calO_S$ module of
rank $1$ (the ``Rapoport condition"). The map $\alpha\colon
\mu_N\otimes_{\ZZ}\gerd_L^{-1}\arr A$ is a rigid
$\Gamma_{00}(N)$-level structure and $\lambda$ is a polarization
data: an isomorphism of $\ol$-modules with a notion of positivity,
$\lambda\colon \calP_\uA \rightarrow \gera$, for one of the
representatives $\gera\in [\Cl^+(L)]$ fixed above. It
follows, since $p$ is unramified in $L$, that the natural morphism
$A\otimes_\ol \calP_A \arr A^\vee$ is an isomorphism (this fact is
sometimes called the ``Deligne-Pappas" condition; they introduced it
in \cite{DP} in the case $p$ is possibly ramified in $L$, in
replacement of the Rapoport condition). The morphism
$X\arr\Spec(W(\kappa))$  is smooth, quasi-projective, of relative
dimension $g$. We let $\Xbar = X \otimes \kappa$ denote the special
fibre of $X$. It is a quasi-projective non-singular variety of
dimension $g$ over $\kappa$, whose irreducible components are in
bijection with $\Cl^+(L)$.

Recall also the moduli space $Y$ that parameterizes data $(\uA,
H)/S$, where $\uA/S$ is as above and $H$ is a finite flat
$\ol$-subgroup scheme of $A[p]$ of rank $p^g$, isotropic relative to
the $\gamma$-Weil pairing induced by  a  $\gamma\in\calP_\uA$ of degree prime to $p$,
\[\xymatrix{A[p] \times A[p] \ar[r]^{1 \times \gamma}_{\cong} & A[p]\times A^\vee[p] \ar[r] & \mu_p.}  \]
It follows that $H$ is isotropic relative to the $r\gamma$-Weil pairing, where $r\in\ol$ is prime to $p$.
Therefore, $H$ is isotropic relative to  the $\delta$-Weil pairing for all $\delta\in\calP_\uA$ of degree prime to $p$.
Since $\calP_\uA$ is generated as a $\ZZ$-module by  such $\delta$, we conclude that $H$ is isotropic relative to the Weil pairing induced by  some $\gamma\in\calP_\uA$ of degree prime to $p$ implies that it is isotropic relative to any Weil pairing induced by an element of $\calP_\uA$.
Henceforth, we will simply call such $H$ isotropic.

\begin{lem} \label{lemma: cyclic is isotropic} Let $\uA/S$ be an object of the kind
parameterized by $X$, where $S$ is a reduced $W(\kappa)$-scheme. Let
$H\subseteq A[p]$ be a finite flat $\ol$-group scheme of $\uA$,
which is a \emph{cyclic $\ol$-module}, where by that we mean that
for every geometric point $x$ of characteristic zero of $S$ the
group scheme $H_x$ is a cyclic $\ol$-module, and for every geometric
point $x$ of characteristic $p$ the Dieudonn\'e module $\DD(H_x)$ is
a cyclic $\ol\otimes_\ZZ k(x)$-module. Then $H$ is isotropic
relative to any $\ol$-polarization.

In particular, if $S$ is a characteristic $p$ scheme, and $H =
\Ker(\Fr_A)$, then $H$ is automatically isotropic.
\end{lem}
\begin{proof}
Let $\mu$ be an $\ol$-polarization. The locus where $H$ is
isotropic relative to $\mu$ can be viewed as the locus where $H
\subseteq \mu(H)^\perp$ under the Weil pairing on $A[p] \times
A^\vee[p]$, and so is a closed subset of $S$; it is enough to prove
it contains every geometric point $x$ of $S$. If $x$ has
characteristic zero, $H_x$ is a cyclic $\ol$-module and
$\mu$-induces an alternating pairing $\langle \cdot,
\cdot\rangle\colon H_x \times H_x \arr \mu_p$ such that $\langle
\ell r, s\rangle = \langle r, \ell s\rangle$ for $\ell \in \ol, r, s
\in H_x$. Let $g$ be a generator of $H_x$; then any other element of
$H_x$ is of the form $\ell g$ for some $\ell \in \ol$. Then $\langle
\ell_1 g, \ell_2 g\rangle = \langle \ell_1 \ell_2 g, g\rangle =
\langle g, \ell_1\ell_2 g\rangle = - \langle \ell_1\ell_2 g,
g\rangle$, and so $H_x$ is isotropic. If $x$ is of characteristic $p$, the argument is the same,
making use of the cyclicity of the $\ol\otimes k$ module
$\DD(A[p])$.

It remains to show that $\DD(\Ker(\Fr_{A_x}))$ is always a cyclic
$\ol\otimes k$-module. This is not automatic, and in fact
it uses the Rapoport condition by which the Lie algebra of $A$
(identified with the tangent space at the identity) is a free $\ol\otimes k$-module of rank one. On the other hand, by a result of Oda
(see \S \ref{subsection:Facts about D modules}), the tangent space
is, up to a twist, the Dieudonn\'e module of the kernel of
Frobenius.
\end{proof}

 We let $\Ybar = Y \otimes
\kappa$ denote the special fibre of $Y$. It is a quasi-projective
variety of dimension $g$ over $\kappa$, which, as we shall see below
is highly singular and reducible (even for a fixed polarization
module), although equi-dimensional. The morphism $\pi\colon \Ybar
\arr \Xbar$ is proper. The space $\Xbar$ was studied by Rapoport
\cite{Rapoport} and Goren-Oort \cite{GorenOort} and the space
$\Ybar$ was studied by Pappas \cite{Pappas}, H. Stamm \cite{Stamm},
and more recently by C.-F. Yu in \cite{CFY}, although we shall make
no use of Yu's work here.

Our main interest in this section is in stratifications of $\Xbar$
and $\Ybar$ and how they relate via the morphism $\pi\colon \Ybar
\arr \Xbar$, but first we provide another interpretation of $Y$.
\begin{lem}\label{lemma: alternative formulation}
The moduli space $Y$ is also the moduli space of data $(f\colon\uA
\arr \uB)$, where $\uA = (A, \iota_A, \lambda_A, \alpha_A)$, $\uB =
(B, \iota_B, \lambda_B, \alpha_B)$ are polarized abelian varieties
with real multiplication and $\Gamma_{00}(N)$-structure, and $f$ is
an $\ol$-isogeny, killed by $p$ and of degree $p^g$, such that $f^\ast \calP_B = p\calP_A$. (In
particular, $\uA$ and $\uB$ have isomorphic polarization modules).
\end{lem}
\begin{proof}
 Let $(A,H)$ be as above. We define $\uB=(B, \iota_B, \lambda_B, \alpha_B)$ to be $A/H$ with the naturally induced real multiplication by $\calO_L$ and $\Gamma_{00}(N)$-level structure; $\lambda_B$ will be defined below.
 Let $f\colon A\arr A/H$ denote the natural isogeny, and let $f^t:A/H \arr A$ be the unique isogeny such that $f^t\circ f=[p]_A$; its kernel is $A[p]/H$. The short exact sequence $0 \arr A[p]/H \arr A/H \overset{f^t}{\arr} A \arr 0$ induces a short exact sequence
 \[
 \xymatrix{0\ar[r]&(A[p]/H)^\vee \ar[r] & A^\vee \ar[r]^{(f^t)^\vee}&(A/H)^\vee \ar[r]& 0.}
 \]

Since the annihilator of $H$ under the Weil pairing $A[p]\times A^\vee[p]\arr\mu_p$ is $(A[p]/H)^\vee$, it follows that  $H$ is isotropic  if and only if for any $\gamma\in\calP_A$, we have $\gamma(H)\subseteq (A[p]/H)^\vee$. For such $H$ and $\gamma\in \calP_A$, we have a commutative diagram:
\begin{equation}\label{equation: i-gamma}\xymatrix{
0 \ar[r]& H \ar[r]\ar[d] & A \ar[r]^f\ar[d]_\gamma & A/H \ar[r]\ar[d]^{i\gamma} & 0 \\
0 \ar[r]& (A[p]/H)^\vee \ar[r] & A^\vee \ar[r]^{(f^t)^\vee} &
(A/H)^\vee \ar[r] & 0, }\end{equation} where we have denoted by
$i\gamma$ the map $A/H \arr (A/H)^\vee$ appearing in the diagram;
$\gamma \mapsto i\gamma$ is an $\ol$-linear homomorphism $i:\calP_A
\arr \calP_B$. It follows from this definition that $i\gamma\circ f
=(f^t)^\vee \circ \gamma$. In particular,
\[
\deg(i\gamma)=\deg(\gamma).
\]
As a result, $i$ is injective.   For every $\gamma\in \calP_A$
we have
\[
f^\ast(i\gamma)=f^\vee\circ i\gamma \circ f=f^\vee \circ (f^t)^\vee \circ \gamma= (f^t\circ f)^\vee \circ \gamma=p\gamma.
\]
Therefore, the  composition $f^\ast\circ i:\calP_A \arr \calP_A$ is
multiplication by $p$. In particular, we have
$f^\ast(\calP_B)\supseteq p\calP_A$. We now show that
$f^\ast(\calP_B)\subseteq p\calP_A$. To this end, consider the map
$f^t:A/H \arr A$ of kernel $A[p]/H$. Let $\gamma\in \calP_A$ be a
polarization of degree prime to $p$. Hence, the polarization
$i\gamma$ is also of degree prime to $p$. To show that $A[p]/H$ is
isotropic it is enough to show it is isotropic relative to
$i\gamma$, that is, $i\gamma(A[p]/H) = H^\vee$ ($H^\vee$ is
naturally identified with the annihilator of $A[p]/H$ in
$(A/H)^\vee$). And indeed, $i\gamma(A[p]/H) = i\gamma\circ f (A[p])
= (f^t)^\vee \gamma(A[p]) = (f^t)^\vee (A^\vee[p]) =
A^\vee[p]/\left(A[p]/H\right)^\vee = H^\vee$. We may now apply the
same arguments made above and conclude that there is an $\ol$-linear
map $j\colon \calP_B \arr \calP_A$, satisfying $(f^t)^\ast \circ j =
p$.

Let $\gamma\in \calP_B$. We claim that $p\cdot j\gamma = f^\ast
\gamma$ (and so $f^\ast(\calP_B)\subseteq p\calP_A$ holds).
To show that, it is enough to show that $(f^t)^\ast p j\gamma =
(f^t)^\ast f^\ast \gamma$. The right hand side is $p^\ast \gamma =
p^2 \gamma$, while $(f^t)^\ast p j\gamma = p(f^t)^\ast j\gamma =
p^2\gamma$.

We now define,
\[ \lambda_B:\calP_B \overset{\cong}{\Arr} \gera, \qquad \lambda_B = \frac{1}{p}\lambda_A \circ f^\ast.\]
It remains to show that the Deligne-Pappas condition holds for
$\uB$. By \cite[Proposition 3.1]{AndreattaGorenRamified}, it
is enough to show that for every prime $\ell$ (including $\ell =
p$), there is an element $\gamma^\prime$ of $\calP_B$ of degree
prime to $\ell$. Let $\gamma \in \calP_A$ be an element of degree
prime to $\ell$, which exists since $\uA$ satisfies the said
condition,  and let $\gamma^\prime = i\gamma$.

\

\id Let $f\colon \uA \arr \uB$ be an isogeny as in the statement of
the lemma and $H = \Ker(f)$. We only need to show that $H$ is
isotropic relative to $\calP_A$. Let $\gamma\in \calP_A$; to show
that $\gamma(H) \subseteq (A[p]/H)^\vee$, is to show that the
composition $(f^t)^\vee \circ \gamma(H) = 0$. Now, applying $f^\ast$
to $f^\ast \calP_B = p\calP_A$ we find $(f^t)^\ast \calP_A =
p\calP_B$. Hence, $(f^t)^\vee \circ \gamma\circ f^t = (f^t)^\ast
\gamma = \delta p = \delta \circ f \circ f^t$ for some $\delta \in
\calP_B$. Therefore, $(f^t)^\vee\circ \gamma (H)= \delta \circ f(H)
= 0$.

\end{proof}

\

\subsection{Some facts about Dieudonn\'e modules}
\label{subsection:Facts about D modules}

\id Let $k$ be a perfect field of positive characteristic~$p$. We
let $\DD$ denote the contravariant Dieudonn\'e functor, $G \mapsto
\DD(G)$, from finite commutative $p$-primary group schemes $G$ over
$k$, to finite length $W(k)$-modules $M$ equipped with two maps
$\Fr\colon M \arr M, \Ver\colon M \arr M$, such that $\Fr(\alpha m)
= \sigma(\alpha) \Fr(m), \Ver (\sigma(\alpha) m) = \alpha \Ver(m)$
for $\alpha \in W(k), m \in M$ and $\Fr\circ \Ver = \Ver\circ \Fr =
[p]$. This functor is an anti-equivalence of categories and commutes
with base change. It follows that if $G$ has rank $p^\ell$ the
length of $\DD(G)$ is $\ell$.

Given a morphism of group schemes $f\colon G \arr H$ we find that
\[ \DD(\Ker(f)) = \DD(G)/\DD(f)(\DD(H)),\]
where, in fact, $\DD(f)(\DD(H))$ depends only on $f(G)$.

Suppose  $f, g\colon G \arr H$ are two morphisms. By
considering the morphism $(f, g)\colon G \arr H \times H$ we find
that
\[ \DD(\Ker(f) \cap \Ker(g)) = \DD(G)/\DD((f, g))(\DD(H \times H)) =
\DD(G)/\left[\DD(f)(\DD(H)) +  \DD(g)(\DD(H))\right].\] On the other
hand, since $\Ker(f)\cap \Ker(g) = \Ker(f \vert_{\Ker(g)})$, we have
\[ \DD(\Ker(f) \cap \Ker(g)) = \DD(\Ker(g))/\DD(f)(\DD(H)),\]
where here we may replace $H$ by any subgroup scheme containing
$f(\Ker(g))$, if we wish.

\

\id $\bullet\;$ The Frobenius morphism
$\Fr_G\colon G \arr G^{(p)}$ induces a \emph{linear} map of
Dieudonn\'e modules
\[ \DD(\Fr_G): \DD(G^{(p)}) \arr \DD(G),\]
and, using that $\DD(G^{(p)}) = \DD(G\otimes_{W(k)} W(k))
=\DD(G)\otimes_{W(k)} W(k)$, which is a (right) $W(k)$-module via
$(m\otimes 1)s = m\otimes s = \sigma^{-1}(s) \cdot m \otimes 1$, we
get the $\sigma$-linear map
\[\Fr\colon \DD(G) \arr \DD(G), \qquad \Fr(tm) = \sigma(t) \Fr(m),\] via the inclusion $\DD(G) \arr
\DD(G)\otimes_{W(k)} W(k)$; it has the same image as $\DD(\Fr_A)$.
Similarly, the Verschiebung morphism $\Ver_G\colon  G \arr G^{(1/p)}$ induces the $\sigma^{-1}$-linear map $\Ver\colon\DD(G)\arr \DD(G)$.

\

\id $\bullet\;$ Let $A/k$ be a $g$-dimensional abelian variety and
$A[p]$ its $p$-torsion subgroup. Then $\DD(A[p])$ is a vector space
of dimension $2g$ over $k$. The group schemes $\Ker(\Fr_A),
\Ker(\Ver_A)$ are subgroups of $A[p]$ of rank $p^g$, where $\Fr_A: A \arr A^{(p)}, \Ver_A: A \arr
A^{(1/p)}$, are the usual morphisms. In fact, $A
\mapsto \Ker(\Fr_A)$ is a functor from abelian varieties over $k$ to
finite commutative group schemes, as follows from the following
commutative diagram
\[\xymatrix{A \ar[r]^{\Fr_A}\ar[d]_f & A^{(p)} \ar[d]^{f^{(p)}}\\
B \ar[r]^{\Fr_B} & B^{(p)}.}\] (Similarly for $A \mapsto
\Ker(\Ver_A)$.) In particular, any endomorphism of $A$ induces
endomorphisms on $\Ker(\Fr_A), \Ker(\Ver_A)$ and on
\[ \alpha(A):= \Ker(\Fr_A) \cap \Ker(\Ver_A).\]
Note that we have,
\[ \DD(\Ker(\Fr_A)) = \DD(A[p])/\DD(\Fr_A)(\DD(A^{(p)}[p]))  = \DD(A[p])/\Fr(\DD(A[p])), \]
and similarly for Verschiebung,
\[ \DD(\Ker(\Ver_A)) = \DD(A[p])/\DD(\Ver_A)(\DD(A^{(1/p)}[p]))  = \DD(A[p])/\Ver(\DD(A[p])). \]
Rank considerations give that $\Fr(\DD(A[p])) = \Ker(\Ver: \DD(A[p])
\arr \DD(A[p]))$ and $\Ver(\DD(A[p])) = \Ker(\Fr: \DD(A[p]) \arr
\DD(A[p]))$.

\

\id $\bullet\;$ The Dieudonn\'e modules of $\Ker(\Fr_A),
\Ker(\Ver_A)$ and $\Ker(A[p])$ are linked to cohomology by the
following commutative diagram \cite{Oda}: \begin{equation}
\label{equation: D modules and cohomology} \xymatrix{0 \ar[r] &
H^0(A, \Omega^1_{A/k})\ar@{=}[d] \ar[r] & H^1_{\rm dR}(A/k) \ar[r]
\ar@{=}[d]& H^1(A, \calO_A) \ar[r]\ar@{=}[d] & 0
\\
0 \ar[r] & \DD(\Ker(\Fr_A))\otimes_k k \ar@{=}[d] \ar[r] & \DD(A[p])
\ar[r] & \DD(\Ker(\Ver_A)) \ar[r] & 0, \\
& \DD(\Ker(\Fr_{A^{(1/p)}})) & & }\end{equation} functorially in
$A$; in the tensor sign $\otimes_kk$,  $k$ is viewed a left
$k$-module relative to the map $a \mapsto a^{1/p}$. In fact, once
one has established a canonical isomorphism $H^1_{\rm dR}(A/k) =
\DD(A[p])$ the rest follows from the theory above.

\

\id $\bullet\;$ Suppose that $A$ has real multiplication,
$\iota\colon \ol \injects \End_k(A)$. Then we have a decomposition
\[ \DD(A[p])  = \oplus_{\beta\in \BB} \DD(A[p])_\beta,\]
where $\DD(A[p])_\beta$ is a two dimensional vector space over $k$
on which $\ol$ acts via $\beta$. The maps $\Fr$ and $\Ver$ act thus:
\[ \Fr \colon  \DD(A[p])_\beta \Arr  \DD(A[p])_{\sigma \circ \beta}, \qquad
\Ver \colon  \DD(A[p])_{\sigma \circ \beta} \Arr
\DD(A[p])_{\beta}.\] Now suppose that $(A, \iota)$ satisfy that
$\Lie(A)$ is a locally free $\ol$-module. It then follows easily
that for every $\beta\in \BB$ the $\beta$ component of the
Dieudonn\'e submodules $\Ker(\Fr)=\Ima(\Ver), \Ima(\Fr)=\Ker(\Ver)$
of $\DD(A[p])$ are one-dimensional over $k$, and, similarly
\[\DD(\Ker(\Fr_A)) = \oplus_{\beta\in \BB} \DD(\Ker(\Fr_A))_{\beta},
\qquad \DD(\Ker(\Ver_A)) = \oplus_{\beta\in \BB}
\DD(\Ker(\Ver_A))_{\beta},\] is a decomposition into one dimensional
$k$-vector spaces.

\

\subsection{Discrete invariants for the points of $\Ybar$}\label{subsec: discrete invariants}
Let $k\supseteq \kappa$ be a field. For $f\colon \uA \arr \uB$
defined over $k$ as in Lemma \ref{lemma: alternative formulation} there is a unique $\ol$-isogeny $f^t\colon
\uB \arr \uA$ such that
\[f^t\circ f = [p_A], \qquad f\circ f^t = [p_B].\]
(For the relation between $f^t$ and the dual isogeny $f^\vee$
see Diagram~(\ref{Diagram: ft and fvee}).) We have induced
homomorphisms:
\begin{align}
\bigoplus_{\beta\in \BB}\Lie(f)_\beta & \colon \bigoplus_{\beta\in
\BB}\Lie(\uA)_\beta \Arr \bigoplus_{\beta\in \BB}\Lie(\uB)_\beta,
\\\nonumber
 \bigoplus_{\beta\in \BB}\Lie(f^t)_\beta & \colon \bigoplus_{\beta\in
\BB}\Lie(\uB)_\beta \Arr \bigoplus_{\beta\in \BB}\Lie(\uA)_\beta.
\end{align}
We note that since $\Lie(\uA)$ is a free $\ol\otimes k$-module,
$\Lie(\uA)_\beta$ is a one dimensional $k$-vector space. Using these
decompositions we define several discrete invariants associated to
the data $(f\colon \uA \arr \uB)$. We let
\begin{align}
\varphi(f) & = \varphi(\uA, H)  = \{ \beta\in \BB\colon
\Lie(f)_{\sigma^{-1} \circ \beta} = 0\}, \notag\\
\eta(f) & = \eta(\uA, H)  = \{ \beta\in \BB\colon \Lie(f^t)_\beta =
0\}, \\
I(f) & = I(\uA, H) = \ell(\varphi(f)) \cap \eta(f) = \{ \beta\in
\BB\colon \Lie(f)_\beta = \Lie(f^t)_\beta= 0\}.\notag
\end{align}
The elements of $I(\uA, H)$ are the \emph{critical indices} of
\cite{Stamm}.
\begin{dfn}
Let $\pe$ be a pair of subsets of $\BB$. We say that $(\varphi,
\eta)$ is an \emph{admissible pair} if $\ell(\varphi^c) \subseteq \eta$.
Given another admissible pair $(\varphi', \eta')$ we say that
\[(\varphi', \eta') \geq \pe,\] if both inclusions $\varphi' \supseteq
\varphi, \eta' \supseteq \eta$ hold.
\end{dfn}

\begin{prop}\begin{enumerate}
\item Let $(\varphi, \eta)$ be an admissible pair. Then $r(\eta^c)
\subseteq \varphi$, and this identity is equivalent to the
admissibility of $(\varphi, \eta)$. Let $I = \ell(\varphi) \cap
\eta$ then
\[\varphi = r(\eta^c) \textstyle\coprod r(I), \qquad \eta = \ell(\varphi^c) \textstyle\coprod I.\]
\item There are $3^g$ admissible pairs.
\item Let $k\supseteq \kappa$ be a field.
Let $(\uA, H)$ correspond to a $k$-rational point of $\Ybar$ then
$(\varphi(\uA, H), \eta(\uA, H))$ is an admissible pair.
\end{enumerate}
\end{prop}
\begin{proof} The first part is elementary. By first choosing $\varphi$ and then choosing $\eta$ subject to
the admissibility condition, the second part follows from the
identity $\sum_{i = 0}^g \binom{g}{i}2^{i} = (1+2)^g= 3^g.$
 Consider the third part. The condition
$f\circ f^t = [p]$ implies for every $\beta\in \BB$ the equality
$\Lie(f)_\beta \circ \Lie(f^t)_\beta = 0$. That means that if
$\beta\not\in\eta(f)$ then $\Lie(f)_\beta  = 0$ and so $\sigma\circ
\beta \in \varphi(f)$, that is $r(\eta(f)^c) \subseteq \varphi(f)$.
\end{proof}

Let $k$ be a prefect field of characteristic $p$. Given $\uA/k$ the
\emph{type} of $\uA$ is defined by
\begin{equation} \tau(\uA) = \{\beta \in \BB: \DD\left(\Ker(\Fr_A)\cap
\Ker(\Ver_A)\right)_\beta\neq 0\}.
\end{equation}
One may also define the type $\tau(\uA)$ as $\{\sigma \circ\beta:
\Ker\left(\Fr_A\colon H^1(A, \calO_A) \arr H^1(A,
\calO_A)\right)_\beta \neq 0\}$. It is an exercise to check that
this definition is equivalent to the one given above. The virtue of
this alternative definition is that it also holds when $A$ is
defined over a non-perfect field $k$, and is stable under base
change. Thus, if $k' $ is a perfect field containing $k$ and $\uA$
is defined over $k$, $\tau(\uA) = \tau(\uA \otimes_k k')$, under any
definition of the right hand side.

Basic properties of Dieudonn\'e modules discussed in \S \ref{subsection:Facts about D modules} imply that
\[\DD(\Ker(\Fr_A)\cap \Ker(\Ver_A)) =
\DD(A[p])/(\Ima \,\DD(\Fr_A) + \Ima\, \DD(\Ver_A)) =
\DD(A[p])/(\Ima\, \Fr + \Ima \,\Ver).\] Since $\DD(A[p])_\beta$ is a two
dimensional $k$-vector space and both $\Ima(\DD(\Fr_A))_\beta$ and
$\Ima(\DD(\Ver_A))_\beta$ are one dimensional, the first assertion
of the following lemma holds.

\begin{lem} \label{lemma: type in terms of Dieudonne modules} Let $\uA$ be as above and $(f\colon \uA \arr \uB)$ a
$k$-rational point of $\Ybar$.
\begin{enumerate}
\item $\beta\in \tau(\uA)$ if and only if one of the following equivalent statements hold:
\begin{enumerate}
\item $\Ima(\DD(\Fr_A))_\beta =
\Ima(\DD(\Ver_A))_\beta$.
\item$\Ima(\Fr)_\beta =
\Ima(\Ver)_\beta$.
\item$\Ker(\Fr)_\beta =
\Ker(\Ver)_\beta $.
\end{enumerate}
\item $\beta \in \varphi(f) \Longleftrightarrow \Ima(\DD(\Fr_A))_\beta
= \Ima(\DD(f))_\beta$.
\item $\beta \in \eta(f) \Longleftrightarrow \Ima(\DD(\Ver_A))_\beta
= \Ima(\DD(f))_\beta$.
\end{enumerate}
(All Dieudonn\'e submodules appearing above are inside $\DD(A[p])$.)
\end{lem}
\begin{proof} The first assertion was already proven above. To prove
(2) we recall the following
commutative diagram:
\[\xymatrix{0 \ar[r] & H^0(A, \Omega^1_{A/k})\ar@{=}[d] \ar[r] & H^1_{\rm dR}(A/k) \ar[r] \ar@{=}[d]& H^1(A, \calO_A) \ar[r]\ar@{=}[d] & 0
\\
0 \ar[r] & \DD(\Ker(\Fr_A))\otimes_k k \ar@{=}[d] \ar[r] & \DD(A[p])
\ar[r] & \DD(\Ker(\Ver_A)) \ar[r] & 0, \\
& \DD(\Ker(\Fr_{A^{(1/p)}})) & & }\] which is functorial in $A$.   The map $\Lie(f)\colon \Lie(A) \arr \Lie(B)$
induces the map
\[
f^\ast\colon \Lie(B)^\ast = H^0(B, \Omega^1_{B/k})
\arr \Lie(A)^\ast = H^0(A, \Omega^1_{A/k}),
\]
which is precisely the
pull-back map $f^\ast$ on differentials. The map $f^\ast$ has
isotypic decomposition relative to the $\ol\otimes k$-module structure.

Now, $\beta \in \varphi(f) \iff \Lie(f)_{\sigma^{-1} \circ \beta} =
0 \iff f^\ast_{\sigma^{-1} \circ \beta} = 0$. Via the identifications in the above diagram, the map $f^\ast$ can also be viewed as
a map \[f^\ast\colon \DD(\Ker(\Fr_{B^{(1/p)}})) \arr
\DD(\Ker(\Fr_{A^{(1/p)}})),\] which is equal to the linear map
$\DD(f^{(1/p)}\vert_{\Ker(\Fr_{A^{(1/p)}})})$. So,
\begin{align*}
f^\ast_{\sigma^{-1} \circ \beta} = 0 & \Longleftrightarrow
\DD(f^{(1/p)}\vert_{\Ker(\Fr_{A^{(1/p)}})})_{\sigma^{-1}\circ
\beta} = 0
\\ & \Longleftrightarrow
\DD(f\vert_{\Ker(\Fr_A)})_\beta = 0.
\end{align*} We  therefore have,
\[ \beta \in \varphi(f) \Longleftrightarrow \DD(f\vert_{\Ker(\Fr_A)})_\beta = 0.\]
Now, $\DD(f \vert_{\Ker(\Fr_A)})_\beta = 0$ if and only if
$\left[\DD(\Ker(\Fr_A))/\DD(f)(\DD(\Ker \;\Fr_B))\right]_\beta \neq
0$ and that is equivalent to
$\DD(A[p])_\beta/\left[\DD(f)(\DD(B[p])) +
\DD(\Fr_A)(\DD(A^{(p)}[p])) \right]_\beta \neq 0$. By considering
dimensions over $k$ we see that this happens if and only if  $\Ima(\DD(f))_\beta
= \Ima(\DD(\Fr_A))_\beta$, as the lemma states.

\

We first show that
\[\Lie(f^t)_\beta = 0 \Longleftrightarrow H^1(f)_\beta = 0.\]
Let $\gamma\in\calP_A$ be an isogeny of degree prime to $p$. Let
$i\gamma\in\calP_B$ be the isogeny constructed in Diagram
(\ref{equation: i-gamma}). Since $f^\vee \circ i\gamma
f=f^*\gamma=p\gamma=\gamma p=\gamma \circ f^t \circ f$, it follows
that $f^\vee \circ i\gamma=\gamma \circ  f^t$ and so the following
diagram is commutative:
 \begin{equation}\label{Diagram: ft and fvee}
 \xymatrix{ A \ar[r]^{\gamma} & A^\vee \\
B \ar[u]^{f^t}\ar[r]_{i\gamma}& B^\vee\ar[u]_{f^\vee}.}
 \end{equation}

Applying $\Lie(\cdot)_\beta$ to the diagram, we obtain
\[\xymatrix{ \Lie(A)_\beta \ar[r]^{\cong} & \Lie(A^\vee)_\beta \\
\Lie(B)_\beta \ar[u]^{\Lie(f^t)_\beta}\ar[r]^{\cong}&
\Lie(B^\vee)_\beta\ar[u]_{\Lie(f^\vee)_\beta},} \] and, hence,
\[ \Lie(f^t)_\beta = 0 \Longleftrightarrow \Lie(f^\vee)_\beta = 0.\]
Since we have a commutative diagram:
\[\xymatrix{
\Lie(A^\vee) \ar[r]^{\cong} & H^1(A, \calO_A) \\
\Lie(B^\vee) \ar[r]^{\cong}\ar[u]^{\Lie(f^\vee)} & H^1(B, \calO_B)
\ar[u]_{H^1(f)}, }\] where we can pass to $\beta$-components, we
conclude that $\Lie(f^t)_\beta = 0 \Longleftrightarrow H^1(f)_\beta
= 0$.

The map $H^1(f)$ can be viewed as $\DD(f\vert_{\Ker(\Ver_A)}):\DD(\Ker(\Ver_B))\arr \DD(\Ker(\Ver_A))$, and hence, $H^1(f)_\beta= 0$ if and only if
$\DD(f\vert_{\Ker(\Ver_A)})_\beta=0$.  This is equivalent to
\[
\DD(A[p])_\beta/\left[\DD(f)(\DD(B[p])) +
\DD(\Ver_A)(\DD(A^{(p)}[p])) \right]_\beta \neq 0.
\]
Dimension considerations show that this happens if and only if  $\Ima(\DD(f))_\beta
= \Ima(\DD(\Ver_A))_\beta$.
\end{proof}
\begin{cor} \label{corollary: type and phi eta}
The following inclusions hold.
\begin{enumerate}
\item $\varphi(\uA, H) \cap \eta(\uA, H) \subseteq \tau(\uA)$.
\item $\varphi(\uA, H)^c \cap \eta(\uA, H) \subseteq \tau(\uA)^c$.
\item $\varphi(\uA, H) \cap \eta(\uA, H)^c \subseteq \tau(\uA)^c$.
\end{enumerate}
If we denote for two sets $S, T$ their symmetric difference by $S
\vartriangle T = (S - T) \cup (T - S)$, we can then formulate these
statements as
\[\varphi(\uA, H) \cap \eta(\uA, H) \subseteq \tau(\uA), \qquad  \tau(\uA)\subseteq \left[\varphi(\uA, H) \vartriangle \eta(\uA, H) \right]^c. \]
\end{cor}

\

\subsection{The infinitesimal nature of $\Ybar$}\label{subsection: infinitesimal nature of Ybar}
In \cite{Stamm} Stamm studied the completed local ring of $Y$ at a
closed point $\Qbar$ of its special fiber and concluded
Theorem~\ref{thm: Stamm's theorem} below. Since then, local
deformation theory of abelian varieties, and in particular the
theory of local models, have developed and we have found it more
enlightening to describe Stamm's result in that language. Our
approach is not different in essence from Stamm's, but results such
as Lemma~\ref{lemma: U phi and V phi} become more transparent in our
description. Our focus is on the completed local rings of $\Ybar$ at
a point $\Qbar$ defined over a perfect field $k\supseteq \kappa$ of
characteristic $p$.

As in \cite{DP}, one constructs a morphism from a Zariski-open
neighborhood $T\subset \Ybar$ of $\Qbar$ to the Grassmann variety
${\bf G}$ associated to the data: $H=(\ol\otimes k)^2$, two free
$\ol\otimes k$-sub-modules of $H$, say $W_1, W_2$, such that under
the $\ol\otimes k$ map $h:H \arr H$ given by $(x, y)\mapsto (y, 0)$,
we have $h(W_1) \subseteq W_2, h(W_2) \subseteq W_1$. Notice that we
can perform the usual decomposition according to $\ol$-eigenspaces
to get
\[ h = \oplus_\beta h_\beta: \oplus_\beta k^2_\beta \arr \oplus_\beta k^2_\beta, \]
such that each $h_\beta$ is the linear transformation
corresponding to two-by-two matrix $M=\left(\begin{smallmatrix} 0 & 1 \\
0 & 0
\end{smallmatrix}\right)$. Furthermore, $W_i = \oplus_\beta
(W_i)_\beta$,
and $(W_i)_\beta$ is a one-dimensional $k$-vector space contained in
$k^2$. We have $MW_1 \subseteq W_2, MW_2 \subseteq W_1$.

The basis for this construction is Grothendieck's crystalline theory
\cite{Grothendieck}; see also \cite{de Jong}.  Let $(f\colon \uA
\arr \uB)$ correspond to $\Qbar$. The $\ol\otimes k$-module $H$ is
isomorphic to $H^1_{\rm dR}(\uA/k)$. By the elementary divisors
theorem, we can then identify $H^1_{\rm dR}(\uB/k)$ with $H$, and
possibly adjust the identification of $H^1_{\rm dR}(\uA/k)$ with
$H$, such that the induced maps $f^\ast$ and $(f^t)^\ast$ are both
the map $h$ defined above. Let $W_A = H^0(A, \Omega^1_{A/k}) =
\Lie(A)^\ast \subset H$ be the Hodge flitration, and similarly for
$W_B$. Then we have $h(W_A) \subseteq W_B, h(W_B) \subseteq W_A$, and so we get a point ${\bf Q}$ of the Grassmann
variety ${\bf G}$ described above. Let $\calO = \OhatQbar$ and
$\germ$ be the maximal ideal. By Grothendieck's theory, the
deformations of $(f\colon\uA \arr \uB)$ over $R:=\calO/\germ^p$
(which carries a canonical divided power structure) are given by
deformation of the Hodge filtration over that quotient ring. Namely,
are in bijection with free, direct summands, $\ol\otimes R$-modules
$(W_A^R, W_B^R)$ of rank one of $H \otimes_k \calO/\germ^p =
(\ol\otimes R)^2$ such that $h(W_A^R) \subseteq W_B^R,
h(W_B^R)\subseteq W_A^R$, and $W_A^R \otimes k = W_A, W_B^R \otimes
k = W_B$. This, by the universal property of the Grassmann variety
is exactly $\widehat{\calO}_{\bf G, Q}/\germ^p_{\bf G, Q}$. A
boot-strapping argument as in \cite{DP} furnishes an isomorphism of
the completed local rings themselves, even in the arithmetic
setting.

 To study the singularities and uniformization of the completed local
rings, we may reduce, by considering each $\beta\in \BB$
individually, to the case of the Grassmann variety parameterizing
two one dimensional subspaces $W_1, W_2$ of $k^2$ that are
compatible: $MW_1 \subseteq W_2, MW_2 \subseteq W_1$.  (We have simplified the notation from $(W_{1})_\beta$ to $W_1$, etc.) Fix
then such a pair $(W_1, W_2)$. If $W_1 \neq \Ker(M)$ then $W_2 =
MW_1 = \Ker(M)$, and
the same holds for any deformation of $W_1$ and so $W_2$ is
constant, being $\Ker(M)$. In this case, we see that the local
deformation ring is $k[\![x]\!]$, where the choice of letter $x$
indicates that it is $W_1$ that is being deformed. If $W_1= \Ker(M)$
and $W_2 \neq \Ker(M)$ then the situation is similar and we see that
the local deformation ring is $k[\![y]\!]$, where the choice of
letter $y$ indicates that it is $W_2$ that is being deformed.
Finally, suppose both $W_1= \Ker(M)$ and $W_2 = \Ker(M)$. The
subspace $W_i$ is spanned by $(1, 0)$ and a deformation of it to a
local artinian $k$-algebra $D$ is uniquely described by a basis
vector $(1, d_i)$ where $d_i\in \germ_D$. The condition that the
deformations are compatible under $f$ is precisely $d_1d_2=0$ and so
we see that the local deformation ring is $k[\![x,y]\!]/(xy)$.

Returning to the situation of abelian varieties
$(f\colon A \arr B)$, corresponding to a point $\Qbar$, the pair $(W_1, W_2)$ is
$(H^0(A, \Omega^1_{A/k})_\beta, H^0(B, \Omega^1_{B/k})_\beta) =
(W_{A,\beta}, W_{B, \beta})$ for $\beta \in \BB$, and the condition
$W_{A,\beta}  = \Ker((f^t)^\ast)_\beta$ is
the condition $\beta \in \eta(\Qbar)$, while the condition
$W_{B,\beta} = \Ker(f^\ast)_\beta$ is the
condition that $\Lie(f)_\beta = 0$, namely, $\sigma \circ \beta \in
\varphi(\Qbar)$. Our discussion,
 therefore, gives immediately the following result.
\begin{thm}\label{thm: Stamm's theorem}
Let $(\uA, f)$ correspond to a point $\Qbar$ of $\Ybar$, defined
over a field $k \supseteq \kappa$. Let $\varphi = \varphi(\uA, f),
\eta = \eta(\uA, f)$ and $I = I(\uA, f)= \ell(\varphi) \cap \eta$,
then
\begin{equation}
\label{equation: local deformation ring at Q bar -- our formulation}
 \widehat{\calO}_{\Ybar, \Qbar} \cong k [\![ \{x_\beta : \beta \in \ell(\varphi)\}, \{y_\beta: \beta \in
 \eta\}]\!]/(\{x_\beta y_\beta: \beta\in
 I\}).
\end{equation}
\end{thm}  This is basically Stamm's theorem,
only that Stamm collects together the variables in the following
fashion and works over $W(k)$ (which is an easy extension of the
argument above).

 \addtocounter{thm}{-1}
\begin{thm}[bis]\label{thm: Stamm's theorem}
Let $(\uA, f)$ correspond to a point $\Qbar$ of $\Ybar$, defined
over a field $k \supseteq \kappa$. Let $I = I(\uA, f)$ then
\begin{equation}
\label{equation: local deformation ring at Q bar}
 \widehat{\calO}_{\Ybar, \Qbar} \cong k [\![ \{x_\beta,
 y_\beta: \beta \in I\}, \{z_\beta: \beta \in I^c\}]\!]/(\{x_\beta y_\beta: \beta\in
 I\}).
\end{equation}
The isomorphism lifts to an isomorphism
\begin{equation}
\label{equation: local deformation ring at Q bar -- arithmetic version}
 \widehat{\calO}_{Y, \Qbar} \cong W(k) [\![ \{x_\beta,
 y_\beta: \beta \in I\}, \{z_\beta: \beta \in I^c\}]\!]/(\{x_\beta y_\beta - p: \beta\in
 I\}).
\end{equation}
\end{thm}

\id The following lemma gives information about a certain
stratification of $\Ybar$ that is the precursor to the
stratification $\{\Wpe\}$ studied extensively in this paper.
\begin{lem}\label{lemma: U phi and V phi}
Given $\varphi\subseteq \BB$ (respectively, $\eta \subseteq \BB$)
there is a locally closed subset $\Up$, and a closed subset
$\Upplus$ (resp. $\Ve$ and $\Veplus$) of $\Ybar$ such that $\Up$
consists of the closed points $\Qbar$ with $\varphi(\Qbar) =
\varphi$, and $\Upplus$ consists of the closed point $\Qbar$ with
$\varphi(\Qbar) \supseteq \varphi$ (resp., the points $\Qbar$ such
that $\eta(\Qbar) = \eta$ and $\eta(\Qbar) \supseteq \eta$).

Furthermore, if $\Qbar\in U_\beta^+$, then $U_\beta^+\cap\Spf(\OhatQbar)$ is equal to $\Spf(\OhatQbar)$ if $\beta\not\in r(I)$, and is otherwise given by the vanishing of $y_{\sigma^{-1}\circ \beta}$.
Similarly, if $\Qbar\in V_\beta^+$, then $V_\beta^+\cap\Spf(\OhatQbar)$ is equal to $\Spf(\OhatQbar)$ if $\beta\not\in I$, and is otherwise given by the vanishing of $x_\beta$.

\end{lem}

\begin{proof}
It suffices to prove that the $\Upplus$ (resp. $\Veplus$) are
closed, because $\Up = \Upplus  - \bigcup_{\varphi' \supsetneq
\varphi} U_{\varphi'}^+$ (resp., $\Ve = \Veplus  - \bigcup_{\eta'
\supsetneq \eta} V_{\eta'}^+$). Furthermore, since $\Upplus =
\bigcap_{\beta\in \varphi}U_{\{\beta\}}^+$, we reduce to the case
where $\varphi = \{\beta\}$ is a singleton (and similarly for
$\Veplus$). From this point we only discuss the case of
$U_{\{\beta\}}^+$, as it is clear that the same arguments will work
for $V_{\{\beta\}}^+$.

Recall that $\Qbar$, corresponding to $(f\colon \uA \arr \uB)$,
satisfies $\varphi(\Qbar) \supseteq \{\beta\}$, if and only if
$\Lie(f)_{\sigma^{-1}\circ \beta}~=~0$. Over $\Ybar$, $\Lie(\uA^{\rm
univ})_\gamma$ and $\Lie(\uB^{\rm univ})_\gamma$, $\gamma \in \BB$,
are line bundles, and
\[
\Lie(f)_\gamma\colon \Lie(\uA^{\rm
univ})_\gamma\Arr\Lie(\uB^{\rm univ})_\gamma
\]
is a morphism of line bundles and consequently its degeneracy locus
$\{\Lie(f)_\gamma = 0\}$ is closed.

Moreover, it follows directly from the above description of the
variables $x_\beta,y_\beta$, and the paragraph before Theorem
\ref{thm: Stamm's theorem} that if $\Qbar\in U_\beta^+$ and
$\beta\in r(I)$,  then $U_\beta^+\cap\Spf(\OhatQbar)$  is given by
the vanishing of $y_{\sigma^{-1}\circ \beta}$. Assume now that
$\Qbar\in U_{\{\beta\}}^+$ and $\beta\not\in r(I(\Qbar))$. Since
$\beta\in\varphi(\Qbar)$, we have $\beta\not\in r(\eta(\Qbar))$. We
show that
\[
U=\bigcup_{(\varphi,\eta)\leq(\varphi(\Qbar),\eta(\Qbar))} U_\varphi
\cap V_\eta
\]
is a Zariski open subset of $\Ybar$ which contains $\Qbar$ and lies
entirely inside $U_{\{\beta\}}^+$. First note that if
$(\varphi(\Qbar^\prime),\eta(\Qbar^\prime))\leq
(\varphi(\Qbar),\eta(\Qbar))$, then $\beta\in
r(\eta(\Qbar)^c)\subseteq r(\eta(\Qbar^\prime)^c) \subseteq
\varphi(\Qbar^\prime)$, proving that $U \subseteq U_{\{\beta\}}^+$.
Secondly, it is clear that $\Qbar \in U$. Finally, $U$ is a Zariski
open subset of $\Ybar$, as we have
\[
\Ybar-U=\bigcup_{\Qbar\not\in U_\varphi^+ \cap V_\eta^+ }
U_\varphi^+ \cap V_\eta^+
\]
from definitions.
\end{proof}

\

\subsection{Stratification of $\Ybar$} \label{subsection: stratification of Y bar}
\begin{prop}
For an admissible pair $\pe$ there is a locally closed subset $\Wpe$
of $\Ybar$ with the following property: A closed point $\Qbar$ of
$\Ybar$ has invariants $\pe$ if and only if $\Qbar \in \Wpe$.
Moreover, the subset
\[ \Zpe = \bigcup_{(\varphi', \eta') \geq \pe} W_{(\varphi', \eta')}\]
is closed.
\end{prop}
\begin{proof}

The proposition follows from Lemma \ref{lemma: U phi and V phi} as we can define \[\Wpe = \Up \cap \Ve,
\qquad \Zpe = \Upplus \cap \Veplus.\]

\end{proof}

Let $\tau \subseteq \BB$. Recall the stratification on $\Xbar$
introduced in \cite{GorenOort, G}. There is a locally closed subset
$W_\tau$ of $\Xbar$ with the property that a closed point $\Pbar$ of
$\Xbar$ corresponding to $\uA$ belongs to $W_\tau$ if and only if
$\tau(\uA) = \tau$. There is a closed subset $Z_\tau$ of $\Xbar$
with the property that a closed point $\Pbar$ of $\Xbar$
corresponding to $\uA$ belongs to $Z_\tau$ if and only if $\tau(\uA)
\supseteq \tau$. The main properties of these sets are the
following:
\begin{enumerate}
\item The collection $\{W_\tau: \tau \subseteq \BB\}$ is a
stratification of $\Xbar$ and $\overline{W}_\tau = Z_\tau =
\bigcup_{\tau' \supseteq \tau} W_{\tau'}$.
\item Each $W_\tau$ is non-empty, regular, quasi-affine and equi-dimensional
of dimension $g - \sharp\; \tau$. \item The strata $\{W_\tau\}$
intersect transversally. In fact, let $\Pbar$ be a closed
$k$-rational point of $\Xbar$. There is a choice of isomorphism
\begin{equation}\label{equation: local def ring of X at P bar} \widehat{\calO}_{X, \Pbar} \cong
W(k)[\![t_\beta: \beta \in \BB ]\!],
\end{equation}
inducing
\begin{equation}\label{equation: local def ring at P bar} \widehat{\calO}_{\Xbar, \Pbar} \cong
k[\![t_\beta: \beta \in \BB ]\!],
\end{equation}
such  that for $\tau' \subseteq \tau(\Pbar)$, $W_{\tau'}$ (and $Z_{\tau'})$ are
given in $\Spf (\widehat{\calO}_{\Xbar, \Pbar})$ be the equations
$\{t_\beta = 0: \beta \in \tau'\}$.
\end{enumerate}
Let $\epsilon: \uA^{\rm univ} \arr \Xbar$ be the universal object.
The Hodge bundle $\calL = \epsilon_\ast \Omega^1_{\uA^{\rm
univ}/\Xbar}$ is a locally free sheaf of
$\ol\otimes\calO_{\Xbar}$-modules and so decomposes in line bundles
$\calL_\beta$, $\calL = \oplus_\beta \calL_\beta$. Let $h_\beta$ be
the partial Hasse invariant, which is a Hilbert modular form of
weight $p\sigma^{-1}\circ\beta-\beta$, i.e., a global section of the
line bundle $\calL^p_{\sigma^{-1}\circ \beta}\otimes
\calL_\beta^{-1}$, as in \cite{G}.  Then, the divisor of $h_\beta$
is reduced and equal to $Z_{\{\beta\}}$. For every closed point
$\Pbar \in Z_{\{\beta\}}$, one can trivialize the line bundle
$\calL^p_{\sigma^{-1}\circ \beta}\otimes \calL_\beta^{-1}$ over
$\Spf(\OhatPbar)$ and thus view $h_\beta$ as an element of
$\OhatPbar$. The variable $t_\beta$ can be chosen to coincide with
that function $h_\beta$. We now prove the fundamental
properties of the stratification of $\Ybar$.

\begin{thm}\label{theorem: fund'l facts about the stratificaiton of Ybar}
Let $\pe$ be an admissible pair, $I = \ell(\varphi) \cap
\eta$.
\begin{enumerate}
\item $\Wpe$ is non-empty, and its Zariski closure is $\Zpe$. The
collection $\{\Wpe\}$ is a stratification of $\Ybar$ by $3^g$
strata.
\item $\Wpe$ and $\Zpe$ are equi-dimensional, and
\[ \dim (\Wpe) = \dim(\Zpe) = 2g - (\sharp\; \varphi + \sharp \eta).\]
\item The irreducible components of $\Ybar$ are the irreducible
components of the strata $Z_{\varphi, \ell(\varphi^c)}$ for $\varphi
\subseteq \BB$.
\item Let $\Qbar$ be a closed point of $\Ybar$ with invariants
$\pe$, $I = \ell(\varphi) \cap \eta$. For an admissible pair
$\peprime$, we have $\Qbar\in Z_{\varphi^\prime,\eta^\prime}$ if and only if we have:
\[ \varphi \supseteq \varphi' \supseteq \varphi - r(I), \qquad \eta \supseteq \eta' \supseteq \eta - I.\]
In that case, write $\varphi' = \varphi - J, \eta' = \eta - K$ (so that $\ell(J) \subseteq I, K \subseteq I$ and
$\ell(J) \cap K = \emptyset$). We have:
\[
\widehat{\calO}_{Z_{\varphi', \eta'}, \Qbar} = \widehat{\calO}_{\Ybar, \Qbar}/\calI,
\]
where $\calI$ is the ideal
\[\calI = \left\langle \{x_\beta: \beta \in I - K\}, \{y_\gamma: \gamma \in I - \ell(J)\} \right\rangle.\]
This implies that each stratum in the stratification $\{\Zpe\}$ is
non-singular.
\end{enumerate}
\end{thm}
\begin{proof} We begin with the proof of assertion
(4), keeping the notation $\varphi = \varphi(\Qbar), \eta =
\eta(\Qbar), I = I(\Qbar)$. We need the following fact.

\

\id {\bf Claim:} \emph{There exists a Zariski open  set $U$ with
$\Qbar \in U$, such that for every closed point $\Qbar'\in U$ one
has
\[\varphi(\Qbar') \supseteq \varphi - r(I), \qquad \eta(\Qbar') \supseteq \eta - I.\]
In words, locally Zariski, $\pe$ can become smaller only at
$\beta\in I$}.

\

To prove the claim, choose
\[ U = \Ybar - \left[\bigcup_{\beta\in r(\eta - I)}U_\beta^+ \;\cup \bigcup_{\beta \in \ell(\varphi) - I}V_\beta^+\right].\]
We verify that this choice of $U$ is adequate. Firstly, $\Qbar \in
U$. Indeed, since $\varphi = r(\eta^c) \cup r(I)$, if
$\beta \in r(\eta) - r(I)$ then $\beta \not\in \varphi$. Similarly,
since $\eta = \ell(\varphi^c) \cup I$, if $\beta\in
\ell(\varphi) - I$ then $\beta \not\in \eta$. That shows that $\Qbar
\not\in \bigcup_{\beta\in r(\eta - I)}U_\beta^+ \;\cup
\bigcup_{\beta \in \ell(\varphi) - I}V_\beta^+$ and so that $\Qbar
\in U$.

Let $\Qbar' \in U$ then: (i) Since $\varphi(\Qbar') \subseteq r(\eta
- I)^c$ we have $\ell(\varphi(\Qbar')^c) \supseteq \eta - I$ and by
admissibility $\eta(\Qbar') \supseteq \eta - I$; (ii) $\eta(\Qbar')
\subseteq (\ell(\varphi) - I)^c$ and so $r(\eta(\Qbar')^c) \supseteq
\varphi - r(I)$ and admissibility gives $\varphi(\Qbar') \supseteq
\varphi - r(I)$. Thus, the set $U$ is contained in $U^+_{\varphi -
r(I)} \cap V^+_{\eta - I}$, and our claim is proved.

\

By Lemma \ref{lemma: U phi and V phi}, for the
ideal $\calI$ in the theorem,
\begin{align*}
V(\calI) & = \bigcap_{\gamma\in I - \ell(J)}U_{\sigma\circ \gamma}^+
\; \cap \bigcap_{\beta\in I - K} V_\beta^+ \; \cap \;
\Spf(\widehat{\calO}_{\Ybar, \Qbar})\\
& = \left(\bigcap_{\gamma\in I - \ell(J)}U_{\sigma\circ \gamma}^+ \;
\cap \; U^+_{\varphi - r(I)}\right) \; \cap \;
\left(\bigcap_{\beta\in I - K} V_\beta^+ \; \cap \; V^+_{\eta - I}
\right)\; \cap \;
\Spf(\widehat{\calO}_{\Ybar, \Qbar}) \\
& = Z_{\varphi - J, \eta - K}.
\end{align*}
(We made use of the Claim in the second equality.) This concludes
the proof of assertion (4).

\

\id
 There is a point $\Qbar$ with invariants $\pe = (\BB, \BB)$ -- it
corresponds to $(\uA, H)$, where $\uA$ is superspecial and $H$ is
the kernel of Frobenius. The point $\Qbar$ belongs to every strata
$\Zpe$ and hence each $\Zpe$ is non-empty. The above computations
also show that $\Zpe$ is  pure dimensional and $\dim(\Zpe) = 2g -
(\sharp\; \varphi + \sharp\; \eta)$.

Since $\Zpe-\Wpe=\bigcup_{(\varphi', \eta')\gneqq \pe}Z_{(\varphi', \eta')}$ is a union of lower-dimensional strata, it follows that $\Wpe$ is non-empty for all admissible $\pe$.  The computations above show that $\Wpe$ is
pure-dimensional and $\dim (\Wpe) =2g - (\sharp\; \varphi+ \sharp\; \eta)$ as well.

We know that $\Zpe$ is closed and contains $\Wpe$, hence
$\overline{W}_{\varphi, \eta}$. Dimension considerations imply that
$\overline{W}_{\varphi, \eta}$ must be a union of irreducible
components of $\Zpe$. If $\overline{W}_{\varphi, \eta} \neq \Zpe$,
then the remaining components of $\Zpe$ are contained in
$\bigcup_{(\varphi', \eta')\gneqq \pe}Z_{(\varphi', \eta')}$, which
is not possible by dimension considerations.

It remains only to prove assertion (3). First note that, by
admissibility, $\dim(\Zpe)=g$ exactly when $\eta=\ell(\varphi)^c$.
Let $C$ be an irreducible component of $\Ybar$. Since $C$ is
contained in the union of all $g$-dimensional closed strata, it must
be contained in a single one, i.e., $C\subseteq\Zpe$ for some
$\pe$. Therefore, $C$ must be an irreducible component of $\Zpe$.
Conversely, every irreducible component of
$Z_{\varphi,\ell(\varphi)^c}$  is $g$-dimensional, and hence an
irreducible component of $\Ybar$. In particular, $\Ybar$ is of pure
dimension $g$.
\end{proof}

\begin{cor}
The singular locus $\Ybar^{\rm sing}$ of $\Ybar$ has the following
description.
\[\Ybar^{\rm sing} = \Ybar - \bigcup_{\varphi \subset \BB} W_{\varphi, \ell(\varphi^c)} = \bigcup_{\tiny{\begin{matrix}
\pe\\
\ell(\varphi) \cap \eta \neq \emptyset
\end{matrix}}}\Wpe. \]
If $\Qbar \in \Ybar^{\rm sing} $ then there are $2^{\sharp\;
I(\Qbar)}$ irreducible components passing through $\Qbar$, all
$g$-dimensional.
\end{cor}

\

\begin{dfn}\label{definition: horizontal
strata}  Recall that $\mathbb{B}_\mathfrak{p}= \{\beta\in\mathbb{B}:
\beta^{-1}(pW(\kappa)) = \mathfrak{p}\}$. Let $\gert$ be an ideal of $\ol$ dividing $p$.  Let $\gert^\ast = p/\gert$ (so $\gert \gert^\ast =
p\ol$). Let $\BB_\gert = \cup_{\gerp \vert \gert} \BB_\gerp$, and
let $f(\gert) = \sum_{\gerp\vert \gert} f(\gerp/p)$ be the sum of
the residue degrees.
\end{dfn}

The following proposition is clear from definitions.
\begin{prop}\label{proposition: ordinary strata} We
have
\begin{enumerate}
\item $\Ybar_F=Z_{\BB,\emptyset}$, and $\Ybar_V=Z_{\emptyset,\BB}$.
\item $\Xbar^{\rm ord}=W_\emptyset$.
\item $\Ybar^{\rm ord}=\cup_{\gert\vert p} W_{\BB_\gert,\BB_{\gert^\ast}}$.
\item $\Ybar^{\rm ord}_F= W_{\BB,\emptyset}$; $\Ybar^{\rm ord}_V= W_{\emptyset,\BB}$.
\end{enumerate}
\end{prop}

\

\subsection{The fibres of $\pi\colon \Ybar \arr \Xbar$}

\

\id {\bf A certain Grassmann variety.} Fix a closed $k$-rational
point $\Pbar$ of $\Xbar$ corresponding to $\uA$, where $k$ is
an algebraically closed field. Let $\DD = \DD(\uA[p]) =
\oplus_{\beta\in \BB} \DD_\beta$; each $\DD_\beta$ is a
$2$-dimensional vector space over $k$ on which $\ol$ acts via
$\beta$.  Recall from \S\ref{subsection:Facts about D modules} that
the kernel of Frobenius and the Kernel of Verschiebung, two
Dieudon\'e submodules of $\DD$, decompose as
\[\Ker(\Fr) = \oplus_{\beta\in \BB} \Ker (\Fr)_\beta, \qquad \Ker (\Ver) = \oplus_{\beta\in \BB}
\Ker (\Ver)_\beta,\] where each $\Ker (\Fr)_\beta, \Ker (\Ver)
_\beta$ is a one dimensional subspace of $\DD_\beta$. By Lemma \ref{lemma: type in terms of Dieudonne modules}, we have
\[\beta \in \tau(\uA) \Longleftrightarrow \Ker (\Fr)_\beta = \Ker
(\Ver)_\beta.\] Consider the variety $\scrG = \scrG(\Pbar)$
parameterizing subspaces $\HH = \oplus \HH_\beta$ of $\DD$
satisfying the conditions:
\begin{itemize}
\item $\HH_\beta \subset \DD_\beta$ is 1-dimensional.
\item $\Fr(\HH(\beta))\subseteq \HH_{\sigma \circ \beta}$.
\item $\Ver(\HH(\beta))\subseteq \HH_{\sigma^{-1} \circ \beta}$.
\end{itemize}
We give $\scrG$ the scheme structure of a closed reduced subscheme
of $(\PP^1_k)^g$. It is a generalized Grassmann variety.

Define a morphism
\[ g\colon \pi^{-1}(\Pbar)_{\rm red} \Arr \scrG,\]
as follows. We use the identification $\DD = H^1_{\rm dR}(\uA,
\calO_{\uA})$.  The universal family $(f\colon \uA^{\rm univ} \arr
\uB^{\rm univ})$ over the reduced fibre $\pi^{-1}(\Pbar)_{\rm red}$
produces a sub-vector bundle of $\DD \times \pi^{-1}(\Pbar)_{\rm red}$ by considering
$f^\ast \HH^1_{\rm dR}(\uB, \calO_{\uB})$, which point-wise is
$f^\ast \HH^1_{\rm dR}(\uB_x, \calO_{\uB_x})=\DD(f)(\DD(\uB_x[p]))$
($x\in\pi^{-1}(\Pbar)_{\rm red}$),  and so is a subspace of the kind
parameterized by $\scrG$. By the universal property of Grassmann
variety $(\PP^1_k)^g = {\rm Grass}(1, 2)^g$, we get a morphism
$g\colon : \pi^{-1}(\Pbar)_{\rm red} \arr (\PP^1_k)^g$ that factors
through $\scrG$, because it does so at every closed point of
$\pi^{-1}(\Pbar)_{\rm red}$. We note that for every $x$ as above
$\DD/f^\ast \HH^1_{\rm dR}(\uB_x, \calO_{\uB_x}) = \DD(\Ker(f_x))$
and so it is clear that $g$ is injective on geometric points and in
fact, by the theory of Dieudonn\'e modules, bijective. We have
therefore constructed a bijective morphism
\[ g\colon \pi^{-1}(\Pbar)_{\rm red} \Arr \scrG.\]
Since $g$ is a morphism between projective varieties, it is closed
and hence it is a homeomorphism. We will use this in what follows.

  It will be convenient for us to think of the fibre
$\pi^{-1}(\Pbar)_{\rm red} $ entirely in terms of $\scrG$. To this
end, we provide some definitions. For $\HH \subset \DD$ as above,
define
\[ \varphi(\HH) = \{ \beta \in \BB: \HH_\beta = \Ker(\Ver)_\beta\},\]
and
\[ \eta(\HH) = \{ \beta \in \BB: \HH_\beta = \Ker(\Fr)_\beta\}.\]
\begin{lem}
Let $H \subset \uA[p]$ be a subgroup scheme such that $(\uA, H) \in
\pi^{-1}(\Pbar)$. Let $f:A\arr A/H$ be the canonical map. Let $\HH = \Ima [\DD(f)]
= \Ker[\DD(\uA[p]) \arr \DD(H)]$. Then,
\[ \varphi(\uA, H) = \varphi(\HH), \qquad \eta(\uA, H) = \eta(\HH).\]
\end{lem}
\begin{proof}
By Lemma \ref{lemma: type in terms of Dieudonne modules}, we have
\begin{align*}
\varphi(\uA, H) & = \{ \beta \in \BB: \Ima(\DD(f))_\beta= \Ima (\DD(\Fr_A))_\beta\} \\
& = \{\beta \in \BB: \HH_\beta = (\Ima(\Fr))_\beta\}\\
& = \{\beta \in \BB: \HH_\beta = \Ker(\Ver)_\beta\}\\
& =\varphi(\HH).
\end{align*}
The argument for $\eta(\uA,H)$ is similar.
 \end{proof}

We can now study the induced stratification on $\pi^{-1}(\Pbar)_{\rm
red} $ by means of $\scrG$.
\begin{cor}\label{corollary: dimension of fibre intersect strata}
$\pi^{-1}(\Pbar)_{\rm red} \cap \Wpe$ is homeomorphic to the locally
closed subset of\ $\scrG$ parameterizing subspaces $\HH$ with
$\varphi(\HH) =\varphi, \eta (\HH) = \eta$. Its dimension is thus at
most $g - \sharp \; (\varphi\cup \eta)$.
\end{cor}

\begin{rmk} We will see (Corollary  \ref{cor: dimension of strata on fibers; precise}) that the equality holds if the fibre is non-empty.
\end{rmk}

\

\id {\bf The relation between the stratifications on $\Ybar$ and
$\Xbar$.}
\begin{thm}\label{theorem: pi of Zpe}
\begin{enumerate}
\item Let $C$ be a component of $\Zpe$; then $\pi(C)$ is a component
of $Z_{\varphi\cap \eta}$.
\item $\pi(\Zpe) = Z_{\varphi\cap \eta}$.
\item On every component of $\Zpe$ (or $\Wpe$) the type is
generically $\varphi\cap \eta$.
\end{enumerate}
\end{thm}

\begin{rmk}
The type is not necessarily constant on $\Wpe$. In fact,
\[\pi(\Wpe) = \bigcup_{\tiny{\begin{matrix}
\left[\varphi(\uA, H) \vartriangle \eta(\uA, H) \right]^c \supseteq
\tau'\\ \tau' \supseteq \varphi\cap \eta
\end{matrix}}} W_{\tau'}.\]
See Proposition~\ref{proposition: star} below.
\end{rmk}

\begin{proof}
For every point $y \in \pi(C)$, $\dim (\pi^{-1}(y) \cap C) \leq g -
\sharp\; (\varphi\cup \eta)$, by Corollary~\ref{corollary: dimension
of fibre intersect strata}. Therefore,
\begin{align*}
\dim (\pi(C)) & \geq \dim(C) - (g - \sharp\; (\varphi\cup \eta)) \\
& = 2g - (\sharp\; \varphi + \sharp\; \eta) - (g - \sharp\;
(\varphi\cup \eta)) \\ & = g - \sharp\; (\varphi\cap \eta).
\end{align*}
On the other hand, since $\tau(y) \supseteq \varphi \cap \eta$, we
have $\pi(C) \subseteq Z_{\varphi\cap\eta}$. Moreover, $\dim(Z_{\varphi\cap\eta})= g - \sharp\;
(\varphi\cap \eta)$. Since $\pi$ is proper,  $\pi(C)$   is closed and irreducible. By comparing the
dimensions, we conclude that $\pi(C)$ is an irreducible component of
$Z_{\varphi\cap \eta}$. This is part (1) of the theorem.

\

\id We now prove part (2) of the theorem. Let $C \subseteq
Z_{\varphi \cap \eta}$ be an irreducible component, and let $\Pbar$
be a closed point of $C$ corresponding to $\uA$. We want to prove that $\pi^{-1}(\Pbar)
\cap \Zpe \neq \emptyset$. We provide two proofs for that.

\medskip

\id $\bullet$ \emph{A ``pure thought" argument}. $\pi^{-1}(\Pbar)
\cap \Zpe$ depends entirely on $\DD(\uA[p])$, which is determined by
the type \cite[Theorem 3.8]{GorenOort}. It is a consequence of part
(1) that there are other points $\Pbar'$ of the same type as $\Pbar$
lying in the image of $\Zpe$. Therefore, $\pi^{-1}(\Pbar) \cap \Zpe
\neq \emptyset$.

\

\id The second method is more instructive and will be used again in
the sequel.

\medskip

\id $\bullet$ \emph{An explicit construction}.  We construct a
Dieudonn\'e submodule $\HH \subseteq \DD(\uA[p])$ subject to
the conditions:
\begin{align*}
\forall \beta \in \varphi: & \qquad \HH_\beta = \Ker(\Ver)_\beta,
\\
\forall \beta \in \eta: & \qquad \HH_\beta = \Ker(\Fr)_\beta.
\end{align*}
\begin{lem}\label{lemma: constructing pi -1 P cap Wpe}
For any choice of $\HH_\beta$ for $\beta \in (\varphi \cup \eta)^c$,
$\HH = \oplus_{\beta\in \BB} \HH_\beta$ satisfies the required
conditions: namely, it is a Dieudonn\'e submodule, corresponding to
a subgroup $H \subset \uA[p]$ with invariants $(\varphi',
\eta') \geq \pe$.
\end{lem}
\begin{proof}{(Of Lemma)}
First note that the formulas $\ell(\varphi) \cup \eta = \varphi \cup
r(\eta) = \BB$ give
\[ \eps \in (\varphi \cup\eta)^c \Longrightarrow \begin{cases}
\sigma \circ \eps \in \varphi\\ \sigma^{-1} \circ \eps \in \eta.
\end{cases}\]
Therefore,
\[ \HH_{\sigma \circ \eps} = \Ker(\Ver)_{\sigma \circ \eps}, \qquad \HH_{\sigma^{-1} \circ \eps} = \Ker(\Fr)_{\sigma^{-1} \circ \eps}.\]
Since we always have $\Fr(\HH_\eps) \subseteq \Ker(\Ver)_{\sigma
\circ \eps}, \Ver(\HH_\eps) \subseteq \Ker(\Fr)_{\sigma^{-1} \circ
\eps}$, we conclude that
\[\Fr(\HH_\eps) \subseteq \HH_{\sigma \circ \eps}, \quad \Ver(\HH_\eps) \subseteq \HH_{\sigma^{-1} \circ \eps}, \qquad \eps \in (\varphi \cup\eta)^c.\]
Suppose now that $\eps \in \varphi$. Then $\HH_\eps =
\Ker(\Ver)_\eps$ and so $V(\HH_\eps) = 0 \subset \HH_{\sigma^{-1}
\circ \eps}$. On the other hand, we have that $\Fr(\HH_\eps) \subseteq \HH_{\sigma\circ
\eps}$ if $\eps \in \varphi\cap \eta$ (because $\Fr(\HH_\eps) = 0$),
or if $\sigma\circ \eps \in \varphi$ (because then $\Fr(\HH_\eps)
\subseteq \Ker(\Ver)_{\sigma \circ\eps} = \HH_{\sigma \circ\eps}$).
Therefore, it is enough to show that we cannot have $\eps\in \varphi,
\eps \not\in \varphi\cap \eta, $ and $\sigma\circ \eps \not\in
\varphi$. This follows readily from $\ell(\varphi^c) \subseteq
\eta$.

The argument for $\eps \in \eta$ is entirely similar and hence
omitted. The claim concerning the invariants is clear from the
construction.
\end{proof}

The existence of a  bijective morphism  $g\colon
\pi^{-1}(\Pbar)_{\rm red} \Arr \scrG$,  implies that
$\pi^{-1}(\Pbar)\cap \Zpe\neq\emptyset$, and hence part (2) follows.
Finally, part (3) follows immediately from part (1).
\end{proof}

\id We remark that since $g$ is a homeomorphism, the second method
of the proof gives the
 following interesting result:
 \begin{cor}\label{cor: dimension of strata on fibers; precise} Let $\Pbar$ be a closed point of $\Xbar$ such that $\tau(\Pbar) \supseteq \varphi \cap \eta$.
Then,
\[   \dim(\pi^{-1}(\Pbar) \cap
\Zpe) =\dim(\pi^{-1}(\Pbar) \cap
\Wpe)=g - \sharp\; (\varphi \cup \eta).\]
 \end{cor}
\begin{cor}\label{corollary: on every component there's a supsp pt}
On every irreducible component of $\Zpe$ there is a point $\Qbar$
such that $\pi(\Qbar)$ is superspecial.
\end{cor}
\begin{proof}
Let $C \subseteq \Zpe$ be an irreducible component. Then $\pi(C)$ is
an irreducible component of $Z_{\varphi\cap\eta}$, and hence
contains a superspecial point; in the case $p$ is inert this follows
from $W_\tau$ being quasi-affine for $\tau\neq \BB$
\cite[Proposition 2.19]{GorenOort}. In the general case, see \cite[\S
1]{G}. In both cases note that $W_\tau$ is denoted there $W_\tau^0$
and $Z_\tau$ is denoted there $W_\tau$.
\end{proof}

\begin{dfn}A closed stratum $\Zpe$ is called {\emph{horizontal}} if $\pi:\Zpe\arr \Xbar$ is finite and surjective (equivalently, dominant and quasi-finite).
\end{dfn}

\begin{cor}  The horizontal strata are exactly the
$Z_{\BB_\gert,\BB_{\gert^\ast}}$, for all $\gert\vert p$. (See Definition \ref{definition: horizontal
strata}).\end{cor}

\begin{rmk}  In the Appendix, we prove that the morphism $\pi\colon Z_{\BB_\gert,\BB_{\gert^\ast}}\arr \Xbar$ is a finite-flat and purely inseparable morphism of degree $p^{f(\gert)}$.
\end{rmk}

\id Our next goal is a detailed study of the fibre $\pi^{-1}(\Pbar)_{\rm red}$,
where $\Pbar$ is a superspecial point; it is used in the proof of
Theorem~\ref{theorem: C cap YF cap YV} below.

Let $\Pbar$ be a superspecial point of $\Xbar$, corresponding to $\uA$ defined over a
perfect field $k$. Let $\DD = \DD(\uA[p])$. Let $S \subset
\BB$ be a \emph{spaced} subset: that is, $\beta\in
S\Rightarrow \sigma \circ \beta \not\in S$. We define
\[ \calF_S \subseteq \pi^{-1}(\Pbar)_{\rm red},\]
to be the closed subset of $\pi^{-1}(\Pbar)_{\rm red}$ whose geometric points
are $\Qbar = (\uA, H)$ such that, letting $\HH = \Ker[\DD(\uA[p]) \arr \DD(H)]=\DD(\uA[p]/H)$,
\[ \HH_\beta = \Ker(\Fr)_\beta \; (= \Ker(\Ver)_\beta), \qquad \forall \beta \not\in S.\]
Let $\calF_S^\circ$ denote the open subset of $\calF_S$ where
$\HH_\beta \neq \Ker(\Fr)_\beta, \forall \beta \in S$.

\begin{lem}\label{lem: fibre over a supsp pt} Let $S\subset\BB$ be a spaced subset.
\begin{enumerate}
\item $\calF_S$ and $\calF_S^\circ$ are irreducible of dimension
$\sharp\;S$. In fact, there are geometrically bijective,
finite morphisms $\calF_S \arr (\PP_k^1)^{\sharp\; S}$ and
$\calF_S^\circ \arr (\AA_k^1)^{\sharp\; S}$.
\item $\calF_{S_1} \cap \calF_{S_2} = \calF_{S_1\cap S_2}$.
\item The collection of locally closed sets $\{\calF^\circ_{S}: S \subseteq \BB \; {\rm
spaced}\}$ forms a stratification of $\pi^{-1}(\Pbar)_{\rm red}$ by
irreducible locally closed subsets. In fact,
\[\calF_S^\circ = W_{S^c, S^c} \cap \pi^{-1}(\Pbar)_{\rm red}.\]
\end{enumerate}
\end{lem}
\begin{proof} The proof is essentially a series of simple
observations.  First, $\calF_S$ is homeomorphic via $g\colon
\pi^{-1}(\Pbar)_{\rm red} \Arr \scrG$ to the closed subset of
$\scrG(\Pbar)$ where $\HH_\beta = \Ker(\Fr)_\beta = \Ker
(\Ver)_\beta$ for all $\beta\not\in S$, and no conditions are
imposed at $\beta \in S$ (cf. the proof of Lemma~\ref{lemma:
constructing pi -1 P cap Wpe}). This subset of $\scrG$, viewed with
the reduced induced structure, is clearly isomorphic to
$(\PP_k^1)^{\sharp\; S}$. Similarly, $\calF_S^\circ$ is mapped via
$g$ to a closed subscheme of $\scrG$ isomorphic  to
$(\AA_k^1)^{\sharp\; S}$.

The second part of the lemma is immediate from the definitions (the intersection is considered set-theoretically).

As for the last part, first note that $\overline{\calF_S^\circ} =
\calF_S$; indeed, considered on $\scrG$ this just says that
$(\AA_k^1)^{\sharp\; S}$ in dense in $(\PP_k^1)^{\sharp\; S}$. Next,
from the definitions we have
\[ \calF_S = \textstyle\coprod_{T\subseteq S}\calF_T^\circ, \]
and
\[ \calF_S^\circ = W_{S^c, S^c} \cap \pi^{-1}(\Pbar)_{\rm red}.\]
\end{proof}

\begin{thm}\label{theorem: C cap YF cap YV}
Let $C$ be an irreducible component of $\Zpe$. Then
\[C \cap \Ybar_F \cap \Ybar_V \neq
\emptyset.\]
\end{thm}
\begin{rmk}
Note that $\Ybar_F \cap \Ybar_V$ consists of points $(\uA, H)$,
where $H$ is both the kernel of Frobenius and the kernel of
Verschiebung. Such an abelian variety $A$ is superspecial, and the
points $\Ybar_F \cap \Ybar_V$ are thus the finitely many points
$(\Fr_A: \uA \arr \uA^{(p)})$, where $\uA$ ranges over the finitely
many superspecial points in $\Xbar$. Those finitely many points
``hold together" all the components of $\Ybar$.
\end{rmk}
\begin{proof}
Let $S = (\varphi\cup \eta)^c$. Then $S$ is spaced: If $\beta \in S$
then $\beta \in \varphi^c \cap \eta^c$ and so $\sigma\circ \beta \in
r(\eta^c) \subseteq \varphi$ and so $\sigma\circ \beta \not\in
\varphi^c$, hence $\sigma\circ \beta \not\in S$.

Let $\Qbar\in C$ be a closed point such that $\Pbar = \pi(\Qbar)$ is
superspecial (Corollary~\ref{corollary: on every component there's a
supsp pt}). Consider $\calF_S^\circ \subseteq \pi^{-1}(\Pbar)$.
Recall from Lemma~\ref{lem: fibre over a supsp pt} that
$\calF_S^\circ = W_{S^c, S^c} \cap \pi^{-1}(\Pbar)$, and we have
$\varphi \subseteq S^c, \eta \subseteq S^c$. Any $x\in \calF_S^\circ$ belongs to $\Zpe$, and since $\Zpe$ is nonsingular (Theorem~\ref{theorem: fund'l facts about the stratificaiton of Ybar}), there exists a \emph{unique} component of $\Zpe$ passing through $x$; we call that
component $C_x$. We distinguish two cases.

\

\id \emph{Case 1: $\dim(\calF_S^\circ) \geq 1$.} This case is
equivalent to $\varphi \cup \eta \neq \BB$, by Lemma~\ref{lem: fibre
over a supsp pt}.  By assumption, $\calF_S^\circ$ has infinitely
many points, while $\Zpe$ has only finitely many components.
Therefore, there exists a component $C'$ of $\Zpe$ such that $C'
\cap \calF_S^\circ$ is dense in $\calF_S^\circ$ and, thus, $C'
\supseteq \calF_S$. But then, since $\Zpe$ is nonsingular,   no
other component of $\Zpe$ intersects $\calF_S$.

By assumption $\Qbar \in \Zpe$ is so that $\pi(\Qbar)$ is superspecial. Hence, by Lemma \ref{lemma: type in terms of Dieudonne modules}, we have $\varphi(\Qbar) = \eta(\Qbar)$.  Now,
$(\varphi(\Qbar), \eta(\Qbar))\geq \pe$ gives that $\varphi(\Qbar)
\supseteq \varphi \cup \eta = S^c$ and $\eta(\Qbar)\supseteq \varphi
\cup \eta = S^c$. Because
\[ \calF_S = \{ x\in \pi^{-1}(\Pbar): (\varphi(x), \eta(x))\geq (S^c, S^c)\},\]
we conclude that $\Qbar \in \calF_S$. It follows that $C' = C$, and
since $\calF_S \cap \Ybar_F\cap \Ybar _V \neq \emptyset$, our proof
is complete in this case.

\

\id \emph{Case 2: $\dim(\calF_S^\circ) =0$.} This case is equivalent
to $\varphi \cup \eta = \BB$, and so $S = \emptyset$. In this case
$\calF_S = \calF_S^\circ = \Ybar_F \cap \Ybar_V \cap
\pi^{-1}(\Pbar)$. As above, $\Qbar \in \calF_S$ and so $\Qbar \in
\Ybar_F \cap \Ybar_V \cap \pi^{-1}(\Pbar)$.

\

\id This complete the proof.
\end{proof}
\id The proof reveals an interesting fact:
\begin{cor}
Let $\pe$ be an admissible pair, and $S = (\varphi \cup \eta)^c$. Let
$\Pbar$ be a superspecial point.  There is a unique component $C$ of
$\Zpe$ that intersects $\pi^{-1}(\Pbar)$. We have $C \cap
\pi^{-1}(\Pbar) = \calF_S$. Moreover, any superspecial point $\Pbar$
belongs to $\pi(\Zpe)$.
\end{cor}

\id We end this section by refining our knowledge on the
relationship between the strata on $\Ybar$ and on $\Xbar$.

\begin{prop}\label{proposition: star} Let $\pe$ be an admissible
pair. Then,
\[\pi(\Wpe) = \bigcup_{\tiny{\begin{matrix}
\left[\varphi(\uA, H) \vartriangle \eta(\uA, H) \right]^c \supseteq
\tau'\\ \tau' \supseteq \varphi\cap \eta
\end{matrix}}} W_{\tau'}.\]
Furthermore, each fibre of $\pi: \Wpe \arr \Xbar$ is affine
and irreducible of dimension $g - \sharp\; (\varphi\cup \eta)$.
\end{prop}
\begin{proof} The proof is similar to the second proof appearing in
Theorem~\ref{theorem: pi of Zpe}. We know by
Corollary~\ref{corollary: type and phi eta} that $\pi(\Wpe)
\subseteq \bigcup_{\tau^\prime} W_{\tau'}$, where $\tau^\prime$ is as above. Given a point $\Pbar$ in $W_{\tau'}$,
where $\left[\varphi(\uA, H) \vartriangle \eta(\uA, H) \right]^c
\supseteq \tau'\supseteq \varphi\cap \eta$, consider the reduced
fibre $\pi^{-1}(\Pbar)_{\rm red}$. We want to construct a point $\HH$ of the
Grassmann variety $\scrG(\Pbar)$, such that $\HH_\beta$ is
a one dimensional subspace of $\DD_\beta$, and
\[\HH_\beta = \begin{cases}
\Ker(\Ver)_\beta & \beta \in \varphi, \\
\Ker(\Fr)_\beta & \beta \in \eta, \\
\not \in \{\Ker(\Fr)_\beta, \Ker(\Ver)_\beta \} & \beta \not\in
\varphi \cup \eta.
\end{cases}\]
We need to check that $\HH$ thus defined is indeed a point of
$\scrG(\Pbar)$, which amounts to being stable under the maps $\Fr$ and
$\Ver$. This is a straightforward calculation:
\begin{itemize}
\item $\Fr \;\HH_\beta \subseteq \HH_{\sigma \circ \beta}$.

\id We distinguish cases: \begin{enumerate} \item
 $\beta \in \eta$. Then
this is clear as $\Fr \;\HH_\beta  = \{0\}$.

\item $\beta \not\in \eta$. In this case, since
$\beta \in \ell(\varphi) \cup \eta$, it follows that $\beta \in
\ell(\varphi)$, that is, $\sigma\circ \beta \in \varphi$. Then, $\Fr\; \HH_\beta \subseteq
\Ker(\Ver)_{\sigma \circ \beta} = \HH_{\sigma\circ \beta}$.
\end{enumerate}
\item $V \;\HH_\beta \subseteq \HH_{\sigma^{-1}\circ \beta}$.

\id The argument is entirely similar, where one distinguishes the cases:
(i) $\beta \in \varphi$; (ii) $\beta \not\in \varphi$, which implies $\sigma^{-1} \circ \beta
\in \eta$.
\end{itemize}
We now show that
\[ \varphi(\HH) = \varphi, \qquad \eta(\HH) = \eta.\]
Clearly, $\varphi(\HH) \supseteq \varphi$ and $\eta(\HH) \supseteq
\eta$. Note the following: (i) If $\beta \in (\varphi\cup \eta)^c$,
then by definition $\HH_\beta \neq \Ker(\Ver)_\beta$, and so $\beta
\not\in \varphi(\HH)$; (ii) If $\beta \in \eta - \varphi$, then
$\HH_\beta = \Ker(\Fr)_\beta$. On the other hand, by assumption, $\tau' \cap (\eta
- \varphi) = \emptyset$ and so $\Ker(\Fr)_\beta \neq \Ker(\Ver)_\beta$ and so
$\HH_\beta \neq \Ker(\Ver)_\beta$. That is, $\beta \not\in
\varphi(\HH)$. Put together, these facts show that $\varphi(\HH) =
\varphi$. A similar argument gives $\eta(\HH) = \eta$.

\

\id It follows from these considerations that $\pi^{-1}(\Pbar) \cap
W_{\varphi, \eta} \neq 0$. The above calculations show that if
$\Pbar$ is $k$-rational, then $\pi^{-1}(\Pbar) \cap W_{\varphi,
\eta}$ maps under $g$ to a closed subscheme of $\scrG$ isomorphic to
$\AA^{g - \sharp\; (\varphi\cup \eta)}_k$. This proves the second
claim of the proposition (a finite morphism is affine).
\end{proof}

\

\subsection{The Atkin-Lehner automorphism} The Atkin-Lehner
automorphism on $Y$ is the morphism
\[ w\colon Y \arr Y,\]
characterized by its action on points:
\[ w(\uA, H) = (\uA/H, A[p]/H),\]
equivalently,
\[ w\left((\uA,\alpha_A) \overset{f}{\Arr} (\uB, \alpha_B)\right) =
\left((\uB,\alpha_B) \overset{f^t}{\Arr} (\uA, p\circ
\alpha_A)\right).\] Thus,
\[ w^2 = \langle p\rangle,\]
where $\langle p\rangle$ is the diamond operator whose effect on
points is
\[ (\uA, \alpha_A, H) \mapsto (\uA, p \circ \alpha_A, H).\]
It follows that a power of $w^2$ is the identity and so $w$ is an
automorphism of $Y$.

\begin{prop}\label{proposition: w and phi eta}
The Atkin-Lehner automorphism acts on the stratification of $\Ybar$
by
\[
w(\Wpe) = W_{r(\eta), \ell(\varphi)}, \qquad w(\Zpe) = Z_{r(\eta), \ell(\varphi)}.
\]
In particular, we have $w(Z_{\BB_\gert,\BB_{\gert^\ast}})=Z_{  \BB_{\gert^\ast},\BB_\gert   }$.
\end{prop}
\begin{proof}
 Suppose that $(f\colon \uA \arr \uB)$ is parameterized by a closed
 point $\Qbar$ in $\Wpe$. Then
 \begin{align*}
 \varphi(\Qbar) & = \{\beta \in B: \Lie(f)_{\sigma^{-1} \circ \beta} = 0\}, \\
 \eta (\Qbar) & = \{\beta \in B: \Lie(f^t)_\beta = 0\}.
 \end{align*}
 As $w(f\colon (\uA, \alpha_A) \arr (\uB, \alpha_B)) = (f^t\colon (\uB, \alpha_B) \arr (\uA,
 p\circ\alpha_A))$,
\begin{align*}
 \varphi(w(\Qbar)) & = \{\beta \in B: \Lie(f^t)_{\sigma^{-1} \circ \beta} = 0\}
 = r(\eta(\Qbar)),\\
 \eta (\Qbar) & = \{\beta \in B: \Lie(f)_\beta = 0\} = \ell(\varphi(\Qbar)).
 \end{align*}
\end{proof}

\begin{lem}\label{lemma: w and local parameters}
For the choice of parameters in the
uniformization~(\ref{equation: local deformation ring at Q bar}) at $\Qbar$ and
$w(\Qbar)$, the homomorphism
\[w^\ast\colon \widehat{\calO}_{\Ybar, w(\Qbar)}\Arr \widehat{\calO}_{\Ybar, \Qbar},\]
is given by
\[
w^\ast(x_{\beta,w(\Qbar)}) = y_{\beta,\Qbar}, \qquad w^\ast(y_{\beta,w(\Qbar)})  = x_{\beta,\Qbar}, \qquad  w^\ast(z_{\gamma,w(\Qbar)})  = z_{\gamma,\Qbar},
\]
for $\beta \in I(\Qbar)$, and $\gamma \not\in I(\Qbar)$.
\end{lem}
\begin{proof}
We refer to the discussion before Theorem \ref{thm: Stamm's theorem} in the following. In terms of the deformation theory discussed there, the morphism $w$ switches the role of $W_A^R$ and $W_B^R$ in a manner compatible with the $\ol$-action. Therefore, it is clear that in the first formulation of Theorem \ref{thm: Stamm's theorem} we have $w^*(x_{\beta,w(\Qbar)})=y_{\beta,\Qbar}$ for $\beta \in \ell(\varphi(w(\Qbar)))=\eta(\Qbar)$, and $w^*(y_{\beta,w(\Qbar)})=x_{\beta,\Qbar}$ for $\beta \in \eta(w(\Qbar))=\ell(\varphi(\Qbar))$. The result now follows, since in the uniformization  (\ref{equation: local deformation ring at Q bar}) the parameters $x_{\gamma,\Qbar}$ for $\gamma\in\ell(\varphi(\Qbar))-I(\Qbar)$, and $y_{\gamma,\Qbar}$ for $\gamma\in\eta(\Qbar)-I(\Qbar)$ have been replaced with $z_{\gamma,\Qbar}$, and a similar reassignment has taken place at $w(\Qbar)$.
\end{proof}

\id {\bf Caveat.} \emph{It is difficult for a point to $\Qbar$ to equal $w(\Qbar)$; for that to happen, one must have, among other things, that $\Qbar\in \Ybar_F \cap \Ybar_V$ and $p\equiv 1 \pmod{N}$. Nonetheless, if $\Qbar=w(\Qbar)$, then it must be understood that in Lemma \ref{lemma: w and local parameters}, we consider two possibly different sets of parameters at $\Qbar$ and $w(\Qbar)$ despite the fact that our notation does not reflect that.}

\

We will use the following lemma in the sequel. Let $\Qbar \in \Ybar_F$, and $\Pbar=\pi(\Qbar)$. Then $s(\Pbar)=\Qbar$.

\begin{lem}\label{lemma: section and pi}
For the choice of parameters in the
uniformization~(\ref{equation: local deformation ring at Q bar}) at
$\Qbar$, there is a choice of   parameters as in (\ref{equation: local def ring at P bar}) at $\Pbar$ such that the homomorphism
\[s^\ast\colon \widehat{\calO}_{\Ybar, \Qbar}\Arr \widehat{\calO}_{\Xbar, \Pbar},\]
is given by
\[
s^\ast(x_{\beta}) = t_{\beta}, \qquad s^\ast(y_{\beta})  = 0, \qquad  s^\ast(z_{\gamma})  = t_{\gamma},
\]
for $\beta \in I(\Qbar)$, and $\gamma \not\in I(\Qbar)$.
\end{lem}
\begin{proof} Let $\Pbar$ correspond to $\uA$; then $\Qbar$ corresponds to $(\Fr:\uA \arr \uA^{(p)})$, and we have $\varphi(\Qbar)=\BB$ and $I(\Qbar)=\eta(\Qbar)$. Fix an index $\beta\in \BB$, and let $W_1,W_2,M,R$ be the data corresponding to $\Qbar$ ``at the $\beta$-component," as in \S \ref{subsection: infinitesimal nature of Ybar}. Applying deformation theory in a similar way, one can provide an isomorphism as in (\ref{equation: local def ring at P bar}) at $\Pbar$ which has the desired properties:  the relevant Grassmann problem is to provide $W_1^R$ inside $R^2$ lifting $W_1$, and the parameter $t_\beta$ describes the deformation of $W_1$. Since, in the notation of the first formulation of Theorem \ref{thm: Stamm's theorem}, the parameter  $x_\beta$  describes the deformation of $W_1$, we can choose $t_\beta$ such that $s^*(x_\beta)=t_\beta$. Note that in  the isomorphism (\ref{equation: local deformation ring at Q bar})  $x_\beta$ is renamed to $z_\beta$ if $\beta\in I(\Qbar)^c$, and that is how we have recorded this in the statement of the lemma. Finally, since $\Fr^*(\omega_{A^{(p)}})=0$,  any deformation of $W_2$ is constant, equal to $\Ker(M)$, and hence $s^*(y_\beta)=0$ whenever $y_\beta$ is defined, that is,  for $\beta\in\eta(\Qbar)=I(\Qbar)$.
\end{proof}

\

\subsection{The infinitesimal nature of $\pi\colon \Ybar \arr \Xbar$}
Let $k$ be a finite field containing $\kappa$, and $\Qbar$ a closed
point of $\Ybar$ with residue field $k$. Let $\Pbar = \pi (\Qbar)$; let $\varphi = \varphi(\Qbar), \eta = \eta(\Qbar), I = I(\Qbar)$, and
$\tau = \tau(\Pbar)$. Choose the following isomorphisms
\begin{align}
 \widehat{\calO}_{\Ybar, \Qbar} &\cong k [\![ \{x_\beta,
 y_\beta: \beta \in I\}, \{z_\beta: \beta \in I^c\}]\!]/(\{x_\beta y_\beta: \beta\in
 I\}),\label{equation: stamm isomorphism}\\
 \widehat{\calO}_{\Xbar, \Pbar} &\cong
k[\![t_\beta: \beta \in \BB ]\!], \label{equation: GO isomorphism}
\end{align}
as explained in \S\S \ref{subsection: infinitesimal nature of Ybar}-\ref{subsection: stratification of Y bar}.
The irreducible components of $\Ybar$ through $\Qbar$ are in
bijection with subsets $J \subseteq I$. To such $J$, we have associated the
ideal
\[ \calI_J = \langle \{x_\beta: \beta \not\in J\}, \{y_\beta; \beta \in J\} \rangle\]
in $\widehat{\calO}_{\Ybar, \Qbar}$. By Theorem~\ref{theorem: fund'l
facts about the stratificaiton of Ybar}, the closed set $V(\calI_J)$
in $\Spf\; \widehat{\calO}_{\Ybar, \Qbar}$ corresponding to it is
$Z^{\wedge\Qbar}_{r(\eta - J)^c, \eta - J}$, the formal completion of $Z_{r(\eta - J)^c, \eta - J}$ at $\Qbar$.

\

\id The following lemma, despite appearances, plays a key role in
obtaining the results concerning the canonical subgroup. We shall
refer to it in the sequel as ``Key Lemma".
\begin{lem}[Key Lemma]\label{lemma: key lemma }
Let $\beta\in \varphi\cap \eta$ and $\pi^\ast\colon
\widehat{\calO}_{\Xbar, \Pbar} \arr \widehat{\calO}_{\Ybar, \Qbar}$
the induced ring homomorphism.

\begin{enumerate}
\item \fbox{$\sigma \circ \beta \in \varphi, \sigma^{-1}\circ \beta
 \in \eta$} \; In this case,
\[\pi^\ast(t_\beta) = u x_\beta + v y_{\sigma^{-1}\circ \beta}^p,\]
for some units $u, v\in \widehat{\calO}_{\Ybar, \Qbar}$.
\item\fbox{$\sigma \circ \beta \in \varphi, \sigma^{-1}\circ \beta
\not \in \eta$} \; In this case,
\[\pi^\ast(t_\beta) = u x_\beta,\]
for some unit $u\in \widehat{\calO}_{\Ybar, \Qbar}$.
\item\fbox{$\sigma \circ \beta \not\in \varphi, \sigma^{-1}\circ \beta
 \in \eta$} \; In this case,
\[\pi^\ast(t_\beta) = v y_{\sigma^{-1}\circ \beta}^p,\]
for some unit $v\in \widehat{\calO}_{\Ybar, \Qbar}$.
\item\fbox{$\sigma \circ \beta \not \in \varphi, \sigma^{-1}\circ \beta
\not \in \eta$} \; In this case,
\[\pi^\ast(t_\beta) = 0.\]
\end{enumerate}
\end{lem}

\begin{proof}
We first prove assertion (1). Fix $\beta$. We have \[\beta \in
\varphi\cap \eta, \quad\sigma \circ \beta \in \varphi, \quad
\sigma^{-1}\circ \beta
 \in \eta.\]

 Note that $\sigma \circ \beta \in \varphi \Rightarrow  \beta \in
 \ell (\varphi)\Rightarrow \beta \in \ell(\varphi) \cap \eta = I.$
 Similarly, $\beta \in \varphi \Rightarrow \sigma^{-1}\circ \beta
 \in \ell(\varphi) \Rightarrow \sigma^{-1}\circ \beta \in
 \ell(\varphi) \cap \eta = I$. That is to say, both $\beta$ and
 $\sigma^{-1}\circ \beta$ are critical indices.

 Let $J \subseteq I$, $\varphi_0 = r(\eta - J)^c, \eta_0 = \eta -
 J$. Then $V(\calI_J) = Z^{\wedge \Qbar}_{\varphi_0, \eta_0}$. In
 the following analysis, divided into three cases, we obtain information
 about $\pi^\ast(t_\beta)$ modulo various ideals $\calI_J$, which is
 then assembled to produce the final result $\pi^\ast(t_\beta) = u x_\beta + v y_{\sigma^{-1}\circ
 \beta}^p$. The three cases do not cover all possibilities, but they
 suffice for the following.

\

\id \underline{\textsc{Case A}: \;\;$\{\sigma^{-1}\circ \beta,
\beta\}\subseteq J$}.

\id Since $\sigma^{-1}\circ \beta \in J$ we have $\beta \in r(J)$
and, since $\varphi_0 = r(\eta^c) \cup r(J)$, we have $\beta \in
\varphi_0$. Therefore, $\beta \in \varphi(\Qbar')$ for any point
$\Qbar'\in Z_{\varphi_0, \eta_0}$ and so, by Corollary \ref{corollary: type and phi eta}, $\beta\in \tau(\Qbar')$ if
and only if $\beta \in \eta(\Qbar')$. This can be rephrased as
saying that the vanishing locus of $\pi^\ast(t_\beta)$ on $Z^{\wedge
\Qbar}_{\varphi_0, \eta_0}$ lies inside the closed formal subscheme
$V_\beta^+ \cap Z^{\wedge \Qbar}_{\varphi_0, \eta_0}$ ($V_\beta^+$
was defined in Lemma~\ref{lemma: U phi and V phi}), which in the
completed local ring is defined by the vanishing of $x_\beta$. This
implies that
\[x_\beta \in \sqrt{\langle \pi^\ast(t_\beta) \rangle},\]
in the ring
\[ \widehat{\calO}_{\Ybar, \Qbar}/\calI_J \cong k [\![ \{x_\gamma: \gamma \in J\}, \{y_\gamma:\gamma \in I - J\},
\{z_\gamma:\gamma \in \BB - I\}]\!], \] which is a power-series
ring. Therefore, there exists a positive integer $M(J)$ such that
\[\pi^\ast(t_\beta) - u(J)x_\beta^{M(J)} \in \calI_J,\]
for $u(J)$ a (lift of a unit mod $\calI_J$ and hence a) unit in
$\OhatY$.

\

\id \underline{\textsc{Case B}: \;\;$\{\sigma^{-1}\circ \beta,
\beta\}\subseteq I-J$}.

\id Since $\beta \in \eta$, we have $\beta \in \eta - J= \eta_0$ in
this case. On the other hand, $\sigma^{-1}\circ \beta \not\in J$ and
so $\beta \not\in r(J)$, and also $\sigma^{-1}\circ\beta \in \eta$
and so $\beta \not\in r(\eta)^c$. Together these imply that $\beta
\not \in r(\eta)^c \cup r(J) = \varphi_0$. Arguing as in
\textsc{Case A} (where we had $\beta \in \varphi_0, \beta\not\in
\eta_0$), we deduce that there is a positive integer $N(J)$ such
that
\[ \pi^\ast(t_\beta) - v(J)y_{\sigma^{-1} \circ \beta}^{N(J)} \in \calI_J,\]
for some unit $v(J)\in \OhatY$.

\

\id \underline{\textsc{Case C}: \;\;$\sigma^{-1}\circ \beta \in J,
\beta\not\in J$}.

\id The assumption implies that $\beta \in r(J)$ and so $\beta\in
\varphi_0$. Also, $\beta\in I - J\subseteq \eta - J$ and so $\beta
\in \eta_0$. This implies that $\beta \in  \varphi(\Qbar') \cap
\eta(\Qbar')$ for any closed point $\Qbar' \in Z_{\varphi_0,
\eta_0}$ and hence that $\pi(Z_{\varphi_0, \eta_0}) \subseteq
Z_{\{\beta\}}$. Therefore, $\pi^\ast(t_\beta)$ vanishes identically
on $Z_{\varphi_0, \eta_0}$, that is,
\[\pi^\ast(t_\beta) \in \calI_J.\]

\

\id As mentioned above the remaining case, $\sigma^{-1}\circ \beta
\not\in J, \beta \in J$, is not needed in the following. We now
proceed with the proof of assertion (1).

Taking $J = I$ and $J = \emptyset$, we find $M = M(I), u$ and $N =
N(\emptyset), v$, as in \textsc{Case A} and \textsc{Case B},
respectively. We can therefore write
\[\pi^\ast(t_\beta) = ux_\beta^M + vy_{\sigma^{-1}\circ \beta}^N + E,\]
where the ``error term" $E \in \calI_I \cap\calI_\emptyset \subset
\OhatY$. Reducing this equation modulo $\calI_J$ for every $J$
falling under \textsc{Case C}, we find, using Lemma~\ref{lemma:
dagger} below, that
\[ E \in \bigcap_{\tiny{\begin{matrix}
J \subseteq I,\\
\beta\not\in J, \sigma^{-1}\circ \beta \in J. \end{matrix}}} \calI_J
= \langle x_\beta, y_{\sigma^{-1}\circ \beta}\rangle .
\]
Therefore, we may write $E = Ax_\beta + By_{\sigma^{-1}\circ \beta}$
and, by reducing modulo $\calI_I$ and $\calI_\emptyset$, we find
that
\[ \pi^\ast(t_\beta) = ux_\beta^M + vy_{\sigma^{-1}\circ \beta}^N +
Ax_\beta + By_{\sigma^{-1}\circ \beta}, \qquad A \in \langle
y_\gamma: \gamma \in I\rangle, B \in \langle x_\gamma:\gamma\in
I\rangle.\] Choosing $J$ as in \textsc{Case B}, we can find $N(J),
v(J)$, such that $\pi^\ast(t_\beta) - v(J) \cdot y_{\sigma^{-1}\circ
\beta}^{N(J)} \in \calI_J$. Hence, we find that
$vy_{\sigma^{-1}\circ \beta}^N + B y_{\sigma^{-1}\circ \beta} = v(J)
y_{\sigma^{-1}\circ \beta}^{N(J)}$ in the local ring
\[R_J = \OhatY/\calI_J = k[\![ \{x_\gamma: \gamma \in J\},
\{y_\gamma: \gamma \in I-J\}, \{z_\gamma: \gamma\in \BB - I\}]\!].\]
Let $v_0$ and $v_0(J)$ denote, respectively, the constant term in
$v$ and $v(J)$ in the power-series ring $R_J$. Because $B \in
\langle x_\gamma: \gamma \in I\rangle$, it follows that $v_0 =
v_0(J)$ and $N = N(J)$. That implies that $B \in \langle
y_{\sigma^{-1}\circ \beta}^{N-1}\rangle + \calI_J$ for all $J$ as in
\textsc{Case B}. Therefore, \begin{align*} B \in
\bigcap_{\tiny{\begin{matrix}
J \subseteq I\\
\{\sigma^{-1}\circ \beta, \beta\}\subseteq I - J
\end{matrix}}} \langle
y_{\sigma^{-1}\circ \beta}^{N-1}\rangle + \calI_J & = \langle
y_{\sigma^{-1}\circ \beta}^{N-1}, x_{\sigma^{-1}\circ \beta},
x_\beta \rangle, &{\rm \text{(By Lemma~\ref{lemma: dagger}.)}}
\end{align*}
Write $B = ay_{\sigma^{-1}\circ \beta}^{N-1} + bx_{\sigma^{-1}\circ
\beta} + c x_\beta$, where $a\in \langle x_\beta:\beta\in I\rangle$.
Hence:
\begin{equation} \label{equation: pi star t beta}
\pi^\ast(t_\beta) = ux_\beta^M + vy_{\sigma^{-1}\circ \beta}^N +
Ax_\beta + a y_{\sigma^{-1}\circ \beta}^N + cx_\beta
y_{\sigma^{-1}\circ \beta}.
\end{equation} To proceed, we need a sub-lemma.

\medskip

\id {\bf Sub-lemma.} $M=1, N=p$.
\begin{proof}{(Sub-lemma)}
First we prove the statement for $\Qbar\in \Ybar_F \cap
\Ybar_V$. At such a point $\Qbar$, $\varphi(\Qbar) = \eta(\Qbar)
= \BB = I$. Since $\Ybar_F = Z_{\BB, \emptyset}$, it follows that
$V(\calI_I)$ is the image of $\Ybar_F$ in $\Spf (\OhatY)$. Since
there is a section to $\pi\vert_{\Ybar_F}$, we find that
$\pi^\ast(Z_{\{\beta\}})$ is a reduced Weil divisor on $\Ybar_F$, and hence
$\pi^\ast (t_\beta) = ux_\beta^M$ in $\OhatY/\calI_I$ implies that
$M=1$. A similar argument using $\Ybar_V$ and  applying Proposition \ref{prop: pi, infinitesimally, on horizontal strata} gives that $N = p$.

Now for the general case:  the point $\Qbar$ belongs to $Z_{\varphi_0, \eta_0} = Z_{r(\eta)^c \cup r(J), \eta -
J}$ for $J=I$, and so it belongs to an irreducible component $C$  of  $Z_{\varphi_0, \eta_0}$. Let $D$ be the
pull-back of the Cartier divisor $Z_{\{\beta\}}$, under $\pi\colon C
\arr \Xbar$. Since $\beta \in \varphi(\Qbar) \cap \eta(\Qbar) =
\varphi\cap \eta$, but $\beta\in \varphi_0-\eta_0$, we see that
 $D$ is a non-zero Cartier divisor on $C$ containing $\Qbar$, and since $C$ is
non-singular, in fact a non-zero effective Weil divisor on $C$. Let
us write then
\[ D = n_1D_1 + n_2 D_2 + \dots + n_rD_r,\]
a sum of irreducible Weil divisors $D_i$, and say $\Qbar\in D_1$.
Since $Z_{\{\beta\}}$ is given formally locally by the vanishing of
$t_\beta$, we see that $D$ at $\Qbar$ is given by the vanishing of
$\pi^\ast(t_\beta)$ modulo $\calI_I$, which is equal to
$ux_\beta^M$. This, in turn, implies that $D$ is locally irreducible
and hence $\Qbar$ belongs only to $D_1$ and $M=n_1$.

\medskip

\id {\bf Sub-sub-lemma.} $D_1\cap \Ybar_F \cap \Ybar_V \neq
\emptyset$.

\begin{proof}{(Sub-sub-lemma)}   We have $\beta \in  \varphi_0$ and so $\beta \in \varphi(\Rbar)$ for any closed point $\Rbar \in
Z_{\varphi_0, \eta_0}$. Therefore, $\pi(\Rbar) \in Z_{\{\beta\}}$ if
and only if $\beta \in \eta(\Rbar)$ (Corollary~\ref{corollary: type
and phi eta}). Therefore, at least set theoretically,
\[ Z_{\varphi_0, \eta_0} \cap \pi^{-1}(Z_{\{\beta\}}) = Z_{\varphi_0, \eta_0 \cup \{\beta\}}.\]
Since $Z_{\varphi_0, \eta_0}$ is a disjoint union of its components,
every irreducible component of $C \cap \pi^{-1}(Z_{\{\beta\}})$
(i.e., $D_1, \dots, D_r$) is an irreducible component of
$Z_{\varphi_0, \eta_0 \cup \{\beta\}}$. By Theorem~\ref{theorem: C
cap YF cap YV}, every such irreducible component contains a point
$\Rbar$ that also belongs to $\Ybar_F \cap \Ybar_V$. That ends
the proof of the sub-sub-lemma.
\end{proof}
\id Back to the proof of the sub-lemma. Let $\Rbar \in
D_1\cap\Ybar_F \cap \Ybar_V$ be a closed point. Repeating the
argument above now for the point $\Rbar$, we find that $M = n_1 =
M(\Rbar)$, where $M(\Rbar)$ is the exponent $M$ occurring in
Equation~(\ref{equation: pi star t beta}) at the point $\Rbar$.
Using the argument already done for the special case of points on
$\Ybar_F \cap \Ybar_V$, we conclude that $M(\Rbar) = 1$ and
hence $M = 1$. A similar argument gives that $N=p$ and this
concludes the proof of the Sub-lemma.
\end{proof}
\id We can now refine Equation~(\ref{equation: pi star t beta}) and
write
\begin{equation*}
\pi^\ast(t_\beta) = (u+A + cy_{\sigma^{-1} \circ \beta})x_\beta +
(v+a) y_{\sigma^{-1}\circ\beta}^p.
\end{equation*}
Note that both $u':= u + A + cy_{\sigma^{-1} \circ \beta}$ and
$v':=v+a$ are units in $\OhatY$ as $ A + cy_{\sigma^{-1} \circ
\beta}$ and $a$ are in the maximal ideal and $u, v$ are units. Thus,
\begin{equation*}
\pi^\ast(t_\beta) = u'x_\beta + v'y_{\sigma^{-1}\circ \beta}^p,
\end{equation*}
concluding the proof of assertion (1) of the Lemma.

\

\id Assertions (2) and (3) are proved in a similar fashion and the
argument, if anything, is easier. We discuss now assertion (4).
Recall that in this case,
\[\beta \in\varphi\cap \eta, \quad \sigma \circ \beta \not\in \varphi, \quad \sigma^{-1}\circ \beta \not\in \eta.\]
This implies that \[\{\sigma^{-1} \circ \beta, \beta\}\subseteq
I^c,\] because $\beta \not\in \ell(\varphi) \Rightarrow \beta
\not\in I$ and $\sigma^{-1}\circ \beta \not\in \eta \Rightarrow
\sigma^{-1}\circ \beta \not\in I$. Consequently, for every $J
\subseteq I$, $\beta \in \eta - J = \eta_0$ and also $\beta \in
r(\eta)^c \subseteq r(\eta)^c \cup r(J) = \varphi_0$. Thus, $\beta
\in \varphi_0 \cap \eta_0$ on any such $Z_{\varphi_0, \eta_0}$, which
implies that $\pi^\ast(t_\beta) = 0$ mod $\calI_J$ for all $J
\subseteq I$. Since these correspond to all the irreducible
components through $\Qbar$, $\pi^\ast(t_\beta) = 0$.
\end{proof}

During the proof of the Key-Lemma we have used the following result.

\begin{lem}
\label{lemma: dagger}
Let $I$ be a subset of $\BB$ and consider the ring $k[\![\{x_\beta:\beta \in I\}, \{y_\beta: \beta \in I\}, \{z_\beta: \beta \in \BB - I\}]\!]/\langle
\{x_\beta y_\beta: \beta \in I\}\rangle$ and for $J \subseteq I$ its ideal
$\calI_J = \langle \{x_\beta: \beta \not\in J\}, \{y_\beta; \beta \in J\} \rangle$. Then:
\begin{enumerate}
\item $\underset{{\tiny{\begin{matrix}
J \subseteq I,\\
\beta\not\in J, \sigma^{-1}\circ \beta \in J. \end{matrix}}}}{\bigcap} \calI_J
= \langle x_\beta, y_{\sigma^{-1}\circ \beta}\rangle$.
\item  $\underset{{\tiny{\begin{matrix}
J \subseteq I\\
\{\sigma^{-1}\circ \beta, \beta\}\subseteq I - J
\end{matrix}}}}{\bigcap} \langle
y_{\sigma^{-1}\circ \beta}^{N-1}\rangle + \calI_J = \langle
y_{\sigma^{-1}\circ \beta}^{N-1}, x_{\sigma^{-1}\circ \beta},
x_\beta \rangle$.
\end{enumerate}
\end{lem}

\begin{proof} Consider the power series ring
\[ R=  k[\![ x_\beta, y_\beta, z_\gamma]\!]_{\beta\in I, \gamma\in I^c}\]
and the preimage of the above ideals, under the natural projection,
\[\calI_J^0 = \langle \{x_\beta: \beta \not\in J\}, \{y_\beta; \beta \in J\} \rangle,\qquad  \calI_J^1 = \langle y_{\sigma^{-1} \circ \beta}^{N-1}\rangle
+ \calI_J^0.\]
It is enough to show that in $R$ we have
\begin{enumerate}
\item $\underset{{\tiny{\begin{matrix}
J \subseteq I,\\
\beta\not\in J, \sigma^{-1}\circ \beta \in J. \end{matrix}}}}{\bigcap} \calI_J^0
= \langle x_\beta, y_{\sigma^{-1}\circ \beta}, \{x_\gamma y_\gamma: \gamma \in I\}\rangle$.
\item  $\underset{{\tiny{\begin{matrix}
J \subseteq I\\
\{\sigma^{-1}\circ \beta, \beta\}\subseteq I - J
\end{matrix}}}}{\bigcap}
\calI_J^1 = \langle
y_{\sigma^{-1}\circ \beta}^{N-1}, x_{\sigma^{-1}\circ \beta},
x_\beta, \{x_\gamma y_\gamma: \gamma \in I\} \rangle$.
\end{enumerate}
Recall that a monomial ideal of $R$ is an ideal generated by monomials. The ideals $\calI_J^0,
\calI_J^1$ are monomial ideals. The statements follow easily from the following result. Let $\gera = \langle f_1, \dots, f_a \rangle,
\gerb = \langle g_1, \dots, g_b\rangle$, be two monomial ideals of $R$ then $\gera \cap \gerb$ is a monomial ideal of $R$ and
\[ \gera\cap \gerb = \langle \{ \lcm(f_i, g_j) : 1 \leq i \leq a, 1 \leq j \leq b\}\rangle\]
(cf. \cite[\S 15, Exercise 15.7]{Eisenbud}).
\end{proof}

At a point $\Qbar$, the Key Lemma gives information only about $\pi^\ast(t_\beta)$ with $\beta\in\varphi(\Qbar)\cap\eta(\Qbar)$. If no such $\beta$ exists, that is, if $\varphi(\Qbar)\cap\eta(\Qbar)=\emptyset$, then the admissibility condition implies that  $(\varphi(\Qbar),\eta(\Qbar))=(\BB_\gert,\BB_{\gert^\ast})$ for some $\gert\vert p$, and so $\Qbar$ belongs to the horizontal stratum $W_{\BB_\gert,\BB_{\gert^\ast}}$. The following lemma studies this situation; it was in fact used in the proof of the Key Lemma.

In the case at hand, we have $I(\Qbar)=\emptyset$, and hence the isomorphism (\ref{equation: stamm isomorphism}) becomes
\begin{equation}\label{equation: stamm isomorphism at horizontal strata}
\widehat{\calO}_{\Ybar, \Qbar} \cong k[\![z_\beta: \beta \in \BB ]\!].
\end{equation}

\begin{prop}\label{prop: pi, infinitesimally, on horizontal strata} Let $\Qbar \in W_{\BB_\gert,\BB_{\gert^\ast}}$, and $\Pbar=\pi(\Qbar)$.  We can choose isomorphisms as in (\ref{equation: GO isomorphism}) at $\Pbar$, and (\ref{equation: stamm isomorphism at horizontal strata}) at $\Qbar$, such that:
\[
\pi^{*}(t_\beta)=\begin{cases} z_\beta\qquad \beta  \in \BB_\gert \\    z_\beta^p \qquad \beta \in \BB_{\gert^\ast}.\end{cases}
\]
\end{prop}

\begin{proof} See Appendix, Lemma \ref{lemma: pi, infinitesimally, on horizontal strata (Appendix)}.
\end{proof}

\begin{rmk} To prove the sub-lemma that appeared in the proof of the Key Lemma, we appealed to Theorem \ref{theorem: C cap YF cap YV}.  We remark that the sublemma can also be proven directly from Proposition \ref{prop: pi, infinitesimally, on horizontal strata}, by considering the horizontal strata that pass through the point $\Qbar$. There, we only used information concerning the infinitesimal description of $\pi$ on two horizontal strata, that is $\Ybar_F$ and $\Ybar_V$; this is a special case of Proposition \ref{prop: pi, infinitesimally, on horizontal strata}.
\end{rmk}
 \

\

\section{Extension to the cusps}

\subsection{Notation} We will let $\cX$ denote the minimal  compactification of $X$ defined over $W(\kappa)$, and $\cXbar$ its special fibre over $\kappa$ (see \cite{Chai} and the references therein).   Define $\cY$, $\cYbar$  similarly. The morphisms $\pi$, $w$, and $s$ extend to morphisms $\pi: \cY \arr \cX$, $w:\cY \arr \cY$, and $s\colon \cXbar \arr \cYbar$;   Reducing modulo $p$, we obtain $\pi\colon\cYbar\arr\cXbar$.

\

\subsection{Extension of the stratification}
Let $\tau\subseteq \BB$. We define $\WW_\tau=W_\tau$ and $\ZZ_\tau=Z_\tau$, unless $\tau=\emptyset$, in which case we set $\WW_\emptyset=W_\emptyset\cup (\cXbar-\Xbar)$ and $\ZZ_\emptyset=\cXbar$. The collection $\{\WW_\tau\}_{\tau\subseteq \BB}$ is a stratification of $\cXbar$. In fact, we find that $\ZZ_\tau$ is the Zariski closure of $Z_\tau$ in $\cXbar$. This follows from the fact that  $\cup_{\tau\neq\emptyset} Z_\tau$ is Zariski closed in $\cXbar$ as it is the union of the vanishing loci of the partial Hasse invariants which have constant $q$-expansion at infinity \cite{G}.   For a point $\Pbar\in\cXbar-\Xbar$, we define $\tau_\beta(\Pbar)=0$ for all $\beta\in \BB$.

For an admissible pair $\pe$, define  $\cZpe$ to be the Zariski closures of $\Zpe$ in $\cYbar$.  We define $\cWpe=\cZpe - \cup_{(\varphi^\prime,\eta^\prime) > \pe} \ZZ_{\varphi^\prime,\eta^\prime}$.

\begin{thm} \label{theorem: extension of stratification to cusps} Let $\pe$ be an admissible pair.
\begin{enumerate}
\item $\cWpe=\Wpe$ and $\cZpe=\Zpe$, unless there is $\gert\vert p$ such that $\pe=(\BB_\gert,\BB_{\gert^\ast})$, in which case $\WW_{\varphi,\eta}-W_{\varphi,\eta}=\ZZ_{\varphi,\eta}-Z_{\varphi,\eta}$ is non-empty, lies inside $\cYbar-\Ybar$, and consists of finitely many cusps if $\cYbar= {^{\rm m}}\cYbar$.
\item $\dim(\cWpe)=\dim (\Wpe) = 2g - (\sharp\; \varphi + \sharp \eta)$.
\item The irreducible components of $\cYbar$ are the irreducible
components of the strata $\ZZ_{\varphi, \ell(\varphi^c)}$ for $\varphi
\subseteq \BB$.
\item If $(\varphi^\prime,\eta^\prime)$ is another admissible pair, then $\cZpe \cap \ZZ_{\varphi^\prime,\eta^\prime}=\Zpe \cap Z_{\varphi^\prime,\eta^\prime}$.
\item The collection $\{\cWpe\}_{\varphi,\eta}$ is a stratification of $\cYbar$.
\end{enumerate}
\end{thm}

\begin{proof} For an admissible pair $(\varphi,\eta)$, we have $\varphi\cap\eta=\emptyset$ if and only if $\pe=(\BB_\gert,\BB_{\gert^\ast})$ for some $\gert\vert p$. This is because, by assumptions,  $\ell(\varphi^c)\subseteq \eta \subseteq \varphi^c$, and hence $\ell(\varphi)=\varphi$ and $\eta=\varphi^c$, implying $\varphi=\BB_\gert$ and $\eta=\BB_{\gert^\ast}$ for some $\gert\vert p$.   Let  $Z=\cup_{\varphi\cap\eta\neq\emptyset} Z_{\varphi,\eta}$.  We  have
\[
\cYbar-Z=\bigcup_{\gert\vert p} W_{\BB_\gert,\BB_{\gert^\ast}} \cup (\cYbar-\Ybar).
\]
We have $\pi(Z)\subseteq \cup_{\tau\neq\emptyset} Z_{\tau}$ by part (1) of Corollary \ref{corollary: type and phi eta}. Also, since $\pi(\cYbar-\Ybar)=\cXbar-\Xbar$, and by Corollary \ref{corollary: type and phi eta} , we see that $\pi(\cYbar-Z)$ lies in $\WW_\emptyset$. Therefore, $Z=\pi^{-1}(\cup_{\tau\neq\emptyset} Z_\tau)=\pi^{-1}(\cup_{\tau\neq\emptyset} \ZZ_\tau)$, and in particular, $Z$ is Zariski closed in $\cYbar$. This proves (1).

Statement (2) is clear from part (2) of Theorem \ref{theorem: fund'l facts about the stratificaiton of Ybar}.  Statement (3) follows from part (3) of Theorem \ref{theorem: fund'l facts about the stratificaiton of Ybar}, since the irreducible components of $\cYbar$ are the Zariski closures of the irreducible components of $\Ybar$ in $\cYbar$. Statement (5) is clear from the definition.

We prove (4) in the case where $(\varphi,\eta)=(\BB_\gert,\BB_{\gert^\ast})$ and $(\varphi^\prime,\eta^\prime)=(\BB_\gerz,\BB_{\gerz^\ast})$ for two distinct ideals $\gert, \gerz$ dividing $p$ (the other cases follow easily from (1)).   By part (3), $\ZZ_{\BB_\gert,\BB_{\gert^\ast}}$ and $\ZZ_{\BB_\gerz,\BB_{\gerz^\ast}}$ are unions of irreducible components of $\cYbar$. Since $\gert\neq\gerz$,  $\ZZ_{\BB_\gert,\BB_{\gert^\ast}}$ and $\ZZ_{\BB_\gerz,\BB_{\gerz^\ast}}$ have no irreducible components in common, as their intersection has dimension strictly less than each of them.   Since  $\cYbar$ is integral (indeed normal) at every point of $\cYbar-\Ybar$, it follows that
 $\ZZ_{\BB_\gert,\BB_{\gert^\ast}}\cap \ZZ_{\BB_\gerz,\BB_{\gerz^\ast}} \cap  (\cYbar-\Ybar)=\emptyset$, and the result follows.
\end{proof}

\begin{dfn}\label{definition: Lambda cusps}
By Theorem \ref{theorem: extension of stratification to cusps}, every closed point $\Qbar$ in $\cYbar-\Ybar$ belongs to a unique horizontal stratum $\WW_{\BB_\gert,\BB_{\gert^\ast}}$. We call $\Qbar$ a $\gert$-cusp. We define $(\varphi(\Qbar),\eta(\Qbar))=( \BB_\gert,\BB_{\gert^\ast})$; we also set $I(\Qbar)=\ell(\varphi(\Qbar))\cap \eta(\Qbar)=\emptyset$. Note that in this case, $(\varphi(\Qbar),\eta(\Qbar))$ is an admissible pair.
\end{dfn}

\begin{prop}\label{proposition: Lambda of w}
If $\Qbar$ is  a $\gert$-cusp, then $w(\Qbar)$ is a $\gert^\ast$-cusp.
\end{prop}

\begin{proof}
This follows from Proposition \ref{proposition: w and phi eta}.
\end{proof}

Given $\varphi\subseteq \BB$,
define  $\UU_\varphi^+$ to be the Zariski closure of $U_\varphi^+$ in $\cYbar$.  Similarly, for $\eta\subseteq\BB$, define $\VV_\eta^+$ to be the Zariski closure of $V_\eta^+$ in $\cYbar$.
\begin{lem}\label{lemma: extension of U phi and V eta} Let $\varphi$ and $\eta$ be subsets of $\BB$.
\begin{enumerate}
\item $\UU_\varphi^+$ is the closed subset of $\cYbar$ consisting of points $\Qbar$ with $\varphi(\Qbar) \supseteq \varphi$. Similarly, $\VV_\eta^+$ is the closed subset of $\cYbar$ consisting of points $\Qbar$ with $\eta(\Qbar) \supseteq \eta$

\item Assume $\beta \in \BB$. If $\Qbar\in \UU_{\{\beta\}}^+$, then $\UU_{\{\beta\}}^+\cap\Spf(\OhatQbar)$ is equal to $\Spf(\OhatQbar)$ if $\beta\not\in r(I(\Qbar))$, and is otherwise given by the vanishing of $y_{\sigma^{-1}\circ \beta}$. If $\Qbar\in \VV_{\{\beta\}}^+$, then $\VV_{\{\beta\}}^+\cap\Spf(\OhatQbar)$ is equal to $\Spf(\OhatQbar)$ if $\beta\not\in I$, and is otherwise given by the vanishing of $x_{\beta}$.
\end{enumerate}
\end{lem}

\begin{proof} Statement (1) follows from part (1) of Theorem \ref{theorem: extension of stratification to cusps}. To prove (2), it is enough to consider $\Qbar\in \cYbar-\Ybar$; the rest is covered by Lemma \ref{lemma: U phi and V phi}.  In that case, there is an ideal $gert\vert p$ such that $(\varphi(\Qbar),\eta(\Qbar))=(\BB_\gert,\BB_{\gert^\ast})$, $I(\Qbar)=\emptyset$, and $\beta\in\BB_\gert$. Then, $\WW_{\BB_\gert,\BB_{\gert^\ast}}$ is a Zariski open subset of $\cYbar$ (being the complement of  $\cup_{\pe\neq(\BB_\gert,\BB_{\gert^\ast})} \ZZ_{\varphi,\eta}$) containing $\Qbar$ which lies entirely in $\UU^+_{\{\beta\}}$. This implies that $\UU_{\{\beta\}}^+\cap\Spf(\OhatQbar)$ is equal to $\Spf(\OhatQbar)$. \end{proof}

 \begin{dfn}\label{definition: ordinary strata with cusps}
Let $\cXbar^{\rm ord}=\WW_\emptyset$ and  $\cYbar^{\rm ord}=\cup_{\gert\vert p}  \WW_{\BB_\gert,\BB_{\gert^\ast}}$. We define $\cYbar^{\rm ord}_F$ to be $\WW_{\BB,\emptyset}$; it is the image of $\cXbar^{\rm ord}$ under the section $s$. We define $\cYbar^{\rm ord}_V$ to be $\WW_{\emptyset,\BB}$; it is equal to $w(\cYbar^{\rm ord}_F)$. Compare with Proposition \ref{proposition: ordinary strata}.
\end{dfn}

\

\

\section{Valuations and a dissection of $\Yrig$}
\subsection{Notation}  We denote the completions of $\cX$, $\cY$ along their special fibres, respectively, by $\gerX$, $\gerY$. These are quasi-compact quasi-separated topologically finitely generated (i.e., {\emph{admissible}}) formal schemes over $W(\kappa)$. By Raynaud's work, one can associate to them their rigid analytic generic fibres, respectively, $\Xrig$, $\Yrig$, which are quasi-compact, quasi-separated rigid analytic varieties over $\QQ_\kappa$. Since $\cX$, $\cY$ are proper over $W(\kappa)$, one sees that $\Xrig$, $\Yrig$ are in fact, respectively, the analytifications of $\cX \otimes_{W(\kappa)} \QQ_\kappa$, $\cY\otimes_{W(\kappa)} \QQ_\kappa$. Note that these spaces have natural models defined over $\QQ_p$, which we denote, respectively, by $\gerX_{\rig,\QQ_p}$, $\gerY_{\rig,\QQ_p}$. We say that a point $P\in\Xrig$ has \emph{cuspidal reduction} if $\Pbar$ is closed point of $\cXbar-\Xbar$, and otherwise we say it has \emph{non-cuspidal reduction}. We use a similar terminology for points of $\Yrig$.

\

\subsection{Valuations on $\Xrig$ and $\Yrig$}\label{subsection: valuations on X and Y}
\id Let $\CC_p$ be the completion of an algebraic closure of $\QQ_p$. It
has a valuation $\val:\CC_p \arr \QQ \cup \{\infty\}$ normalized so
that $\val(p)=1$. Define
\[
\nu(x)=\min\{\val(x),1\}.
\]

\begin{lem} \label{lemma: independence from parameters} Let  $P\in \Xrig$ and $Q\in \Yrig$ be points of non-cuspidal reduction.
\begin{enumerate}
\item Let $\{t_\beta\}_{\beta \in \BB}$ and $\{t^\prime_\beta\}_{\beta \in \BB}$ be two
sets of parameters at $\Pbar$ as in (\ref{equation: local def ring of X at P bar}). For any $\beta\in\tau(\Pbar)$ there is
$f_\beta\in\OhatP$,  $\epsilon_\beta\in\OhatP^\times$ such that
$t^\prime_\beta=\epsilon_\beta t_\beta +pf_\beta$.

\item Let $\{x_\beta\}_{\beta\in I(\Qbar)}$,  $\{x^\prime_\beta\}_{\beta\in I(\Qbar)}$ be parameters as in (\ref{equation: local deformation ring at Q bar -- arithmetic
version}). For any $\beta\in I(\Qbar)$, there is $g_\beta \in\OhatQ$,
$\delta_\beta\in\OhatQ^\times$ such that
$x^\prime_\beta=\delta_\beta x_\beta +pg_\beta$.
\end{enumerate}
\end{lem}

\begin{proof}  We denote the reduction modulo $p$ of a parameter $t$ by
$\overline{t}$. Let $\Pbar=\spe(P) \in \Xbar$ be defined over $k$.
Reducing modulo $p$ we obtain isomorphisms
$k[\![\overline{t}_\beta]\!] \cong \OhatPbar \cong
k[\![\overline{t}^\prime_\beta]\!]$, where for each $\beta$ the
vanishing loci of $\overline{t}_\beta$ and
$\overline{t}^\prime_\beta$ are both equal  to $Z_\beta \cap \Spf(\OhatPbar)$.  This proves the claim. A similar proof works for the second part; by Lemma \ref{lemma: U phi and V phi}, the vanishing locus of $\overline{x}_\beta$ is $V_\beta^+ \cap \Spf(\OhatQbar)$.
\end{proof}

We now define valuation vectors for points on $\Xrig$ and $\Yrig$. Let $P\in\Xrig$.  Let $D_P={\rm
sp}^{-1}(\Pbar)$, which, by Berthelot's construction,  is the rigid analytic space associated to
$\Spf(\OhatP)$. If $P$ has non-cuspidal reduction, the parameters $t_\beta$ in (\ref{equation: local def ring of X at P bar}) are
functions on $D_P$. We define $\nu_\gerX(P)=(\nu_\beta(P))_{\beta\in
\mathbb{B}}$, where the entries $\nu_\beta(P)$ are given by
\[
\nu_\beta(P)=\begin{cases} \nu(t_\beta(P)) & \beta \in
\tau(\overline{P}),\\ 0 & \beta \not\in \tau(\overline{P}).
\end{cases}
\]
If $P$ has cuspidal reduction, this gives $\nu_\beta(P)=0$ for all $\beta\in\BB$.
By Lemma \ref{lemma: independence from parameters}, the above definition
is independent of the choice of parameters as in (\ref{equation: local def ring of X at P bar}).   In particular, we can define $\nu_\gerX$ using
parameters that are liftings of the partial Hasse invariants at the
point $P$. More precisely, let $h_\beta$ be the partial Hasse
invariant discussed in \S \ref{subsection: stratification of Y bar}. Then
$\nu_\beta(P)=\nu(\tilde{h}_\beta(P))$, where $\tilde{h}_\beta$ is
any lift of $h_\beta$ locally  at $P$.

Similarly, for $Q\in\Yrig$, we define
$\nu_\gerY(Q)=(\nu_\beta(Q))_{\beta\in \mathbb{B}}$, where
 \[
 \nu_\beta(Q)=\begin{cases}
1 & \beta \in \eta(\overline{Q})-I(\overline{Q}),\\
\nu(x_\beta(Q))&\beta \in I(\overline{Q}),\\
0&\beta \not\in \eta(\overline{Q}).
\end{cases}
\]
Again, by Lemma \ref{lemma: independence from parameters}, this
definition is independent of the choice of parameters as in  (\ref{equation: local deformation ring at Q bar -- arithmetic version}).
For a point $Q\in\Yrig$ of cuspidal reduction,  the above definition simplifies as follows: there is a unique $\gert\vert p$ such that $\Qbar$ is  a $\gert$-cusp  (see Definition \ref{definition: Lambda cusps}), and
\[
 \nu_\beta(Q)=\begin{cases}
1 & \beta \in \BB_{\gert^\ast},\\
 0&\beta \not\in \BB_{\gert^\ast}. \end{cases}
\]
\medskip

Let ${\bf 1}=(1)_{\beta\in \BB}$ denote the constant vector of $1$'s, ${\bf 0}=(0)_{\beta\in \BB}$ the constant vector of $0$'s, etc.
\begin{prop}\label{proposition: valuations under w}
For any $Q \in \Yrig$, we have
$\nu_\gerY(Q)+\nu_\gerY(w(Q))={\bf 1}$.
\end{prop}

\begin{proof}  First we assume that $Q$, and hence $w(Q)$, have non-cuspidal reduction. Fix isomorphisms as
in  Lemma \ref{lemma: w and local parameters} at $\Qbar$ and $w(\Qbar)$. To avoid confusion, we
will decorate any  parameter with the point at which it is defined.
For example, we will use $x_{\beta,\Qbar}$ to denote the parameter $x_\beta$ chosen at $\Qbar$.   Assume
that $\beta \in I(\Qbar)=I(w(\Qbar))$ (see Proposition \ref{proposition: w and phi eta}).
Then,
\begin{equation*}
\nu_\beta(Q)=\nu(x_{\beta,\Qbar}(Q))=\nu(w^*y_{\beta,w(\Qbar)}(Q))=\nu(y_{\beta,w(\Qbar)}(w(Q))).
\end{equation*}
Using the relation $ x_{\beta,w(\Qbar)} y_{\beta,w(\Qbar)}=p$, we see that
\[
\nu_\beta(Q)=1-\nu(x_{\beta,w(\Qbar)}(w(Q)))=1-\nu_\beta(w(Q)).
\]
By definition, $0<\nu_\beta(Q)<1$
if and only if $\beta \in I(\Qbar)$. Since $I(\Qbar)=I(w(\Qbar))$, to
prove the claim it remains to show that $\nu_\beta(Q)=1$ if and only
if $\nu_\beta(w(Q))=0$. We have $\nu_\beta(Q)=1$ if and only if $\beta
\in
\eta(\Qbar)-I(\Qbar)=\ell(\varphi(w(\Qbar)))-I(w(\Qbar))=\BB-\eta(w(\Qbar))$.
But this, by definition, is equivalent to $\nu_\beta(w(Q))=0$.

Now, assume $Q$ has cuspidal reduction. There is a unique $\gert\vert p$, such that $\Qbar$ is  a $\gert$-cusp. By Proposition \ref{proposition: Lambda of w},  $w(\Qbar)$  is a  $\gert^\ast$-cusp. The result now follows from the definition of $\nu_\gerY$.
\end{proof}

\

\subsection{The valuation cube}
 Let $\Theta=[0,1]^\BB$ be the unit cube in
$\RR^\BB\cong\RR^g$.  Its ``open faces'' can be encoded by vectors
${\bf a}=(a_\beta)_{\beta\in\BB}$ such that
$a_\beta\in\{0,\ast,1\}$.  The face  corresponding to ${\bf a}$ is
the set
\[
\calF_{\bf a}:=\{ {\bf v}=(v_\beta)_{\beta\in\BB}\in \Theta: v_\beta=a_\beta\ {\rm if}\
a_\beta\neq\ast,\ {\rm and}\ 0<v_\beta<1\ {\rm otherwise}\}.
\]
There are $3^g$ such faces. The {\emph{star}} of an open face $\calF$ is
${\rm Star}(\calF)=\cup_{\overline{\calF^\prime} \supseteq \calF}
\calF^\prime$, where the union is over all open faces $\calF^\prime$
whose topological closure contains $\calF$. For ${\bf a}$ as above,
we define

\begin{align*}
\eta({\bf a})&=\{\beta\in\BB: a_\beta\neq 0\},\\
I({\bf a})&=\{\beta\in\BB:a_\beta=\ast\},\\
\varphi({\bf a})&=r(\eta({\bf a})^c\cup I({\bf a}))=\{\beta\in\BB:
a_{\sigma^{-1}\circ\beta}\neq 1\}.
\end{align*}

\medskip

\begin{thm} \label{theroem: theorem of cube} There is a  one-to-one correspondence  between the open faces of $\Theta$ and the strata
$\{\WW_{\varphi,\eta}\}$ of $\cYbar$, given by
\[
\calF_{\bf a} \mapsto
\WW_{\varphi({\bf a}),\eta({\bf a})}.
\]
It has the following properties:

\begin{enumerate}

\item  $\nu_\gerY(Q) \in \calF_{\bf a}$ if and only if $\Qbar\in \WW_{\varphi({\bf a}),\eta({\bf a})}$.

\item $\dim(\WW_{\varphi({\bf a}),\eta({\bf a})})=g-\dim(\calF_{\bf a})=\sharp\,\{\beta: a_\beta\neq\ast\}$.

\item If $\calF_{\bf a}\subseteq\overline{\calF}_{\bf b}$, then $\WW_{\varphi({\bf b}),\eta({\bf b})}\subseteq\overline{\WW_{\varphi({\bf a}),\eta({\bf a})}}$ and vice versa;
that is, the correspondence is order reversing. In particular,
$\nu_\gerY(Q)\in {\rm Star}(\calF_{\bf a}) \Longleftrightarrow \Qbar
\in \ZZ_{\varphi({\bf a}),\eta({\bf a})}$.

\end{enumerate}
\end{thm}

\begin{proof}
(1) is clear from the definitions. To prove (2), we write
\[
g-\dim(\calF_{\bf a})=g-\sharp\{\beta\in\BB:
a_\beta=\ast\}=2g-(\sharp\varphi({\bf a})+\sharp\eta({\bf
a}))=\dim(\WW_{\varphi({\bf a}),\eta({\bf a})}),
\]
using  Theorem \ref{theorem: extension of stratification to cusps} for the last equality.

Next, we prove (3). We have $\calF_{\bf
a} \subseteq \overline{\calF}_{\bf b}$ if and only if the following
hold:
\begin{align*}
a_\beta&=\ast \Rightarrow b_\beta=\ast, \\
a_\beta&=0 \Rightarrow  b_\beta\neq 1, \\
a_\beta&=1 \Rightarrow b_\beta\neq 0.
\end{align*}

These conditions, in turn, are equivalent to the following:
\begin{align*}
 I({\bf a}) \subseteq I({\bf b}), \\
 r(\eta({\bf a}))^c \subseteq\varphi({\bf b}),\\
 \ell(\varphi({\bf a}))^c \subseteq \eta({\bf b}).
\end{align*}

The above conditions are  equivalent to $ (\varphi({\bf
b}),\eta({\bf b}))\geq (\varphi({\bf a}),\eta({\bf a}))$, because we
can write $\eta({\bf a})=I({\bf a})\cup \ell(\varphi({\bf a}))^c
\subseteq I({\bf b}) \cup \eta({\bf b})=\eta({\bf b})$, and
similarly, $\varphi({\bf a})=r(I({\bf a}))\cup r(\varphi({\bf a}))^c
\subseteq r(I({\bf b})) \cup \varphi({\bf b})=\varphi({\bf b})$. The
other direction  of equivalence follows easily  using the
admissibility of all the pairs $(\varphi,\eta)$ appearing above.
Finally, by Theorem \ref{theorem: extension of stratification to cusps},  $(\varphi({\bf
b}),\eta({\bf b}))\geq (\varphi({\bf a}),\eta({\bf a}))$ is equivalent
to $\WW_{\varphi({\bf b}),\eta({\bf
b})}\subseteq\ZZ_{\varphi({\bf a}),\eta({\bf a})}=\overline{\WW}_{\varphi({\bf a}),\eta({\bf a})}$. The
proof is complete.
\end{proof}

\

\

\section{The canonical subgroup}

\subsection{Some admissible open subsets of $\Xrig$ and $\Yrig$}
Let $K$ be a discretely valued complete subfield of $\CC_P$ with
uniformizer $\varpi$ and ring of integers $\ok$. Let $\gerZ$ be an
admissible formal scheme over $\ok$ and $\gerZ_\rig$ the rigid
analytic space over $K$ associated to it \`a la Raynaud. Let $C$ be
a closed subscheme of $\overline{\gerZ}$, the special fibre of $\gerZ$. Let
$|\varpi|\leq\lambda\leq 1$ be an element of $p^\QQ$.

One can define $[C]_{\leq\lambda}$, {\emph{the closed tube of $C$ of radius $\lambda$}},
as in \cite[1.1.8]{Berthelot}. It is  a quasi-compact admissible
open of $\gerZ_\rig$ defined as follows: if $C$
is defined by the vanishing of functions $f_1,\dots,f_m$ in $\calO(\overline{\gerZ})$ with lifts
$\tilde{f}_1,\dots,\tilde{f}_m$ in $\calO(\gerZ)$, then
$[C]_{\leq\lambda}$ is defined by the inequalities
$|\tilde{f}_i|_{\rm sup}\leq\lambda$ for $1\leq i\leq m$. Note that if $\lambda=1$ this gives the entire $\gerZ_\rig$.  Under the above
assumptions on $\lambda$, this definition is independent of the choice of $f_i$'s and their lifts. In the general case, $[C]_{\leq\lambda}$
can be constructed in the same way by using local generators for the
ideal of $C$ and gluing these local constructions.  The gluing is
possible in view of the independence of the local construction of
the set of generators of the ideal of $C$. This independence also
implies the following: if $\Qbar$ is a closed point of $\overline{\gerZ}$, then
$\spe^{-1}(\Qbar)\cap[C]_{\leq \lambda}$ is the locus where
$|\tilde{g}_i|_{\rm sup}\leq \lambda$, where $\tilde{g}_i$'s are any
set of functions in $ \calO_{\gerZ,\Qbar}$ whose reductions, $g_i$,
define the closed subscheme
$C\cap\Spf(\widehat{\calO}_{\overline{\gerZ},\Qbar})$.

If $C$ is a Cartier divisor on  $\overline{\gerZ}$, then
one can similarly define $[C]_{\geq\lambda}$, which is a quasi-compact
admissible open in $\gerZ_\rig$: write $C$ locally as the vanishing
of a function $f$ which lifts to $\tilde{f}\in \calO(\gerZ)$, and
define $[C]_{\geq\lambda}$ locally by the inequality
$|\tilde{f}|_{\rm sup}\geq \lambda$. This will be independent of the
choice of $f$ for $\lambda$ as above, and that allows  gluing the
  local constructions. If $\Qbar$ is a closed point of $\overline{\gerZ}$,
then $\spe^{-1}(\Qbar)\cap[C]_{\geq \lambda}$ is the locus where
$|\tilde{g}|_{\rm sup}\geq \lambda$, where $\tilde{g}$ is any
function in $\calO_{\gerZ,\Qbar}$ whose reduction defines the
closed subscheme $C\cap\Spf(\widehat{\calO}_{\overline{\gerZ},\Qbar})$.

\begin{lem} Let $\beta\in\BB$.  Let $a\in [0,1]\cap \QQ$.
\begin{enumerate}
\item  $[\UU^+_{\sigma\circ\beta}]_{\leq(1/p)^{1-a}}$ is a quasi-compact
admissible open in $\Yrig$ whose points are
\[
\{Q\in\Yrig: \nu_\beta(Q)\leq a\}.
\]
\item   $[\ZZ_{\{\beta\}}]_{\geq(1/p)^a}$ is a quasi-compact admissible open in $\Xrig$ whose points are
\[
\{P\in\Xrig: \nu_\beta(P)\leq a\}.
\]
\item Similarly, $[\ZZ_{\{\beta\}}]_{\leq(1/p)^a}$ is a quasi-compact admissible
open in $\Xrig$ whose points are
\[
\{P\in\Xrig: \nu_\beta(P)\geq a\}.
\]
\end{enumerate}
\end{lem}

\begin{proof}  It suffices to calculate the points of $[\UU^+_{\sigma\circ\beta}]_{\leq(1/p)^{1-a}}\cap\; \spe^{-1}(\Qbar)$ for every closed point $\Qbar\in\cYbar$. If $a=1$, then $[\UU^+_{\sigma\circ\beta}]_{\leq(1/p)^{1-a}}=\Yrig$ and
the result follows. Assume $a<1$. Let $\Qbar$ be a closed point in
$\cYbar$. Then, by  Lemma \ref{lemma: extension of U phi and V eta}, $\UU^+_{\sigma\circ\beta}\cap\Spf(\OhatQbar)$ is given
by the vanishing of
\[
\begin{cases}
y_\beta \qquad & \beta \in I(\Qbar),\\
1 \qquad & \beta \not\in \ell(\varphi(\Qbar)),\\
0\qquad & \beta\in \ell(\varphi(\Qbar))-\eta(\Qbar).
\end{cases}
\]

In the first case, $\spe^{-1}(\Qbar) \cap
[\UU^+_{\sigma\circ\beta}]_{\leq(1/p)^{1-a}}$ is given by the
inequality: $|y_\beta(Q)|\leq (1/p)^{1-a}$, or, equivalently,
$\nu_\beta(Q)=\nu(x_\beta(Q))\leq a$. In the second case,  $\spe^{-1}(\Qbar) \cap
[\UU^+_{\sigma\circ\beta}]_{\leq(1/p)^{1-a}}$ is empty. The result follows, as in this case,
$\beta\in\eta(\Qbar)-I(\Qbar)$, and hence, $\nu_\beta(Q)=1>a$. In the last case, $\spe^{-1}(\Qbar) \subseteq
[\UU^+_{\sigma\circ\beta}]_{\leq(1/p)^{1-a}}$. The result again
follows, as $\beta\not\in\eta(\Qbar)$ implies that
$\nu_\beta(Q)=0\leq a$.

Now we prove part (2).  Again, the case $a=1$ is immediate, and we assume $a<1$. Let
$\Pbar$ be a closed point in $\cXbar$.  The stratum $\ZZ_{\{\beta\}}$ is a
divisor on $\cXbar$. In fact, $\ZZ_{\{\beta\}}\cap\Spf(\OhatPbar)$ is
given by the vanishing of
\[
\begin{cases}
t_\beta \qquad & \beta \in \tau(\Pbar),\\
1 \qquad & \beta \not\in \tau(\Pbar).\\
\end{cases}
\]

In the first case, $\spe^{-1}(\Pbar) \cap [\ZZ_{\{\beta\}}]_{\geq(1/p)^a}$
is given by the inequality: $|t_\beta(P)|\geq (1/p)^a$, or,
equivalently, $\nu_\beta(P)=\nu(t_\beta(P))\leq a$. In the second case, $\spe^{-1}(\Pbar) \subseteq
[\ZZ_{\{\beta\}}]_{\geq(1/p)^a}$. The result follows, as in this case we always have $\nu_\beta(P)=0\leq a$.

The remaining statement can be proved in the same way. \end{proof}

\begin{cor} \label{corollary: admissible} Let ${\bf a}=(a_\beta)_{\beta\in\BB}$ and ${\bf b}=(b_\beta)_{\beta\in\BB}$
both belong to $\Theta\cap\QQ^\BB$. Assume  that for each $\beta$, we
have $a_\beta \leq b_\beta$. There is a quasi-compact admissible
open $\Yrig[{\bf a},{\bf b}]$ of $\Yrig$ whose points are
\[
\{Q\in\Yrig: a_\beta\leq \nu_\beta(Q)\leq b_\beta,\; \forall\beta\in\BB\}.
\]
Similarly, there exits a quasi-compact admissible open $\Xrig[{\bf
a},{\bf b}]$ of $\Xrig$ whose points are
\[
\{P\in\Xrig: a_\beta\leq \nu_\beta(P)\leq b_\beta,\; \forall \beta\in\BB\}.
\]
\end{cor}

\begin{proof} Define
\[
 \Yrig[{\bf a},{\bf b}]=\bigcap_{\beta\in\BB}
[\UU^+_{\sigma\circ\beta}]_{\leq(1/p)^{1-b_\beta}} \bigcap
\bigcap_{\beta\in\BB}
w([\UU^+_{\sigma\circ\beta}]_{\leq(1/p)^{a_\beta}}).
\]
This is a finite intersection of quasi-compact admissible opens in a
quasi-separated rigid analytic space, and hence is a quasi-compact
admissible open of $\Yrig$ with the desired property. Similarly,
define
\[
\Xrig[{\bf a},{\bf b}]=\bigcap_{\beta\in\BB}
[\ZZ_{\{\beta\}}]_{\geq(1/p)^{b_\beta}} \bigcap \bigcap_{\beta\in\BB}
[\ZZ_{\{\beta\}}]_{\leq(1/p)^{a_\beta}},
\]
which is a quasi-compact admissible open of $\Xrig$ with the desired
property.
\end{proof}

\begin{prop}\label{prop: admissible covering} Let $\Gamma$ be a subset of $\Theta$ with the property that if $(a_\beta)_{\beta\in\BB} \in \Gamma$ and $b_\beta\leq a_\beta$ for all $\beta\in\BB$, then $(b_\beta)_{\beta\in\BB}\in \Gamma$.
\begin{enumerate}
\item There is an admissible open subset  $\Yrig \Gamma$ of $\Yrig$ whose points are $\{Q\in\Yrig:  \nu_\gerY(Q)\in \Gamma\}$.  The collection $\{\Yrig[{\bf 0},{\bf a}]: {\bf a}\in \Gamma\}$ is an admissible covering of  $\Yrig \Gamma$.

\item There is an admissible open subset  $\Xrig \Gamma$ of $\Xrig$ whose points are $\{P\in\Xrig:  \nu_\gerX(P)\in \Gamma\}$.  The collection $\{\Xrig[{\bf 0},{\bf a}]: {\bf a}\in \Gamma\}$ is an admissible covering of  $\Xrig \Gamma$.
\end{enumerate}
 \end{prop}

\begin{proof} It is enough to show that the collection of quasi-compact opens  $\{\Yrig[{\bf 0},{\bf a}]: {\bf a}\in \Gamma\}$ is an admissible covering: that is, for any affinoid algebra $A$, and  any $f\colon \Spm(A)\arr \Yrig$ whose image lies in the union of the subsets in this collection, the pull-back covering of $\Spm(A)$ has a finite sub-covering. This follows from a standard application of the maximum modulus principle. The second statement follows in the same way.
\end{proof}

\

\subsection{The section on the ordinary locus} By Corollary \ref{corollary: admissible}, we have the following
admissible opens
\begin{align*}
\Xord:&=\Xrig[{\bf 0},{\bf 0}]=\{P\in\Xrig:\nu_\gerX(P)={\bf 0}\},\\
\Yoord:&=\Yrig[{\bf 0},{\bf 0}]=\{Q\in\Yrig:\nu_\gerY(Q)={\bf 0}\},
\end{align*}
of $\Xrig$ and $\Yrig$, respectively.  By Theorem \ref{theroem: theorem of cube}, and Definition \ref{definition: ordinary strata with cusps}, we have:
\begin{align*}
\Xord&=\spe^{-1}(\cXbar^{\rm ord})=\spe^{-1}(\WW_\emptyset),\\
\Yoord&=\spe^{-1}(\cYbar^{\rm ord}_F)=\spe^{-1}(\WW_{\BB,\emptyset}).
\end{align*}
Let $\Yord$ be $\spe^{-1}(\cYbar^{\rm ord})=\pi^{-1}(\spe^{-1}(\cXbar^{\rm ord}))=\pi^{-1}(\Xord)$.

 \begin{prop}\label{proposition: ordinary section}
There is a section,
\[
s^\circ\colon\Xord \arr \Yord,
\]
to $\pi\colon\Yord \arr \Xord$, whose image is $\Yoord$.
\end{prop}
 \begin{proof} Let $\gerX^\circ$  be the open formal subscheme of $\gerX$ supported on the ordinary locus; similarly, define $\gerY^\circ$. The special fibre of $\gerX^\circ$ is $\cXbar^{\rm ord}$ and its generic fibre is $\Xord$; similar results hold for $\gerY^\circ$.  The morphism $\pi:\gerY^\circ \arr \gerX^\circ$ is proper and quasi-finite, and hence finite.  We will show that $\pi\vert_{\Yoord}\colon \Yoord \arr \Xord$ is an isomorphism. It is enough to prove this locally on the base.

 Let $U=\Spf(A)$ be an affine open formal subscheme of $\gerX^\circ$.  Let $\pi^{-1}(U)=\Spf(B)\subset \gerY^\circ$. The
morphism $\pi$ induces a finite  ring homomorphism $\pi^*\colon A\arr B$.  Since $\cYbar^{\rm ord}_F$ is a union of connected components of
$\cYbar^{\rm ord}$, we have a decomposition $\cYbar^{\rm ord}=\cYbar^{\rm ord}_F \cup (\cYbar^{\rm ord}-\cYbar^{\rm ord}_F)$ leading to a decomposition
\[
B\otimes \kappa=e_1(B
\otimes \kappa) \oplus e_2(B\otimes\kappa),
\]
where $e_1$ and $e_2$
are idempotents satisfying $e_1+e_2=1$, and $\Spec(e_1(B\otimes \kappa))=\Spec(B\otimes \kappa)\cap\cYbar^{\rm
ord}_F$. Using Hensel's lemma for the polynomial $x^2-x$, we can
lift these idempotents to idempotents $\tilde{e}_1$ and
$\tilde{e}_2$ in $B$. The composite homomorphism $A\otimes\kappa \arr
B\otimes\kappa \arr e_1(B\otimes\kappa)$ is an isomorphism by the
existence of the Kernel-of-Frobenius section.  Therefore, the
composite homomorphism $A \arr B \arr \tilde{e}_1 B$ is a
finite morphism whose reduction modulo $p$ is an isomorphism, and whose generic fibre is finite-flat (using an argument as in Lemma \ref{lemma: pi is flat on compactified horizontal strata}). Since both the generic fibre and special fibre of this map are flat, it follows that it is finite flat; having a reduction modulo $p$ which is an isomorphism, it follows that it is an isomorphism.  Let $V=\Spf(\tilde{e}_1B)$. We have shown that $\pi\vert_{V_\rig}: V_\rig \arr U_\rig$ is an isomorphism. To finish the proof, we need to prove that $V_\rig=\pi^{_1}(U_\rig) \cap \Yoord$.  But this is true, since $V_\rig$ is the region in $\Spf(B)_\rig=\pi^{-1}(U_\rig)$ which specializes to $\Spec(e_1(B\otimes \kappa))=\Spec(B\otimes \kappa)\cap \cYbar^{\rm ord}_F$.
\end{proof}

\

\subsection{The main theorem}\label{subsec: main
theorem}
Let $\Gamma=\{{\bf a}\in\Theta: a_\beta+pa_{\sigma^{-1}\circ\beta}<p\
{\rm for\ all}\ \beta\in\BB\}$. Then, by Proposition  \ref{prop:
admissible covering}, we have the following admissible open sets:
\begin{align*}
\calU&:=\Xrig\Gamma=\{P\in\Xrig:  \nu_\beta(P)+p\nu_{\sigma^{-1}\circ\beta}(P)<p, \; \forall  \beta\in\BB\},\\
\calV&:=\Yrig\Gamma=\{Q\in\Yrig:  \nu_\beta(Q)+p\nu_{\sigma^{-1}\circ\beta}(Q)<p,
\; \forall  \beta\in\BB\}.
\end{align*}
 Recall that  $\BB={\rm
Emb}(L,\mathbb{Q}_\kappa)=\textstyle\coprod_{\mathfrak{p}}
\mathbb{B}_\mathfrak{p}$, where $\mathfrak{p}$ runs over prime
ideals  of $\mathcal{O}_L$ dividing $p$. For $\gerp|p$, let
\begin{align*}
\calV_\gerp&:=\{Q\in
\Yrig:\nu_\beta(Q)+p\nu_{\sigma^{-1}\circ\beta}(Q)<p\quad \forall
\beta \in \BB_\gerp\},\\
\calW_\gerp&:=\{Q\in
\Yrig:\nu_\beta(Q)+p\nu_{\sigma^{-1}\circ\beta}(Q)>p\quad \forall
\beta \in \BB_\gerp\}.
\end{align*}
By Proposition \ref{prop: admissible covering} these are admissible
open sets. Note that $\calV=\cap_{\gerp|p} \calV_\gerp$. Let
\[
\calW :=\bigcup_{\emptyset\neq S\subseteq\{\mathfrak{p}\vert p\}}
\left[ \bigcap_{\mathfrak{p}\in S} \mathcal{W}_\mathfrak{p}
\cap \bigcap_{\mathfrak{p}\not\in S} \mathcal{V}_\mathfrak{p}
\right].
\]

We now prove our main theorem on the existence of canonical
subgroups of abelian varieties with real multiplication.

\begin{thm}[The Canonical Subgroup Theorem] \label{theorem: canonical subgroup} Let notation be as above.
\begin{enumerate}
\item $\pi(\calV)=\calU$.
\item There is a section $s^\dagger\colon\calU \arr \calV$ extending
$s^\circ\colon \Xord \arr \Yoord$.
\end{enumerate}

\end{thm}

\begin{proof}
Before we prove the theorem, we need a number of results.

\begin{lem}\label{lemma: key lemma in terms of valuations} Let $Q\in\Yrig$,  $P=\pi(Q)$, and  $\beta \in\BB$. Then,
$\beta\in\varphi(\Qbar)\cap\eta(\Qbar)$ if and only if
$\nu_\beta(Q)\neq 0$ and $\nu_{\sigma^{-1}\circ\beta}(Q)\neq 1$. In
that case, $Q$ has non-cuspidal reduction; choose parameters $\{t_\beta\}_{\beta\in\BB}$ at $\Pbar$
as in (\ref{equation: local def ring of X at P bar}), and parameters $\{x_\beta,y_\beta\}_{\beta\in I(\Qbar)}$ at
$\Qbar$ as in (\ref{equation: local deformation ring at Q bar}). We have:

\[\pi^*(t_\beta)\underset{\mod p\OhatQ}{\equiv}\begin{cases}
ux_\beta+vy_{\sigma^{-1}\circ\beta}^p & {\rm if}\ \nu_\beta(Q)\neq
1,\quad
\nu_{\sigma^{-1}\circ\beta}(Q)\neq 0;\\
ux_\beta & {\rm if}\ \nu_\beta(Q)\neq 1,\quad
\nu_{\sigma^{-1}\circ\beta}(Q)=0;\\
vy_{\sigma^{-1}\circ\beta}^p & {\rm if}\ \nu_\beta(Q)=1,\quad
\nu_{\sigma^{-1}\circ\beta}(Q)\neq0;\\
0 & {\rm if}\ \nu_\beta(Q)=1,\quad \nu_{\sigma^{-1}\circ\beta}(Q)=0.
\end{cases}
\]
In the formulas above, $u,v$ are units in $\OhatQ$. It follows that, respectively,
\[
\nu_\beta(P)=
\begin{cases}
\nu_\beta(ux_\beta(Q)+vy^p_{\sigma^{-1}\circ\beta}(Q)),\\
\nu_\beta(Q),\\
p(1-\nu_{\sigma^{-1}\circ\beta}(Q)),\\
1.
\end{cases}
\]
\end{lem}

\begin{proof}
This follows immediately from the Key Lemma \ref{lemma: key lemma }. The
various cases are obtained by reinterpreting  the conditions
appearing in the lemma in terms of valuations, using directly the
definition of valuations.
\end{proof}

\begin{rmk}\label{remark: automatic critical index} In Lemma \ref{lemma: key lemma in terms of valuations}, it automatically follows that $\{\beta,\sigma^{-1}\circ\beta\}\subseteq I(\Qbar)$  in the first case,  $\beta\in I(\Qbar)$ in the second case, and  $\sigma^{-1}\circ\beta\in I(\Qbar)$ in the third case.

\end{rmk}

\begin{lem}\label{lemma: valuations under pi}
Let $\gerp|p$ and $\beta\in\BB_\gerp$. Let $Q\in\Yrig$, and
$P=\pi(Q)$.
\begin{enumerate}
\item If $Q\in\calV_\gerp$ then $\nu_\beta(P)=\nu_\beta(Q)$.
\item If $Q\in\calW_\gerp$ then $\nu_\beta(P)=p(1-\nu_{\sigma^{-1}\circ\beta}(Q))$.
\end{enumerate}
\end{lem}

\begin{proof} First we deal with the case $Q\in\calV_\gerp$. Then, either $\nu_\beta(Q)=0$ or $0<\nu_\beta(Q)<1$. If $\nu_\beta(Q)=0$, then, by definition,
$\beta\not\in\eta(\Qbar)$. Since  $Q\in\calV_\gerp$ and $\sigma^{-1}\circ\beta\in\BB_\gerp$, it follows that $\nu_{\sigma^{-1}\circ\beta}(Q)\neq 1$, and, by definition, $\sigma^{-1}\circ\beta \in (\eta(\Qbar)-I(\Qbar))^c=\ell(\varphi(\Qbar))$.
 Therefore, $\beta\in\varphi(\Qbar)-\eta(\Qbar)$. Corollary \ref{corollary: type
and phi eta} tells us that $\beta\not\in\tau(\Pbar)$, and hence,
$\nu_\beta(P)=0=\nu_\beta(Q)$.

Now assume $0<\nu_\beta(Q)<1$, and so $Q$ has non-cuspidal reduction. Since $Q\in\calV_\gerp$, we have
$\nu_{\sigma^{-1}\circ\beta}(Q)\neq1$, and hence, by Lemma
\ref{lemma: key lemma in terms of valuations}, $\beta\in
\varphi(\Qbar)\cap\eta(\Qbar)$. There are two cases to consider:

\begin{itemize}
\item $\nu_{\sigma^{-1}\circ\beta}(Q)=0$. In this case,
Lemma \ref{lemma: key lemma in terms of valuations}  implies that
$\nu_\beta(P)=\nu_\beta(Q)$.

\item  $\nu_{\sigma^{-1}\circ\beta}(Q)\neq 0$. In this
case,  by Lemma \ref{lemma: key lemma in terms of valuations}, we
have
\[
\nu_\beta(P)=\nu(ux_\beta(Q)+vy_{\sigma^{-1}\circ\beta}^p(Q)).
\]
Since $Q\in\calV_\gerp$, we have
$\nu_\beta(Q)<p(1-\nu_{\sigma^{-1}\circ\beta}(Q))$ and hence, by
Remark \ref{remark: automatic critical index}, we have
$\nu(ux_\beta(Q))<\nu(vy^p_{\sigma^{-1}\circ\beta}(Q))$. It follows
that $\nu_\beta(P)=\nu(ux_\beta(Q))=\nu_\beta(Q)$.
\end{itemize}

\id Now we deal with the case $Q\in\calW_\gerp$. It follows that
$\nu_\beta(Q)>0$ for any $\beta\in \BB_\gerp$. Hence, either
$\nu_{\sigma^{-1}\circ\beta}(Q)=1$ or
$0<\nu_{\sigma^{-1}\circ\beta}(Q)<1$. If
$\nu_{\sigma^{-1}\circ\beta}(Q)=1$, then $\sigma^{-1}\circ\beta\in
\eta(\Qbar)-I(\Qbar)=\ell(\varphi(\Qbar))^c$, and since $\nu_\beta(Q)>0$, $\beta \in
\eta(\Qbar)$. It follows that   $\beta \in
\eta(\Qbar)-\varphi(\Qbar)$ and hence, by Corollary \ref{corollary:
type and phi eta}, $\beta\not\in\tau(\Pbar)$. So,
$\nu_\beta(P)=0=p(1-\nu_{\sigma^{-1}\circ\beta}(Q))$ as desired.

Now suppose $0<\nu_{\sigma^{-1}\circ\beta}(Q)<1$, and, in particular, that $Q$ has non-cuspidal reduction. There are two
cases:

\begin{itemize}
 \item $\nu_\beta(Q)=1$.  Then, by Lemma \ref{lemma: key
lemma in terms of valuations},  we have $\nu_\beta(P)=p(1-\nu_{\sigma^{-1}\circ\beta}(Q))$.

\item $\nu_\beta(Q)\neq 1$. In this case,  by Lemma
\ref{lemma: key lemma in terms of valuations}, we have
\[
\nu_\beta(P)=\nu(ux_\beta(Q)+vy_{\sigma^{-1}\circ\beta}^p(Q)).
\]
Since $Q\in\calW_\gerp$, we have
$\nu_\beta(Q)>p(1-\nu_{\sigma^{-1}\circ\beta}(Q))$ and hence, by
Remark \ref{remark: automatic critical index}, we have
$\nu(ux_\beta(Q))>\nu(vy^p_{\sigma^{-1}\circ\beta}(Q))$. It follows
that
$\nu_\beta(P)=\nu(vy^p_{\sigma^{-1}\circ\beta}(Q))=p(1-\nu_{\sigma^{-1}\circ\beta}(Q))$.
\end{itemize}
\end{proof}

\begin{cor}
$\pi^{-1}(\calU) \supseteq \calV \cup \calW$.
\end{cor}

\begin{lem}\label{lemma: pi on the too singular locus} Let $\beta\in\BB$, $Q\in\Yrig$, and $P=\pi(Q)$. Suppose
\begin{align*}
(\dagger)\;\qquad &\nu_\beta(Q)+p\nu_{\sigma^{-1}\circ\beta}(Q)\leq p,\\
(\dagger\dagger)\qquad& \nu_{\sigma\circ\beta}(Q)+p\nu_\beta(Q)\geq p.
\end{align*}
Then, $P\not\in \calU$.
\end{lem}

\begin{proof} The conditions imply that $\nu_{\sigma^{-1}\circ\beta}(Q)\neq 1$,  $\nu_\beta(Q)\neq 0$, and, in particular, $Q$ has non-cuspidal reduction. Therefore, by Lemma  \ref{lemma: key lemma in terms of valuations},
$\beta\in\varphi(\Qbar)\cap\eta(\Qbar)$. We distinguish the four cases as
in Lemma \ref{lemma: key lemma in terms of valuations}, and the Key
Lemma.

 \

\noindent {\bf Case A}:  $\nu_\beta(Q)\neq 1$ and
$\nu_{\sigma^{-1}\circ\beta}(Q)\neq 0$. In this case,
$\sigma\circ\beta \in \varphi(\Qbar)$,
$\sigma^{-1}\circ\beta\in\eta(\Qbar)$, and
$\nu_\beta(P)=\nu(ux_\beta(Q)+vy^p_{\sigma^{-1}\circ\beta}(Q))$ for
some units $u,v\in\OhatQ$. Also, by  Remark \ref{remark: automatic
critical index}, we have $\{\beta,\sigma^{-1}\circ\beta\}\subseteq
I(\Qbar)$. Hence, Equation ($\dagger$)   can be rephrased as
$\nu(ux_\beta(Q))\leq \nu(vy^p_{\sigma^{-1}\circ\beta}(Q))$. It
follows that,  in this case, $\nu_\beta(P)\geq
\nu(ux_\beta(Q))=\nu_\beta(Q)$. Since $\nu_\beta(Q)\neq1$, Equation ($\dagger\dagger$) implies that $\nu_{\sigma\circ\beta}(Q)>0$, and hence, $\sigma\circ\beta\in \eta(\Qbar)$.
Since also $\sigma\circ\beta\in\varphi(\Qbar)$, we can  apply Lemma
\ref{lemma: key lemma in terms of valuations} at $\sigma\circ\beta$.
There are two cases to consider:

\begin{itemize}

\item {\bf Case A.I}:  $\nu_{\sigma\circ\beta}(Q)=1$,
$\nu_\beta(Q)\neq 0$.  In this case, we have
$\nu_{\sigma\circ\beta}(P)=p(1-\nu_\beta(Q))$.  Therefore,
\[
\nu_{\sigma\circ\beta}(P)+p\nu_\beta(P)\geq
p(1-\nu_\beta(Q))+p\nu_\beta(Q)= p,
\]
and hence, $P\not\in\calU$.

\item {\bf Case A.II}:     $\nu_{\sigma\circ\beta}(Q)\neq1$,
$\nu_\beta(Q)\neq 0$. By Remark \ref{remark: automatic critical
index}, in this case we have $\sigma\circ\beta\in I(\Qbar)$. We also
know that  $\nu_{\sigma\circ\beta}(P)=\nu(u^\prime
x_{\sigma\circ\beta}(Q)+v^\prime y^p_\beta(Q))$ for some units
$u^\prime,v^\prime \in \OhatQ$.  Equation ($\dagger\dagger$) is equivalent
to $\nu(u^\prime x_{\sigma\circ\beta}(Q))\geq \nu(v^\prime
y_\beta^p(Q))$. It follows that $\nu_{\sigma\circ\beta}(P) \geq
\nu(v^\prime y^p_\beta(Q))=p\nu(y_\beta(Q))$. Therefore,
\[
\nu_{\sigma\circ\beta}(P)+p\nu_\beta(P)\geq
p\nu(y_\beta(Q))+p\nu_\beta(Q)=p\nu(y_\beta(Q))+p\nu(x_\beta(Q))=p,
\]
and hence, $P\not\in\calU$.

\end{itemize}

\noindent {\bf Case B}: $\nu_\beta(Q)\neq 1$ and
$\nu_{\sigma^{-1}\circ\beta}(Q)= 0$. In this case, $\sigma\circ\beta
\in \varphi(\Qbar)$, and $\sigma^{-1}\circ\beta\not\in\eta(\Qbar)$,
and $\nu_\beta(P)=\nu_\beta(Q)$. Equation ($\dagger\dagger$) implies that
$\nu_\beta(Q)\neq 0$. By Remark \ref{remark: automatic critical
index}, we have $\beta\in I(\Qbar)$. Exactly as in Case A, we deduce that
$\sigma\circ\beta\in\varphi(\Qbar)\cap\eta(\Qbar)$. Applying Lemma
\ref{lemma: key lemma in terms of valuations} at $\sigma\circ\beta$
we consider two cases:

\begin{itemize}
\item {\bf Case B.I}:  $\nu_{\sigma\circ\beta}(Q)=1$,
$\nu_\beta(Q)\neq 0$.  In this case, we have
$\nu_{\sigma\circ\beta}(P)=p(1-\nu_\beta(Q))$.  Therefore,
\[
\nu_{\sigma\circ\beta}(P)+p\nu_\beta(P)=
p(1-\nu_\beta(Q))+p\nu_\beta(Q)=p,
\]
and hence, $P\not\in\calU$.

\item {\bf Case B.II}: $\nu_{\sigma\circ\beta}(Q)\neq1$,
$\nu_\beta(Q)\neq 0$. In this case,
$\nu_{\sigma\circ\beta}(P)=\nu(u^\prime
x_{\sigma\circ\beta}(Q)+v^\prime y^p_\beta(Q))$ for some units
$u^\prime,v^\prime \in \OhatQ$. By Remark \ref{remark: automatic
critical index}, we have $\sigma\circ\beta\in I(\Qbar)$. This
implies that Equation ($\dagger\dagger$) is equivalent to $\nu(u^\prime
x_{\sigma\circ\beta}(Q))\geq \nu(v^\prime y_\beta^p(Q))$. It follows
that $\nu_{\sigma\circ\beta}(P) \geq \nu(v^\prime
y^p_\beta(Q))=p\nu(y_\beta(Q))$. Therefore
\[
\nu_{\sigma\circ\beta}(P)+p\nu_\beta(P)\geq
p\nu(y_\beta(Q))+p\nu_\beta(Q)=p\nu(y_\beta(Q))+p\nu(x_\beta(Q))=p,
\]
and hence, $P\not\in\calU$.
\end{itemize}

\id {\bf Case C}: $\nu_\beta(Q)=1$, $\nu_{\sigma^{-1}\circ\beta}(Q)\neq 0$. By Remark \ref{remark: automatic critical index},
we have $\sigma^{-1}\circ\beta \in I(\Qbar)$. In this case,
$\nu_{\beta}(P)= p(1-\nu_{\sigma^{-1}\circ\beta}(Q))$. Since $\nu_\beta(Q)=1$, Equation ($\dagger$) implies that
$\nu_{\sigma^{-1}\circ\beta}(Q)\leq (p-1)/p$. Hence,
\[
\nu_\beta(P)= p(1-\nu_{\sigma^{-1}\circ\beta}(Q))\geq
p(1-(p-1)/p)=1,
\]
which implies that  $\nu_{\sigma\circ\beta}(P)+p\nu_{\beta}(P)\geq
p$, and $P\not\in \calU$.

\

\id{\bf Case D}: $\nu_\beta(Q)=1$, $\nu_{\sigma^{-1}\circ\beta}(Q)=0$.
In this case, $\nu_\beta(P)=1$, hence,
$\nu_{\sigma\circ\beta}(P)+p\nu_{\beta}(P)\geq p$ and $P\not\in
\calU$.
\end{proof}

\begin{cor}\label{corollary: pi inverse of U}
$\pi^{-1}(\calU)=\calV \cup \calW.$
\end{cor}

\begin{proof}
Let $Q\in \Yrig-(\calV \cup \calW)$. For simplicity, we define $\lambda_{\beta}=\nu_\beta(Q)+p\nu_{\sigma^{-1}\circ\beta}(Q)$.
By definition, there is $\gerp|p$ such that $Q\not\in \calV_\gerp \cup \calW_\gerp$.
Since $Q\not\in \calW_\gerp$,   there exists $\gamma\in\BB_\gerp$, such that $\lambda_\gamma\leq p$.
Since $Q\not\in \calV_\gerp$,   there exists $i\geq 1$, such that $\lambda_{\sigma^i\circ\gamma}\geq p$. Let $i$ be the minimal positive integer with this property.
Let $\beta=\sigma^{i-1}\circ\gamma$. If $i=1$, then $\lambda_\beta\leq p$ and $\lambda_{\sigma\circ\beta}\geq p$.
If $i>1$, then by minimality of $i$, we find that $\lambda_\beta< p$ and $\lambda_{\sigma\circ\beta}\geq p$. At any rate, Equations ($\dagger$) and ($\dagger\dagger$) of Lemma \ref{lemma: pi on the too singular locus} hold for $Q$, and so $\pi(Q)\not\in \calU$.
\end{proof}

Let
\[
\Gamma=\{{\bf a}\in\Theta\cap\QQ^\BB:
a_\beta+pa_{\sigma^{-1}\circ\beta}<p\ {\rm for\ all}\
\beta\in\BB\}.
\]
 For ${\bf a}\in \Gamma$ and $S \subseteq \{\gerp\vert p\}$, let $\II^S_{\bf a} = [{\bf c}, {\bf d}]$, where ${\bf c} = (c_\beta), {\bf d} = (d_\beta)$ and
\[ [c_\beta, d_\beta] = \begin{cases} [1-\frac{1}{p}\cdot a_{\sigma\circ\beta},1] & \beta \in \BB_\gerp, \gerp \in S, \\
[0, a_\beta] & \beta \in \BB_\gerp, \gerp \not\in S. \end{cases}\]
By Corollary \ref{corollary: admissible}, $\Yrig\II^S_{\bf a}$ is a quasi-compact admissible open of $\Yrig$.

\begin{cor}\label{corollary: circles} Let notation be as above. We have
\[
\pi^{-1}(\Xrig[{\bf 0},{\bf a}])=\Yrig[{\bf 0},{\bf a}] \cup \bigcup_{\emptyset \neq S \subseteq \{ \gerp\vert p\}}
\Yrig\II^S_{\bf a}.
\]
\end{cor}

\begin{proof}
This follows from Corollary \ref{corollary: pi inverse of U} and
Lemma \ref{lemma: valuations under pi}.
\end{proof}

We continue with the proof of Theorem \ref{theorem: canonical
subgroup}.  Let ${\bf a}\in \Gamma$. For simplicity, we denote  by $\calR$ the quasi-compact admissible open $\cup_{\emptyset \neq S \subseteq \{ \gerp\vert p\}}
\Yrig\II^S_{\bf a}$. By Corollary \ref{corollary:
circles}, the morphism
\[
\pi\colon \Yrig[{\bf 0},{\bf a}] \cup \calR \arr \Xrig[{\bf 0},{\bf a}]\] is finite and  flat of the same degree as
that of $\pi\colon\Yrig\arr\Xrig$, that is, $\prod_{\gerp|p}
(p^{f(\gerp)}+1)$. From the definition, we see that $\Yrig[{\bf
0},{\bf a}] \cap \calR=\emptyset$, and since
both $\Yrig[{\bf 0},{\bf a}]$ and $\calR$
are quasi-compact, and $\Yrig$ is quasi-separated, we see that they
provide an admissible disjoint covering of $\pi^{-1}(\Xrig[{\bf
0},{\bf a}])$. In particular, $\pi\colon\Yrig[{\bf 0},{\bf a}] \arr
\Xrig[{\bf 0},{\bf a}]$ is a finite-flat morphism. We need a lemma.

\begin{lem}\label{lemma: degree is one} The morphism $\pi\colon\Yrig[{\bf 0},{\bf a}] \arr \Xrig[{\bf 0},{\bf
a}]$ is an isomorphism.
\end{lem}

\begin{proof}
Since $\pi\colon\Yrig[{\bf 0},{\bf a}] \arr \Xrig[{\bf 0},{\bf a}]$
is finite-flat, to prove the lemma it is enough to show that
it has constant degree $1$. It is enough to calculate the degree of this
morphism over an admissible open inside every connected component of $\Xrig[{\bf 0},{\bf a}]$.
We use $\Xord =\Xrig[{\bf 0},{\bf 0}]\subset \Xrig[{\bf 0},{\bf
a}]$, which intersects every connected component of $\Xrig[{\bf 0},{\bf a}]$, and whose inverse image under $\pi$ inside $\Yrig[{\bf 0},{\bf
a}]$ is $\Yoord =\Yrig[{\bf 0},{\bf 0}]$ by Corollary
\ref{corollary: circles}. By  Proposition \ref{proposition: ordinary
section}, the morphism $\pi\colon\Yoord\arr\Xord$ is an isomorphism
and hence has degree $1$. This proves the lemma.
\end{proof}

By Lemma \ref{lemma: degree is one}, for any ${\bf a}\in \Gamma$
there is a section $s^{\bf a}\colon\Xrig[{\bf 0},{\bf a}]\arr\calV$
to $\pi$ whose image is $\Yrig[{\bf 0},{\bf a}]$ and which extends
$s^\circ=s^{\bf 0}$. Furthermore, the sections $\{s^{\bf a}\}$ are
compatible on intersections. Since by Proposition \ref{prop:
admissible covering} the collection $\{\Xrig[{\bf 0},{\bf
a}]\}_{{\bf a}\in\Gamma}$ admissibly covers $\calU$, we
conclude that there is a section
\[
s^\dagger\colon \calU \arr \calV
\]
to $\pi$ which extends $s^\circ$. This completes the proof of Theorem \ref{theorem:
canonical subgroup}.
\end{proof}

 \

\subsection{Properties of the canonical subgroup}\label{subsection: Properties of the canonical subgroup}

\begin{dfn}
Let $K\supseteq W(\kappa)$ be a completely valued field. Let
$\underline{A}$ be an abelian variety over $K$. Let $H$ be a
subgroup of $A[p]$ such that $Q=(A,H)\in\Yrig$. We say $H$ (or $Q$) is

\begin{itemize}
\item {\emph{canonical at $\gerp$}}, if $Q\in\calV_\gerp$; that
is, if for all $\beta\in\BB_\gerp$, we have
$\nu_\beta(Q)+p\nu_{\sigma^{-1} \circ\beta}(Q)<p$;

\item  {\emph{anti-canonical at $\gerp$}}, if $Q\in\calW_\gerp$;
that is, for all $\beta\in\BB_\gerp$, we have
$\nu_\beta(Q)+p\nu_{\sigma^{-1} \circ\beta}(Q)>p$;

\item {\emph{too singular at $\gerp$}}, if it is neither
canonical nor anti-canonical at $\gerp$;

\item {\emph{canonical}}, if it is canonical at all $\gerp$
dividing $p$. This is equivalent to $Q$ belonging to  $\calV$. Such
an $H$ is called {\emph{the canonical subgroup}} of $\uA$.

\item  {\emph{anti-canonical}}, if it is anti-canonical at all
$\gerp$ dividing $p$;

\item {\emph{too singular}}, if it is neither
canonical nor anti-canonical; equivalently, if it is too singular at some $\gerp$ dividing $p$.

\end{itemize}

\end{dfn}

We study the reduction properties of the canonical subgroup.

\begin{thm}\label{theorem: reduction of canonical subgroup}
Let $K\supseteq W(\kappa)$ be a completely valued field. Let
$\underline{A}/K$ be an abelian variety corresponding to a point of non-cuspidal reduction
$P\in \mathcal{U}$, and hence
\[
\nu_\beta(\uA)+p\nu_{\sigma^{-1}\circ\beta}(\uA)<p \qquad \forall
\beta\in\BB.
\]
Enlarging $K$, we can assume that there exists $r$ in $K$ such that
${\rm val}(r)=\max\{\nu_\beta(\underline{A}): \beta\in\mathbb{B}\}$,
and that the canonical subgroup of $A$ is defined over $K$.

\begin{enumerate}
\item  The canonical subgroup of $\underline{A}$ reduces to ${\rm Ker}({\rm Fr})$ modulo $p/r$.

\item Assume $K$ contains $r^{1/p}$, a $p$-th root of $r$. Let $C$ be an
anti-canonical subgroup of $\underline{A}$. Then, $C$ reduces to
${\rm Ker}({\rm Ver})$ modulo $p/r^{1/p}$.
\end{enumerate}
\end{thm}

\begin{proof} Let $\uA$ correspond to $P\in\calU$. Let $H$ be the canonical subgroup of $\uA$, and $Q=(\uA,H)$. Since $Q\in\calV$, we have $\nu_\beta(Q)\neq 1$ for all
$\beta\in\BB$. It follows from the definition of valuations that
$\eta(\Qbar)=I(\Qbar)$. Since $I(\Qbar)=\ell(\varphi(\Qbar))\cap
\eta(\Qbar)$, we find that $\eta(\Qbar)\subseteq
\ell(\varphi(\Qbar))$. The admissibility condition
$\ell(\varphi(\Qbar))^c\subseteq \eta(\Qbar)$ then implies that
$\varphi(\Qbar)=\BB$. This shows that $\Qbar$ belongs to
$Z_{\BB,\emptyset}=\Ybar_F=s(\Xbar)$. Since $s$ is a section to
$\pi\colon\Ybar\arr\Xbar$ and $\pi(\Qbar)=\Pbar$, it follows that
$s(\Pbar)=\Qbar$.

Using Lemma \ref{lemma: section and pi}, choose isomorphisms as in (\ref{equation: local deformation ring at Q bar -- arithmetic version}) and (\ref{equation: local def ring of X at P bar}) such that (thinking of variables
modulo $p$) we have $s^*(x_\beta)=t_\beta$ and $s^*(y_\beta)=0$ for
all $\beta\in I(\Qbar)$, and $s^*(z_\gamma)=t_\gamma$ for all $\gamma
\in I(\Qbar)^c$.

We have a commutative diagram, extending the relation $\pi(Q) = P$,
\[\xymatrix@C=50pt{
\Spf(\ok)\ar[r]^{\iota_Q}\ar[dr]_{\iota_P}& \Spf(\OhatQ \hat{\otimes}_{W(\kappa)} \ok)\ar[d]^{\pi}
\\
 &\Spf(\OhatP \hat{\otimes}_{W(\kappa)} \ok
).
}\]
Let $\germ$ be the ideal of $\ok$
generated by all $y_\beta(Q)$ for $\beta\in I(\Qbar)$. Then
$\germ=(p/r)$. We will denote reductions of $\iota_Q$ and $\iota_P$
modulo $\germ$, respectively, by $\overline{\iota}_Q$ and
$\overline{\iota}_P$. Let $\overline{\pi}$ and $\overline{s}$
denote, respectively, the base extension from $\kappa$ to
$\ok/\germ$ of  the natural projection
$\pi\colon\Spf(\OhatQbar)\arr\Spf(\OhatPbar)$, and the
Kernel-of-Frobenius section
$s\colon\Spf(\OhatPbar)\arr\Spf(\OhatQbar)$. We need to show
$\overline{s}\circ\overline{\iota}_P=\overline{\iota}_Q$.   For $f\in
\OhatQbar\otimes_\kappa \ok/\germ$ we have
\[
\overline{\iota}_P^*\circ
\overline{s}^*(f)-\overline{\iota}_Q^*(f)=\overline{\iota}_Q^*\circ\overline{\pi}^*\circ
\overline{s}^*(f)-\overline{\iota}_Q^*(f)=\overline{\iota}_Q^*(\overline{\pi}^*\circ
\overline{s}^*(f)-f),
\]
which we want to show is zero. Since $\overline{s}$ is a section to
$\overline{\pi}$, we see that $\overline{\pi}^*\circ
\overline{s}^*(f)-f$ is in the kernel of $\overline{s}^*$. By our
choice of variables, this kernel is the ideal generated by all the
$y_\beta$ for $\beta\in I(\Qbar)$. So it is enough to show that
$\overline{\iota}_Q^*(y_\beta)=0$ for all $\beta\in I(\Qbar)$. But
$\overline{\iota}_Q^*(y_\beta)$ is the reduction modulo $\germ$ of
$y_\beta(Q)$, which is zero, as $\germ$ contains $y_\beta(Q)$ by
definition.
\end{proof}

We now  prove a result which explains the geometry of the Hecke
correpondence $U_p$ on the  not-too-singular locus of $\Yrig$. In the Appendix, we will describe how to generalize this result so that it would explain the geometry of the
partial $U$-operators, $\{U_\gerp\}$, on the not-too-singular locus.

\begin{thm}\label{theorem: hecke}
Let $K\supseteq W(\kappa)$ be a completely valued field. Let
$\underline{A}/K$ be an abelian variety  corresponding to a point
on $\Xrig$.
 Let $H$ be a subgroup of $\underline{A}$ such that
$(\underline{A},H)\in\mathfrak{Y}_{\rm rig}$.  Assume that $H$ is
canonical at $\mathfrak{p}$.

\begin{enumerate}

\item If
$\nu_\beta(\underline{A})+p\nu_{\sigma^{-1}\circ\beta}(\underline{A})<1$
for all $\beta\in \mathbb{B}_\mathfrak{p}$, then
$\nu_\beta(\underline{A}/H)=p\nu_{\sigma^{-1}\circ\beta}(\underline{A})$
for all $\beta \in \mathbb{B}_\mathfrak{p}$ and $A[p]/H$ is
anti-canonical at $\mathfrak{p}$. In particular, if $H$ is canonical
and
$\nu_\beta(\underline{A})+p\nu_{\sigma^{-1}\circ\beta}(\underline{A})<1$
for all $\beta\in \mathbb{B}$, then
$\nu_\beta(\underline{A}/H)=p\nu_{\sigma^{-1}\circ\beta}(\underline{A})$
for all $\beta \in \mathbb{B}$ and $A[p]/H$ is anti-canonical.

\item If $1<\nu_\beta(\underline{A})+p\nu_{\sigma^{-1}\circ\beta}(\underline{A})<p$
for all $\beta\in \mathbb{B}_\mathfrak{p}$, then
$\nu_\beta(\underline{A}/H)=1-\nu_\beta(\underline{A})$ for all
$\beta \in \mathbb{B}_\mathfrak{p}$ and $A[p]/H$ is  canonical  at
$\mathfrak{p}$.  In particular, if $H$ is the  canonical  subgroup
and if
$1<\nu_\beta(\underline{A})+p\nu_{\sigma^{-1}\circ\beta}(\underline{A})<p$
for all $\beta\in \mathbb{B}$, then
$\nu_\beta(\underline{A}/H)=1-\nu_\beta(\underline{A})$ for all
$\beta \in \mathbb{B}$ and $A[p]/H$ is the canonical subgroup of
$\underline{A}/H$.

\item If there is a prime $\mathfrak{p}\vert p$ , and $\beta,\beta^\prime \in
\mathbb{B}_\mathfrak{p}$ such that we have
$\nu_\beta(\underline{A})+p\nu_{\sigma^{-1}\circ\beta}(\underline{A})\leq
1$ and
$\nu_{\beta^\prime}(\underline{A})+p\nu_{\sigma^{-1}\circ{\beta^\prime}}(\underline{A})\geq
1$, then $\underline{A}/H \not\in \mathcal{U}$.

\item Let $C$ be a subgroup of $\underline{A}$ which is anti-canonical at
$\mathfrak{p}$. Then,
$\nu_\beta(\underline{A}/C)=(1/p)\nu_{\sigma\circ\beta}(\underline{A})$,
for all $\beta \in \mathbb{B}_\mathfrak{p}$, and $ A[p]/C$ is
canonical at $\mathfrak{p}$. In particular, if $C$ is
anti-canonical, then
$\nu_\beta(\underline{A}/C)=(1/p)\nu_{\sigma\circ\beta}(\underline{A})$,
for all $\beta \in \mathbb{B}$ and $A[p]/C$ is the canonical
subgroup of $\uA/C$.

\end{enumerate}
\end{thm}

\begin{proof} We let $Q=(\uA,H)$ and $P=\pi(Q)=\uA$.

\

\noindent  (1): We have $Q\in \calV_\gerp$.  For
$\beta\in\BB_\gerp$, we write
\[
\nu_\beta(wQ)+p\nu_{\sigma^{-1}\circ\beta}(wQ)=1+p-(\nu_\beta(Q)+p\nu_{\sigma^{-1}\circ\beta}(Q))=1+p-(\nu_\beta(P)+p\nu_{\sigma^{-1}\circ\beta}(P))>p,
\]
using Proposition \ref{proposition: valuations under w} and Lemma
\ref{lemma: valuations under pi}. This shows that $wQ\in\calW_\gerp$
and proves that $A[p]/H$ is anti-canonical at $\gerp$. We can also
write
\[
\nu_{\beta}(\uA/H)=p(1-\nu_{\sigma^{-1}\circ\beta}(wQ))=p\nu_{\sigma^{-1}\circ\beta}(Q)=p\nu_{\sigma^{-1}\circ\beta}(P)=p\nu_{\sigma^{-1}\circ\beta}(\uA),
\]
using parts (2) and (1) of Lemma \ref{lemma:
valuations under pi} for the first and third equalities, respectively.

\

\noindent  (2): As above, we find
$\nu_\beta(wQ)+p\nu_{\sigma^{-1}\circ\beta}(wQ)=1+p-(\nu_\beta(P)+p\nu_{\sigma^{-1}\circ\beta}(P))<p$
for all $\beta \in \BB_\gerp$. Hence $wQ\in\calV_\gerp$, and
$A[p]/H$ is canonical at $\gerp$. For $\beta\in\BB_\gerp$, we have
\[
\nu_{\beta}(\uA/H)=
\nu_\beta(wQ)=1-\nu_\beta(Q)=1-\nu_\beta(P)=1-\nu_\beta(\uA),
\]
using part (1) of Lemma  \ref{lemma: valuations under pi}, and the
fact that both $Q$ and $wQ$ belong to $\calV_\gerp$.

\

\noindent  (3): As above, we find that there is $\beta\in\BB_\gerp$
such that $\nu_\beta(wQ)+p\nu_{\sigma^{-1}\circ\beta}(wQ)\geq p$ and
there is $\beta^\prime\in\BB_\gerp$ such that
$\nu_{\beta^\prime}(wQ)+p\nu_{\sigma^{-1}\circ{\beta^\prime}}(wQ)\leq
p$. This implies that $wQ\not\in \calV_\gerp \cup \calW_\gerp$, and
equivalently, $wQ\not\in \calV \cup \calW$. By Corollary
\ref{corollary: pi inverse of U}, we find that
$\uA/H=\pi(wQ)\not\in\calU$.

\

\noindent  (4):  Let $Q^\prime=(\uA,C)$, and $P=\uA=\pi(Q^\prime)$. By
assumption, $Q^\prime\in\calW_\gerp$, and we can write
\[
\nu_\beta(wQ^\prime)+p\nu_{\sigma^{-1}\circ\beta}(wQ^\prime)=1+p-(\nu_\beta(Q^\prime)+p\nu_{\sigma^{-1}\circ\beta}(Q^\prime))<1<p,
\]
for all $\beta\in\BB_\gerp$. Hence $wQ^\prime\in\calV_\gerp$ and
$A[p]/C$ is canonical at $\gerp$. For $\beta\in\BB$ we write
\[
\nu_{\beta}(\uA/C)=\nu_\beta(wQ^\prime)=1-\nu_\beta(Q^\prime)=(1/p)\nu_{\sigma\circ\beta}(P)=(1/p)\nu_{\sigma\circ\beta}(\uA),
\]
using parts (1) and (2) of Lemma \ref{lemma:
valuations under pi} for the first and third equalities, respectively.

\end{proof}

Employing an iterative construction, we can use the above theorem to prove the existence of higher-order canonical subgroups.

\begin{prop}\label{proposition: iterations} Let $\uA$ defined over $K$ correspond to a point $P$ on $\Xrig$. Let $n$ be a non-negative integer. Assume that
\[
\nu_\beta(\underline{A})+p\nu_{\sigma^{-1}\circ\beta}(\underline{A})<p^{1-n}, \qquad \forall\beta\in \mathbb{B}.
\]
 Then, for $1\leq i \leq n+1$, there are isotropic  finite flat subgroup schemes $H_i$ of $A$ of order $p^{ig}$, $\ol$-invariant, and killed by $p^i$, forming an increasing sequence
\[
H_1\subset H_2 \subset \dots \subset H_{n+1},
\]
where $H_1$ is the canonical subgroup of $\uA$, for any $1\leq i\leq n$ we have $p^i H_{n+1}=H_{n+1-i}$, and each $H_i$ is a cyclic $\ol$-module. Furthermore, if $P$ has non-cuspidal reduction, and if $r$ in $K$ is such that
${\rm val}(r)=\max\{\nu_\beta(\underline{A}): \beta\in\mathbb{B}\}$, then $H_i$ reduces modulo $p/r^{p^{i-1}}$ to $\Ker(\Fr^i)$.
\end{prop}

\begin{proof} The case $n=0$ is a consequence of the above results, and thus we can assume that $n\geq 1$. We construct  this flag of subgroups recursively. For simplicity, we denote $\nu_\beta+p\nu_{\sigma^{-1}\circ\beta}$ by $\lambda_\beta$. Let $H_0=\{0\}$ and $H_1$ be the canonical subgroup of $\uA$, which exists by Theorem \ref{theorem: canonical subgroup}. Fix $m<n+1$, and assume that for $1\leq i \leq m$ there is a increasing sequence of subgroups
\[
H_1 \subset \dots \subset H_i \subset \dots \subset H_m,
\]  where for each $1\leq i \leq m$, the subgroup $H_i$ has the stated properties, and, additionally:
\begin{enumerate}
\item $\lambda_\beta(\underline{A}/H_i)<p^{1-n+i}$ for all $\beta\in\BB$;
\item $H_i/H_{i-1}$ is the canonical subgroup of $\uA/H_{i-1}$;
\item $p^jH_i=H_{i-j}$ for all $0\leq j \leq i$.
\end{enumerate}
This holds for $m=1$ by part (1) of Theorem \ref{theorem: hecke}, since $\lambda_\beta(\uA)<p^{1-n}\leq 1$, for all $\beta\in\BB$.

Given the above data, we construct $H_{m+1}$. Since $\lambda_\beta(\underline{A}/H_m)<p^{1-n+m}\leq p$, $\uA/H_m$ has a canonical subgroup  of the form $H_{m+1}/H_m$, where $H_{m+1}$ is an isotropic finite flat subgroup scheme of $A$ of order $p^{(m+1)g}$ containing $H_m$ and  killed by $p^{m+1}$.  We show that $pH_{m+1}=H_m$. Since, by construction, $pH_{m+1} \subseteq H_m$, it  is enough to show that $H_{m+1}\cap A[p]=H_1$. But  since $\lambda_\beta(\underline{A}/H_{m-1})<p^{-n+m}\leq 1$, we can apply part (1) of Theorem \ref{theorem: hecke} to $A/H_{m-1} $ and its canonical subgroup $H_m/H_{m-1}$  to deduce that  $(A[p]+H_m)/H_m$ and $H_{m+1}/H_m$ have trivial intersection, which proves the claim. To carry the induction forward, we only need to prove that when $m+1<n+1$, we have $\lambda_\beta(\uA/H_{m+1})<p^{1-n+m+1}$ for all $\beta\in \BB$. That follows, since when $m+1<n+1$, we have $\lambda_\beta(\underline{A}/H_m)<p^{1-n+m}\leq 1$, and hence, by part (1) of Theorem \ref{theorem: hecke}, we have $\nu_\beta(\underline{A}/H_{m+1})=p\nu_{\sigma^{-1}\circ\beta}(\underline{A}/H_m)$  for all $\beta\in \BB$.

The final statement follows from the iterative construction and part (1) of Theorem \ref{theorem: reduction of canonical subgroup}.
\end{proof}

\

\

\section{Functoriality}\label{section: functoriality}
There are two kinds of functoriality associated with the moduli
spaces $X, Y$, and the many maps we have defined in that context. One
kind of functoriality is coming from the moduli problem itself and
is based on the construction $A \mapsto A\otimes_{\ol} \calO_M$
associating to an abelian variety $A$ with RM by $\ol$ another
abelian variety with RM by $\calO_M$ for an extension of totally
real fields $L \subseteq M$. We show that the canonical subgroup
behaves naturally relative to this construction and deduce a
certain optimality result for canonical subgroups
(Corollary~\ref{Cor: optimality}). The second kind of functoriality
is relative to Galois automorhisms and comes from the fact that the
moduli spaces $X, Y$, are in fact defined over $\ZZ_p$. There is thus
a natural Galois action on our constructions that were made over
$W(\kappa)$, which induces a descent data. This would allow us to
show that the construction of section $s^\dagger$ descends to
$\QQ_p$.

\

\subsection{Changing the field} Let $L \subseteq M$ be totally real fields in which $p$ is
unramified. Let $\BB^L=\Emb(L,\overline{\QQ}_p)$; similarly define
$\BB^M$. Given a subset $S \subseteq \BB^L$, let
$S^M=\{\beta\in\BB^M: \beta_{|_{L}}\in S\}$. In this section, we decorate the notation we have used so far with an $M$
or $L$ as a superscript. For example, we use $X^L$, $X^M$,
$\pi^L$, etc. Given $Q=(\uA,H)\in Y^L$, $P=\pi^L(Q)\in X^L$, we get
points $\epsilon_{L,M}(Q)=(A\otimes_\ol \om,H\otimes_\ol \om)\in
Y^M$, $\epsilon_{L,M}(P)=A\otimes_\ol \om \in X^M$. This induces
morphisms
\[
\epsilon=\epsilon_{L,M}\colon X^L \arr X^M, \qquad
\epsilon=\epsilon_{L,M}\colon Y^L \arr Y^M,\] which extend to morphisms
\[
\epsilon=\epsilon_{L,M}\colon \XX^L \arr \XX^M, \qquad
\epsilon=\epsilon_{L,M}\colon \YY^L \arr \YY^M.\]
 These morphisms fit
into commutative diagrams: \begin{equation} \xymatrix{\YY^L
\ar[r]^\epsilon\ar[d]_{\pi_L} & \YY^M\ar[d]^{\pi_M} \\ \XX^L
\ar[r]^\epsilon & \XX^M,}\qquad \xymatrix{\YY^L
\ar[r]^\epsilon\ar[d]_{w^L} & \YY^M\ar[d]^{w^M} \\ \YY^L \ar[r]^\epsilon
& \YY^M.}
\end{equation}
Let $k\supseteq \kappa$ be a perfect field. Let $\Pbar, \Qbar$ be  $k$-rational points of $\cXbar,\cYbar$, respectively. Using that $\DD(\epsilon(\uA)[p]) = \DD(\uA[p]) \otimes_{\ol} \calO_M$, and so
$\Lie(\epsilon(\uA)) = \Lie(\uA)\otimes_{\ol} \calO_M$ etc., we find
that
\[\tau(\epsilon(\Pbar))=\tau(\Pbar)^M,\] and \[\varphi(\epsilon(\Qbar))=\varphi(\Qbar)^M,
\quad \eta(\epsilon(\Qbar))=\eta(\Qbar)^M, \quad I(\epsilon(\Qbar))=I(\Qbar)^M.
\]
Suppose that $\Qbar\in\Ybar(k)$, and $\Pbar\in\Xbar(k)$. It is clear from the discussion in \S \ref{subsection: infinitesimal
nature of Ybar} that one can choose parameters for
$\widehat{\calO}_{X^L, \Pbar}$, $\widehat{\calO}_{Y^L, \Qbar}$,
$\widehat{\calO}_{X^M, \epsilon(\Pbar)}$, $\widehat{\calO}_{Y^M,
\epsilon(\Qbar)}$ as in loc. cit. so that, in addition, the map
$\epsilon^*\colon\widehat{\calO}_{X^M, \epsilon(\Pbar)} \arr
\widehat{\calO}_{X^L, \Pbar}$ satisfies
\[\epsilon^*(t_\beta)=t_{(\beta\vert_L)},\] and that
$\epsilon^*\colon\widehat{\calO}_{Y^M, \epsilon(\Qbar)} \arr
\widehat{\calO}_{Y^L, \Qbar}$ satisfies
\[\epsilon^*(x_\beta)=x_{(\beta\vert_L)},\quad
\epsilon^*(y_\beta)=y_{(\beta\vert_L)},\quad
\epsilon^*(z_\beta)=z_{(\beta\vert_L)}.\] It is then clear that the
function $\Delta_{L,M}\colon\Theta^L \arr \Theta^M$ given by
$(a_\beta)_{\beta\in\BB^L} \mapsto (b_\beta)_{\beta\in\BB^M}$, where
$b_\beta=a_{(\beta\vert_L)}$, fits into commutative diagrams:
\begin{equation}\xymatrix{\gerX_{\rm rig}^L \ar[d]_{\nu}\ar[r]^{\epsilon_{L,M}} & \gerX_{\rm rig}^M\ar[d]^{\nu} \\
 \Theta^L \ar[r]^{\Delta_{L,M}} & \Theta^M,}
 \qquad \xymatrix{\gerY_{\rm rig}^L \ar[d]_{\nu}\ar[r]^{\epsilon_{L,M}} & \gerY_{\rm rig}^M\ar[d]^{\nu} \\
 \Theta^L \ar[r]^{\Delta_{L,M}} & \Theta^M. }\end{equation}
It follows immediately from the definitions and the discussion above
that
\begin{enumerate}
\item  $\epsilon_{L,M}^{-1}(\calU^M)=\calU^L$, and similarly for $\calV^L$,$\calV^M$, and $\calW^L$,$\calW^M$.
\item   $\epsilon_{L,M}^{-1}(\calV_\gerp^M)=\calV_{\gerp\cap L}^L$.
\end{enumerate}
\begin{cor}\label{Cor: optimality}
Let $\calU^+ \supseteq \calU^L$ be an admissible open of $\gerX_{\rm
rig}^L$ containing
$\nu_{\gerY^L}^{-1}(\frac{p}{p+1},\frac{p}{p+1},\dots,\frac{p}{p+1})$.
The section $s^{\dagger,L}\colon \calU^L \arr \gerY_{\rm rig}^L$
cannot be extended to $\calU^+$.
\end{cor}

\begin{proof}
If $s^{\dagger,L}$ can be extended to $\calU^+$, then the above
functoriality results would imply that
$s^{\dagger,\QQ}\colon\calU^\QQ \arr \gerY_{\rm rig}^\QQ$ can be
extended to $\epsilon_{\QQ,L}^{-1}(\calU^+)\supseteq \calU^\QQ$,
which contains $\nu_{\gerY^\QQ}^{-1}(\frac{p}{p+1})$. This is
impossible by \cite[Theorem 3.9]{GK}.
\end{proof}

\

\subsection{Galois automorphisms} We next discuss the action of ${\rm
Gal}(\overline{\QQ}_p/\QQ_p)$ on $\XX, \YY$, and all the derived maps.
The action is a result of the identifications $\XX=\XX_{\ZZ_p} \otimes_{\ZZ_p} W(\kappa)$, $\YY=\YY_{\ZZ_p} \otimes_{\ZZ_p} W(\kappa)$,
while our constructions used the $W(\kappa)$-structure of $\XX, \YY$. The following facts are easy to verify.
\begin{enumerate}
\item ${\rm Gal}(\overline{\QQ}_p/\QQ_p)$ acts on $\BB$ by composition and acts transitively on each $\BB_\gerp$.
Let $\gamma\in{\rm Gal}(\overline{\QQ}_p/\QQ_p)$; it induces the following maps:
\begin{itemize}
\item $\gamma\colon W(\kappa)\arr W(\kappa)$,
\item $\gamma^*\colon \Spec(W(\kappa)) \arr  \Spec(W(\kappa)) $,
\item $1\times \gamma^*: \YY=\YY_{\ZZ_p} \times_{\Spec(\ZZ_p)} \Spec(W(\kappa)) \arr \YY=\YY_{\ZZ_p} \times_{\Spec(\ZZ_p)} \Spec(W(\kappa))$.
\item $1\times \gamma^*: \XX=\XX_{\ZZ_p} \times_{\Spec(\ZZ_p)} \Spec(W(\kappa)) \arr \XX=\XX_{\ZZ_p} \times_{\Spec(\ZZ_p)} \Spec(W(\kappa))$.
\end{itemize}
Moreover, the following diagram is commutative:
\begin{equation}\label{Diagram: galois}\xymatrix{\YY\ar[d]_\pi \ar[r]^{1\times \gamma^*} &
\YY\ar[d]^\pi \\ \XX \ar[r]^{1\times \gamma^*} & \XX}\end{equation} (and
similarly for $\Xrig, \Yrig$).
\item For $\Qbar$ a closed point of $\cYbar$, and  $\Pbar$ a closed point of $\cXbar$:
\begin{itemize}

\item $\varphi(1\times \gamma^*(\Qbar))=\gamma(\varphi(\Qbar))$,
\item $\eta(1\times \gamma^*(\Qbar))=\gamma(\eta(\Qbar))$,
\item $I(1\times \gamma^*(\Qbar))=\gamma(I(\Qbar))$,
\item $\tau(1\times \gamma^*(\Pbar))=\gamma(\tau(\Pbar))$,
\end{itemize}
where for $S\subseteq \BB$, $\gamma(S)=\{\gamma\circ\beta:\beta\in
S\}$.

 \item Let $\gamma\colon\Theta \arr \Theta$ be given by $(a_\beta)_{\beta\in\BB} \mapsto (b_\beta)_{\beta\in\BB}$, where $b_\beta=a_{\gamma^{-1}\beta}$. Then, the following diagram is commutative:
\begin{equation}\xymatrix{\Yrig \ar[r]^{1\times
\gamma^\ast}\ar[d]^\nu& \Yrig\ar[d]^\nu \\\Theta \ar[r]^\gamma &
\Theta.}\end{equation}

The same statement holds for $\Xrig$.

\item It follows immediately from the definitions and the statements above that
\begin{itemize}
\item  $1\times \gamma^*(\calU)=\calU$,
\item  $1\times \gamma^*(\calV)=\calV$,
\item $1\times \gamma^*(\calW)=\calW$.
\end{itemize}

\end{enumerate}

\

To be able to use our results on the geometry of $\cYbar$, $\cXbar$,
and $\pi\colon\cYbar\arr\cXbar$, we found it convenient to work  over
$W(\kappa)$ throughout the paper. The canonical section $s^\dagger$, however, can be
shown to exist over $\QQ_p$.

\begin{thm}\label{theorem: descent} There are admissible opens
$\;\calU_{\QQ_p} \subset {\gerX_{{\rm rig},\QQ_p}}\!$ and
$\calV_{\QQ_p} \subset {\gerY_{{\rm rig},\QQ_p}}\!$ whose base
changes to $\QQ_\kappa$ are $\calU$, $\calV$, respectively, and a
section
\[
s^\dagger_{\QQ_p}\colon\calU_{\QQ_p} \arr {\gerY_{{\rm rig},\QQ_p}}
\]
to $\pi_{\QQ_p}\colon\gerY_{{\rm rig},\QQ_P}\arr\gerX_{{\rm
rig},\QQ_p}$ whose image is $\calV_{\QQ_p}$, and whose base change
to $\QQ_\kappa$ is $s^\dagger$.
\end{thm}

\begin{proof}
We may think about the points of $\Yrig$, or $\Xrig$, as Galois
orbits of the points of $\Yrig$, or $\Xrig$, over $\Qpbar$, relative
to the action of $\Gal(\Qpbar/\QQ_p)$, and so the points of
$\gerY_{\rig,{\QQ_p}}$ are the Galois orbits of the points of $\Yrig$
relative to $\Gal(\QQ_\kappa/\QQ_p)$. It is then clear that the
section $s^\dagger$ is well-defined on such orbits (cf.
Diagram~\ref{Diagram: galois}). As a set, $\calU_{\QQ_p}$ comprises the $\Gal(\QQ_\kappa/\QQ_p)$-orbits of the
points of $\calU$; similarly, for $\calV_{\QQ_p}$.

Two points may be worth mentioning. The admissibility of
$\calU_{\QQ_p}$ boils down to the fact $\calU$ is a union of affinoids invariant under the Galois action (since they are defined by valuation conditions), and that an ideal of a
$\QQ_\kappa$-Tate algebra which is invariant under
$\Gal(\QQ_\kappa/\QQ_p)$ is generated by power series with
coefficients in $\QQ_p$. The fact that $s^\dagger_{\QQ_p}$ is a
morphism boils down to the fact that for two $\QQ_p$-affinoids
$\calA, \calB$ we have $\Hom(\calA, \calB) = \Hom_{K}(\calA \otimes
K, \calB \otimes K)^{\Gal(K/\QQ_p)}$ for any finite Galois extension
$K/\QQ_p$. Both these statements are easy to check.
\end{proof}

\

\

\section{Appendix} In this section, we construct variants of the
moduli space $Y$ considered in the paper so far. These moduli spaces
are natural and useful to the construction of partial $U$ operators $\{U_\gerp\}$. Since the results are very similar to the results
obtained henceforth, we shall be very brief, except for very
specific results not mentioned previously.

\

\subsection{A variant on the moduli problem} Let $\gert$ be an ideal of $\ol$ dividing $p$,  ${\gert^\ast} = p/\gert$ (so $\gert {\gert^\ast} =
p\ol$), $\BB_\gert = \cup_{\gerp \vert \gert} \BB_\gerp$, and
 $f(\gert) = \sum_{\gerp\vert \gert} f(\gerp)$  the sum of
the residue degrees. Consider the functor associating to a
$W(\kappa)$-scheme, the isomorphism classes of $(\uA, H)$, where $\uA$ is
an object parameterized by $X$, and where $H\subseteq A[p]$ is an
isotropic $\ol$-invariant finite-flat subgroup scheme of rank $p^{f(\gert)}$, on
which $\gert$ acts as zero; $H$ has a well-defined action of
$\ol/\gert \cong \oplus_{\gerp\vert \gert} \ol/\gerp$. There is a
fine moduli scheme $Y(\gert)$ representing this functor, which
actually has a model over $\ZZ_\gerp$; it affords a minimal compactification denoted by $\YY(\gert)$.
In what follows, $\Ybar(\gert),\cYbar(\gert), \Yrig(\gert)$ are defined similar to the case $\gert=p\ol$ studied throughout the paper.

The natural morphism
\[ \pi(\gert)\colon\YY(\gert) \arr \XX\]
is proper, and  finite-flat over $\QQ_\kappa$ of degree
$\prod_{\gerp\vert \gert} (p^{f(\gerp)} + 1)$. Given $\uA$, there
is a canonical decomposition $A[p]  = \prod_{\gerp\vert p}A[\gerp]$,
where $A[\gerp] = \cap_{a\in \gerp} \Ker(\iota(a))$ is a finite flat
group scheme of rank $p^{f(\gerp)}$. The subgroup $H$ appearing
above is a subgroup of $A[\gert] := \prod_{\gerp\vert \gert}
A[\gerp]$. Suppose $A$ is over a characteristic $p$ base; let
$\Ker(\Fr)[\gert] = \Ker(\Fr) \cap A[\gert]$. It is an example of a
subgroup scheme of the kind parameterized by $Y(\gert)$ (cf. Lemma \ref{lemma: cyclic is isotropic}). There is
a natural section
\[ s(\gert)\colon \cXbar \arr \cYbar(\gert),\]
given on non-cuspidal points by
\[ \uA \mapsto (\uA, \Ker(\Fr)[\gert]).\]
We can also present the moduli problem in a ``balanced" way. It
requires making auxiliary choices. The ideal $\gert$ has a natural
notion of positivity, and so it acts on the representatives
$[\Cl^+(L)]$ for the strict class group of $L$. Choose for every
such representative $(\gera, \gera^+)$ an isomorphism $\gamma_{\gert,(\gera,\gera^+)}$ of
$\gert\cdot(\gera, \gera^+)$ with another (uniquely determined)
representative in $[\Cl^+(L)]$. (In the case considered in the body
of the paper $\gert = (p)$, and there is a canonical isomorphism
$p\cdot (\gera, \gera^+) \cong (\gera, \gera^+)$, which is dividing
by $p$.)

Consider the moduli of $(f\colon \uA \arr \uB)$ such that $H =
\Ker(f)$ is an $\ol$-invariant  isotropic finite-flat subgroup
scheme of $A[\gert]$  of order $p^{f(\gert)}$, such that $f^\ast
\calP_{\uB} = \gert \calP_{\uA}$, and
$\lambda_B=\gamma_{\gert,(\gera,\gera^+)}\circ\lambda_A$. There is
an automorphism
\[ w: \YY(\gert) \arr \YY(\gert), \]
constructed as follows. Given $(\uA, H)$ consider $\uA/H$ and its
subgroup scheme $A[p]/H$; take its $\gert$-primary part
$(A[p]/H)[\gert]$. We let \[w(\uA, H) = (\uA/H, (A[p]/H)[\gert]).\]
In terms of the presentation $(f\colon \uA \arr \uB)$, we first find
a unique isogeny $f^t\colon \uB \arr \uA$, such that \[f^t\circ f =
[p].\] We then replace $\Ker(f^t)$ by its $\gert$-primary part to
obtain the subgroup scheme $(A[p]/H)[\gert]$. We note that $w$
permutes the connected components of $\YY(\gert)$; it changes the polarization module.

Given $(\uA,H)$, we let $B = \uA/H$, $f: \uA \arr \uB$ the canonical
map, and consider the diagrams:

\begin{align}
\bigoplus_{\beta\in \BB}\Lie(f)_\beta & \colon \bigoplus_{\beta\in
\BB}\Lie(\uA)_\beta \Arr \bigoplus_{\beta\in \BB}\Lie(\uB)_\beta,
\\\nonumber
 \bigoplus_{\beta\in \BB}\Lie(f^t)_\beta & \colon \bigoplus_{\beta\in
\BB}\Lie(\uB)_\beta \Arr \bigoplus_{\beta\in \BB}\Lie(\uA)_\beta.
\end{align}
We let, as before,
\begin{align}
\varphi(\uA, H) & = \{ \beta\in \BB_\gert\colon
\Lie(f)_{\sigma^{-1} \circ \beta} = 0\}, \\
\eta(\uA, H)  & = \{ \beta\in \BB_\gert\colon \Lie(f^t)_\beta =
0\}, \\
I(\uA, H) & = \ell(\varphi(f)) \cap \eta(f) = \{ \beta\in \BB\colon
\Lie(f)_\beta = \Lie(f^t)_\beta= 0\}.
\end{align}
Note that  $\varphi(\uA, H)$ in fact equals $\{ \beta\in \BB\colon
\Lie(f)_{\sigma^{-1} \circ \beta} = 0\}$, while $\{ \beta\in
\BB\colon \Lie(f^t)_\beta = 0\}$ always contains $\BB_{{\gert^\ast}}$,
which is ``superfluous information". For that reason, and for having
nicer formulas in what follows, we have defined $\eta(\uA, H)$ using
only $\BB_\gert$.

 One checks that these
invariants do not depend on the choice of $\gamma_{\gert,(\gera,\gera^+)}$. A pair
$(\varphi, \eta)$ of subsets of $\BB_\gert$ is called \emph{$\gert$-admissible} if
 $\BB_\gert-\ell(\varphi)\subseteq \eta$. There are $3^{f(\gert)}$
such pairs.

We define, as before,  subsets $\Wpe(\gert), \Zpe(\gert)$ of $\Ybar(\gert)$ for
$\gert$-admissible pairs $\pe$, and their natural extensions $\WW_{\varphi,\eta}(\gert), \ZZ_{\varphi,\eta}(\gert)$ to $\cYbar(\gert)$. This gives us, as before, a stratification of
$\cYbar(\gert)$, with very similar properties to the stratification
studied in the body of the paper; the reader should have no
difficulty listing its properties.\footnote{It should be remarked that since
the strata can be defined by vanishing of sections of line bundles,
they have a natural scheme structure and are, in fact,
reduced. Interestingly, that way one can see that each strata of
$\cYbar(\gert)$ (and also of $\cXbar$) represents a moduli problem,
where the additional data is precisely the vanishing of the sections
referred to above.}

The various moduli spaces $\YY(\gert)$ satisfy the following
compatibility relation: for $\gert, \gerz$ relatively prime ideals
of $\ol$ dividing $p$ we have \begin{equation}
 \YY(\gert\gerz) =
\YY(\gert) \times_\XX \YY(\gerz). \end{equation} Let $(\varphi, \eta)$ be
a $\gert$-admissible pair and $(\varphi', \eta')$ a
$\gerz$-admissible pair. Then \begin{equation}\label{equation: fibre product of strata} \WW_{\varphi,\eta}(\gert) \times_\XX \WW_{\varphi', \eta'}(\gerz) =
\WW_{\varphi\cup\varphi', \eta \cup \eta'}(\gert\gerz), \qquad  \ZZ_{\varphi,\eta}(\gert) \times_\XX
\ZZ_{\varphi', \eta'}(\gerz) = \ZZ_{\varphi\cup\varphi', \eta \cup \eta'}(\gert\gerz).
\end{equation}
Given two points $(\uA, H), (\uA, H')$ on $\Ybar(\gert)$ and $\Ybar(\gerz)$,
respectively, our invariants involve two different quotients of
$\uA$. One verifies (\ref{equation: fibre product of strata}) by
considering the following cartesian diagram:
\[ \xymatrix@=1pc{& \uA\ar[dr]\ar[dl] & \\ \uA/H\ar[dr] &&\uA/H'\ar[dl] \\ &\uA/(H\times H'). & }\]
Alternately, one can prove a lemma similar to Lemma~\ref{lemma: type
in terms of Dieudonne modules}, which describes the invariants in
terms of the relative position of $\Fr(\DD(A[p])), \Ver(\DD(A[p]))$,
and the submodule $M$ of $\DD(A[p])$ such that $\DD(A[p])/M =
\DD(H)$.

Define:
\[ \cYbar(\gerp)_F := \ZZ_{\BB_\gerp, \emptyset}(\gerp), \qquad  \cYbar(\gerp)_V :=\ZZ_{\emptyset, \BB_\gerp}(\gerp).\]
In fact, we have $\cYbar(\gerp)_F=s(\gerp)(\cXbar)$, and $ \cYbar(\gerp)_V=w(\cYbar(\gerp)_F)$.
For every ideal $\gert \vert p$, we have
\[\ZZ_{\BB_\gert, \BB_{{\gert^\ast}}}= \underset{\gerp \vert \gert}{\times_{\cXbar}} \cYbar(\gerp)_F \times
\underset{\gerp \vert {\gert^\ast}}{\times_{\cXbar}} \cYbar(\gerp)_V,\]by Equation  (\ref{equation: fibre product of strata}). In
particular, the choice of $\gert = p$ gives $\cYbar_F$, and the choice
of $\gert = \ol$ gives $\cYbar_V$.

\begin{lem}\label{lemma: pi is flat on compactified horizontal strata}
The morphism
\[ \pi\colon \ZZ_{\BB_\gert, \BB_{{\gert^\ast}}} \Arr \cXbar\]
is finite, flat, and purely inseparable of degree $p^{f({\gert^\ast})}$.
\end{lem}
\begin{proof}
To prove the lemma, it is enough to prove that the morphisms
\[\pi(\gerp)\colon\cYbar(\gerp)_F \Arr \cXbar, \qquad  \pi(\gerp)\colon\cYbar(\gerp)_V \Arr \cXbar,\]
obtained by restriction from $\cYbar(\gerp)$, are finite-flat of degree $1$ and $p^{f(\gerp)}$, respectively. We first consider the situation without compactifications; we have
the morphisms
\[\pi(\gerp)\colon\Ybar(\gerp)_F \Arr \Xbar, \qquad  \pi(\gerp)\colon\Ybar(\gerp)_V \Arr \Xbar.\]
Since the morphism
$\pi(\gerp):\Ybar(\gerp) \arr \Xbar$ is proper and quasi-finite on each of
$\Ybar(\gerp)_F$, $\Ybar(\gerp)_V$, it is in fact finite
\cite[Chapter 3, \S4, Proposition 4.4.2]{EGA}. Since a
finite surjective morphism of non-singular varieties over an
algebraically closed field is flat \cite[Chapter III, Exercise
9.3(a)]{Hartshorne}. It follows that indeed
we have two finite flat morphisms. The existence of a section $\Xbar
\arr \Ybar(\gerp)_F$ shows the morphism $\Ybar(\gerp)_F \arr \Xbar$
is an isomorphism. The generic fibre of $Y(\gerp) \arr X$ is finite
flat of degree $p^{f(\gerp)}+1$, and so it follows that the morphism
$\Ybar(\gerp)_V \arr \Xbar$ is finite flat of degree $p^{f(\gerp)}$.
Since the reduced fibers of this morphism are just singletons, we
conclude that $\Ybar(\gerp)_V \arr \Xbar$ is purely inseparable of
degree $p^{f(\gerp)}$.

We now indicate how to extend the argument to compactifications.  One can use the description of the completed
local ring of a cusp via $q$-expansions. Let $c'$ be a cusp of $\cYbar$, $c=\pi(c')$, viewed over $\overline{\FF}_p$, lying on the connected component of $\Ybar$, respectively, $\cXbar$, corresponding to a polarization module $(\gera, \gera^+)$. To $c$ one associates  a pair $(T, \Lambda)$, of translation module and multiplier group, determined by the group $\Gamma \subset \SL_2(L)$ corresponding to the given polarization datum and level $\Gamma_{00}(N)$ (see \cite[I.\S2-3]{Freitag}). Taking instead the subgroup $\Gamma'$ of $\Gamma$, corresponding to adding the level structure at $p$,  one obtains another pair $(T', \Lambda')$, where $T', \Lambda'$ are finite index subgroups of $T, \Lambda$, respectively. The completed local ring of $c$ on $\cXbar$ is $\overline{\FF}_p[\![ q^\nu]\!]_{\nu\in T^\vee}^\Lambda$ and on $\cYbar$ is $\overline{\FF}_p[\![ q^\nu]\!]_{\nu\in T'^\vee}^{\Lambda'}$ (see loc. cit. and \cite{Chai}). Checking flatness at the cusp $c$ amounts to checking that  $\overline{\FF}_p[\![ q^\nu]\!]_{\nu\in T'^\vee}^{\Lambda'}$ is a free $\overline{\FF}_p[\![ q^\nu]\!]_{\nu\in T^\vee}^\Lambda$-module, which can be verified by a straightforward calculation.
\end{proof}

\begin{lem}\label{lemma: pi, infinitesimally, on horizontal strata (Appendix)} Let $\Qbar \in W_{\BB_\gert,\BB_{\gert^\ast}}$, and $\Pbar=\pi(\Qbar)$.  We can choose isomorphisms as in (\ref{equation: GO isomorphism}) at $\Pbar$, and (\ref{equation: stamm isomorphism at horizontal strata}) at $\Qbar$, such that:
\[
\pi^{*}(t_\beta)=\begin{cases} z_\beta\qquad \beta  \in \BB_\gert \\    z_\beta^p \qquad \beta \in \BB_{\gert^\ast}.\end{cases}
\]

\end{lem}

\begin{proof}  First, we prove the lemma for the cases $\gert=p\ol$ (i.e., $\Qbar\in\Ybar_F^{\rm ord}$), and $\gert=\ol$ (i.e., $\Qbar\in\Ybar_V^{\rm ord}$).
In the first case, the result follows from Lemma \ref{lemma: section and pi} and the fact that $\pi^*$ is an inverse to $s^*$.

Assume now that $\Qbar\in\Ybar_V^{\rm ord}$. Consider the morphism $ \pi \circ w \circ s\colon \Xbar_{\FF_p} \rightarrow \Xbar_{\FF_p} $. Calculating the effect of this map on points, we find that it is equal to the Frobenius morphism $\Fr\colon  \Xbar_{\FF_p}  \arr  \Xbar_{\FF_p}$.   Let $\Qbar_1$ be a point in $\Ybar_F^{\rm ord}$ such that $w(\Qbar_1)=\Qbar$ and let $\Pbar_1=\pi(\Qbar_1)$; then $I(\Qbar_1)=\emptyset$, and $\Pbar=\Fr(\Pbar_1)$.  By Lemma \ref{lemma: section and pi}, there are choices of isomorphisms as in   (\ref{equation: GO isomorphism}) at $\Pbar_1$, and (\ref{equation: stamm isomorphism at horizontal strata}) at $\Qbar_1$, such that $s^*(z_{\beta,\Qbar_1})=t_{\beta,\Pbar_1}$. We may choose parameters $\{t_{\beta,\Pbar}\}_\beta$ at $\Pbar$, such that $\Fr^*(t_{\beta,\Pbar})=t_{\beta,\Pbar_1}^p$, and by Lemma \ref{lemma: w and local parameters}, we can find parameters as in (\ref{equation: stamm isomorphism at horizontal strata}) at $\Qbar$, such that $w^*(z_{\beta,\Qbar})=z_{\beta,\Qbar_1}$. Therefore,
\[
s^* \circ w^*(z_{\beta,\Qbar}^p)=t_{\beta,\Pbar_1}^p.
\]
On the other hand, from the above discussion, we have
\[
s^* \circ w^* \circ \pi^* (t_{\beta,\Pbar})=\Fr^*(t_{\beta,\Pbar})=t_{\beta,\Pbar_1}^p.
\]
Since $s^* \circ w^*$ is an isomorphism, it follows that $\pi^* (t_{\beta,\Pbar})=z_{\beta,\Qbar}^p$.

For a general horizontal stratum, we argue as follows.  The arguments in \S\ref{subsection: infinitesimal nature of Ybar} can be repeated ad verbatim for $\Ybar(\gerp)$ to produce parameters $\{x_{\beta,\gerp}: \beta\in \BB_\gerp; y_{\beta,\gerp}: \beta \in \BB_{\gerp^\ast}\}$   at a point in $W_{\BB_\gerp,\BB_\gerp^\ast}(\gerp)$ (in the notation of the first formulation of Theorem \ref{thm: Stamm's theorem}). In fact, these parameters can be chosen compatibly with the parameters on $\Ybar$; more precisely, we have the following. Let $\Qbar$ be a closed point in $W_{\BB_\gert,\BB_{\gert^\ast}}$; let $\Qbar_\gerp$ be the image in $\Ybar(\gerp)$  and $\Pbar=\pi(\Qbar)=\pi(\gerp)(\Qbar_\gerp)$. We have an isomorphism:
\[
\widehat{\calO}_{\Ybar,\Qbar} \cong  {\otimes}_{\widehat{\calO}_{\Xbar,\Pbar}} \widehat{\calO}_{\Ybar(\gerp),\Qbar_\gerp},
\]
where the tensor product is over all prime ideals $\gerp$ dividing $p$. In terms of the above parameters, this isomorphism can be written as:
\begin{eqnarray*}
k[\![x_\beta: \beta\in\BB_\gert; y_\beta: \beta\in\BB_{\gert^\ast} ]\!] \cong {\otimes}_{\gerp\vert\gert} k[\![x_{\beta,\gerp}: \beta\in\BB_\gerp; y_{\beta,\gerp}: \beta\in\BB_{\gerp^\ast} ]\!] {\otimes} {\otimes}_{\gerp\vert\gert^\ast} k[\![y_{\beta,\gerp}: \beta\in\BB]\!],
\end{eqnarray*}
where the tensor products are over
$\widehat{\calO}_{\Xbar,\Pbar}\cong k[\![t_\beta; \beta\in\BB]\!]$,
and  the images of $x_\beta,y_\beta$ are given, respectively, by
$x_{\beta,\gerp},y_{\beta,\gerp}$, where  $\gerp$ is  the unique prime
ideal such that $\beta\in\BB_\gerp$. This allows us to partially calculate the morphism
\[
\pi(\gerp)^\ast:{\widehat{\calO}_{\Xbar,\Pbar}} \arr \widehat{\calO}_{\Ybar(\gerp),\Qbar_\gerp},
\]
as follows. We first consider the above  isomorphism when  $\gert=\ol$. Using the description of
$\pi^*$ on $\Ybar_V^{\rm ord}=W_{\emptyset,\BB}$ given above, we find that if $\Qbar_\gerp \in W_{\emptyset,\BB_\gerp}(\gerp)$, then for $\beta \in \BB_\gerp$, we have
 $\pi^\ast(\gerp)(t_\beta)=\pi^\ast(t_\beta)=y_\beta^p=y_{\beta,\gerp}^p$, where the first and last equalities are via the isomorphism above.
On the other hand, considering the above isomorphism  when  $\gert=(p)$, and  using the description of $\pi^*$ on $\Ybar_F^{\rm ord}=W_{\BB,\emptyset}$, we can similarly argue that for  $\Qbar_\gerp \in W_{\BB_\gerp,\emptyset}(\gerp)$, and $\beta\in \BB_\gerp$,  we have $\pi^\ast(\gerp)(t_\beta)=\pi^\ast(t_\beta)=x_\beta=x_{\beta,\gerp}$.
Putting these together, we can calculate the map
\[
\pi^*:{\widehat{\calO}_{\Xbar,\Pbar}} \arr
\widehat{\calO}_{\Ybar,\Qbar}
\]
at a point $\Qbar \in W_{\BB_\gert,\BB_{\gert^\ast}}$, as follows. Let $\beta\in\BB_\gerp$. Then, we can write
\[
\pi^\ast(t_\beta)=\pi(\gerp)^\ast(t_\beta)=\begin{cases} x_{\beta,\gerp}=x_\beta &  \Qbar_\gerp\in W_{\BB_\gerp,\emptyset}(\gerp)   \Leftrightarrow  \gerp\vert \gert, \\  y_{\beta,\gerp}^p=y_\beta^p  & \Qbar_\gerp\in W_{\emptyset, \BB_\gerp}(\gerp) \Leftrightarrow \gerp\vert\gert^\ast,\end{cases}
\]
where all the equalities (but the second one) are via the isomorphism above. Renaming the $x_\beta, y_\beta$ to $z_\beta$ as in the second formulation of Theorem \ref{thm: Stamm's theorem}, the claim follows.

 \end{proof}

Define $\Yrig(\gert)$ to be the Raynaud generic fibre of the completion of $\YY(\gert)$ along its special fibre.
For $Q\in\Yrig(\gert)$, and $\beta\in\BB_\gert$, we define
\[\nu_\beta(Q)=
\begin{cases}1 &  \beta\in\eta(\Qbar)-I(\Qbar), \\  \nu(x_{\beta,\gert}(Q)) & \beta\in I(\Qbar),\\ 0 & \beta\not\in \eta(\Qbar),
\end{cases}
\]
where $x_{\beta,\gert}$'s are variables at $\Qbar$, a closed point of $\Ybar(\gert)$, chosen in the same way as in Theorem \ref{thm: Stamm's theorem}. It follows that if $Q\in\Yrig$, and $Q_\gert$ denotes its image in $\Yrig(\gert)$, then $\nu_\beta(Q)=\nu_\beta(Q_\gert)$ if $\beta\in\BB_\gert$. We can generalize the definition of $\calU,\calV,\calW,\calV_\gerp,\calW_\gerp$  in \S \ref{subsec: main theorem} in the obvious way to obtain admissible opens in $\Yrig(\gert)$ denoted $\calV(\gert)$, $\calW(\gert)$, $\calV_\gerp(\gert)$, $\calW_\gerp(\gert)$ (for any ideal $\gerp\vert \gert$). For example,
\[
\calU(\gert)=\{P\in\Xrig: \nu_\beta(P)+p\nu_{\sigma^{-1}\circ\beta}(P)<p \quad \forall \beta\in\BB_\gert \}.
\]Similarly, we can define
\begin{eqnarray*}
\Xord(\gert)=\{P\in\Xrig(\gert): \nu_\beta(P)=0 \quad \forall\beta\in\BB_\gert\},\\
\Yoord(\gert)=\{Q\in\Yrig(\gert): \nu_\beta(Q)=0 \quad \forall\beta\in\BB_\gert\}.
\end{eqnarray*}
 As before, we can apply Hensel's lemma to obtain a morphism $s^\circ(\gert)\colon\Xord(\gert) \arr \Yoord(\gert)$ which is a section to $\pi(\gert)$. Applying the same method as in the proof of Theorem \ref{theorem: canonical subgroup}, we can prove the following.

\begin{thm}\label{theorem: canonical subgroup over gert} Let notation be as above.
\begin{enumerate}
\item $\pi(\gert)(\calV(\gert))=\calU(\gert)$.
\item There is a section $s^\dagger(\gert)\colon\calU(\gert) \arr \calV(\gert)$ to $\pi(\gert)$, extending
$s^\circ(\gert)\colon \Xord \arr \Yoord(\gert)$.
 \end{enumerate}
\end{thm}

 Let $K\supseteq W(\kappa)$ be a completely valued field. Let
$\underline{A}/K$ be an abelian variety corresponding to a point of non-cuspidal reduction
$P\in \mathcal{U}(\gert)$.  Let $Q=s^\dagger(\gert)(P)$ correspond to $(\uA,H_\gert)$. We call $H_\gert$ the {\it $\gert$-canonical subgroup} of $\uA$. Again, we can prove:

\begin{thm}\label{theorem: reduction of canonical subgroup appendix}
Let
$\underline{A}/K$ be an abelian variety corresponding to a point of non-cuspidal reduction
$P\in \mathcal{U}(\gert)$. Enlarging $K$, we can assume that there exists $r_\gert$ in $K$ such that
${\rm val}(r_\gert)=\max\{\nu_\beta(\underline{A}): \beta\in\mathbb{B}_\gert\}$. The $\gert$-canonical subgroup of $\underline{A}$ reduces to ${\rm Ker}({\rm Fr})[\gert]$ modulo $p/r_\gert$.
 \end{thm}

Using the above-explained relationship between the various moduli-spaces we have defined, it is easy to see that if $\uA$ has a $\gert$-canonical subgroup $H_\gert$, and if $\gerz\vert \gert$, then it also has a $\gerz$-canonical subgroup $H_\gerz$, which satisfies $H_\gerz=H_\gert[\gerz]$. In particular, if $\uA$ has a canonical subgroup $H$, then
\[
H=\oplus_{\gerp\vert p} H[\gerp],
\]
and for each $\gerp\vert p$, the subgroup $H[\gerp]$ is the $\gerp$-canonical subgroup of $\uA$.

Finally, we mention that an analogue of Theorem \ref{theorem: hecke} holds for
a general $\Yrig(\gert)$, the formulation of which we leave to the reader. Applied to $\gert=\gerp$, this  completely determines the $p$-adic geometry of the partial $U$-operator, $U_\gerp$, viewed as a correspondence on $\calV(\gerp) \cup \calW(\gerp)$, the {\it not-too-singular locus} of $\Yrig(\gerp)$.

\

\

\

\

\end{document}